\documentclass[english,reqno]{amsart}
\usepackage{hyperref} 
\hypersetup{hidelinks,backref=true,pagebackref=true,hyperindex=true,colorlinks=true,citecolor=red,linkcolor=blue,breaklinks=true,urlcolor=ocre,bookmarks=true,bookmarksopen=false,pdftitle={Title},pdfauthor={Author}}

\usepackage{dsfont}
\usepackage{enumerate}
\usepackage{mathrsfs,esint}
\usepackage{stmaryrd}
\usepackage{amsopn,bm}
\usepackage{amsmath}
\usepackage{amssymb}
\usepackage{amsfonts}
\usepackage{amsbsy}
\usepackage{amscd,indentfirst,epsfig}
\usepackage{amsfonts,amsmath,latexsym,amssymb,verbatim,amsbsy}
\usepackage{amsthm} 
\numberwithin{equation}{section}

\AtBeginDocument{\let\oldtheo\theo
\renewcommand{\theo}{%
\oldtheo\hypertarget{\***@currentHref\endcsname}{}}}
 \def \leq {\leqslant}
\def \geq {\geqslant}
\def\ind#1{\lower5pt\hbox{$\scriptstyle #1$}}
 
\def \ge {\geqslant}
\def \d {\mathrm{d}}
\def\al{\alpha}
\def\e{\varepsilon}

\def \m {\bm{\varpi}}
\def \R{\mathbb R}
\def\S{\mathbb S}

\def\N{\mathbb N}

\def\C{\mathcal C}

\def\A{\mathcal A}
\def\D{\mathscr D}
\def\P{\mathcal P}

\def\T{\mathcal T}
\def\M{\mathcal M}

\def\E{\mathcal{E}}
\def \B{\mathcal{B}}
\def\s{\sigma}
\def\a{\alpha}

\def \ep {\bm{\varepsilon}}

\def \vb {v_\ast}

\def\k{\kappa}

\def\d{\mathrm{d}}

\def\W {\mathbb{W}}

\def \cS {\mathcal{S}}
\def\Q{\mathcal{Q}}

\newtheorem{theo}{Theorem}[section]
\newtheorem{prop}[theo]{Proposition}
\newtheorem{cor}[theo]{Corollary}
\newtheorem{lem}[theo]{Lemma}

\newtheorem{defi}[theo]{Definition}
\newtheorem{nb}[theo]{Remark}
\def \leq {\leqslant}
\def \geq {\geqslant}

\def \Ss {\mathcal{S}}
\newcommand{\verti}[1]{{\left\vert\kern-0.25ex\left\vert\kern-0.25ex\left\vert #1 
    \right\vert\kern-0.25ex\right\vert\kern-0.25ex\right\vert}}

\def \LL {\mathscr{L}_{\alpha}}

\def \ds {\displaystyle}
\def \X {\mathbb{X}}

\def \IR {\int_{\R^d}}

\numberwithin{equation}{section}

\begin{document}

\title{Convergence to self-similarity for ballistic annihilation dynamics}

\author{Ricardo J. {\sc Alonso}}

\address{Departamento de Matem\'{a}tica, PUC-Rio, Rua Marqu\^{e}s de S\~ao Vicente 225, Rio de Janeiro, CEP 22451-900, Brazil.} \email{ralonso@mat.puc-rio.br}

 \author{V\'{e}ronique {\sc Bagland}}

 \address{Universit\'{e} Clermont Auvergne, LMBP, UMR 6620 - CNRS,  Campus des C\'ezeaux, 3, place Vasarely, TSA 60026, CS 60026, F-63178 Aubi\`ere Cedex,
 France.}\email{Veronique.Bagland@math.univ-bpclermont.fr}

 \author{Bertrand {\sc Lods}}

 \address{Universit\`{a} degli
Studi di Torino \& Collegio Carlo Alberto, Department of Economics and Statistics, Corso Unione Sovietica, 218/bis, 10134 Torino, Italy.}\email{bertrand.lods@unito.it}

\maketitle
 \begin{abstract} We consider the spatially homogeneous Boltzmann equation for ballistic annihilation in dimension $d\geq2$. Such model describes a system of ballistic hard spheres that, at the moment of interaction, either annihilate with probability $\alpha \in (0,1)$ or collide elastically with probability $1-\alpha$.  Such equation is highly dissipative in the sense that all observables, hence solutions, vanish as time progresses.  Following a contribution, by two of the authors, considering well posedness of the steady self-similar profile in the regime of small annihilation rate $\alpha\ll1$, we prove here that such self-similar profile is the intermediate asymptotic attractor to the annihilation dynamics with explicit universal algebraic rate.  This settles the issue about universality of the annihilation rate for this model brought in the applied literature.  

\smallskip
\noindent \textbf{Keywords.} Ballistic annihilation, reacting particles, self-similarity, long-time asymptotic, annihilation rate.

\end{abstract}

\tableofcontents

\section{Introduction}

\subsection{Physical motivation and setting of the problem}\label{sec:phys}

In recent years, the physics community proposed several kinetic models in order to test the relevance of non-equilibrium statistical mechanics in systems of \textit{reacting particles}.  Such systems have important applications in different branches of physics and engineering such as surface growth (semiconductors) \cite{spohn} and coarsening processes (dynamics of traffic).  A common feature of these models is that the dissipative nature of the interactions results in the loss of collision invariants and leads to tremendous difficulties for the derivation of suitable hydrodynamic models.

A paradigmatic example of such dissipative models is the one of \emph{granular gas dynamics} which corresponds to a system of $d$-dimensional hard-spheres undergoing inelastic collisions. For such a model, the number of particles and the momentum are conserved, but the kinetic energy is dissipated at each collision. At the kinetic level, the long-time behavior of granular gases is relatively well-understood, at least, in a spatially homogeneous setting: in absence of external forcing, the kinetic energy is continuously decreasing and the solution converges to a singular state described by a Dirac mass, that is, to a complete rest.  Two main questions then arise:

\begin{itemize} 
\item[--]  First, what it the rate of the convergence to zero of the kinetic energy, i.e. \textit{\textbf{how fast a granular gas is cooling down ?}} The precise rate of decay of the kinetic energy is known as \emph{Haff's law} and it has been rigorously proven in \cite{mmjsp} for inelastic hard-spheres with constant inelasticity and, more generally, in \cite{AloLo1} for the case of viscoelastic particles.
\item[--] Second, can we make a more precise description of the long-time behavior of the gas as it goes towards the singular limit ?  More precisely, due to the diffusive nature of collisions, one expects some type of intermediate self-similarity, \emph{i.e. a non Gaussian homogeneous cooling state}. The existence and uniqueness of such self-similar state has been rigorously obtained in \cite{mmjsp,MiMo3}, where it has been proven that in the quasi-elastic regime it is the attractor of any properly rescaled solution, see \cite{MiMo3}.  The case of viscoelastic particles is intrinsically different to that of constant restitution and always produces Gaussian intermediate asymptotic states, see \cite{AloLo1}. 
\end{itemize}
\smallskip

The present contribution aims to answer similar questions for another example of dissipative systems, known as \textbf{\textit{probabilistic ballistic annihilation}}. Such model has been introduced in the 90's by \cite{Ben-Naim,coppex04, coppex05, Kaprivsky, Piasecki,Trizac} and describes a system of $d$-dimensional elastic hard spheres that interact in the following way: particles move freely (ballistically) between collisions and, whenever two particles meet they either annihilate with probability $\alpha \in [0,1]$ (both interacting particles vanish), or they undergo an elastic collision with probability $1 - \alpha$.  Interestingly, as the annihilation probability $\alpha$ ranges from zero to one, the probabilistic ballistic annihilation model will move from describing the dynamic of elastic hard spheres to describing the dynamic of pure annihilation, which are substantially different. Ballistic annihilation is considered to be a very accurate model in the whole range $\alpha\in[0,1]$ (including the pure annihilation case $\alpha=1$) in dimension other than one.  This conclusion has been reached through extensive numerical simulations in the aforementioned references.  In dimension one, the kinetic approach has been shown to mistakenly predict the correct dynamic relaxation for the pure annihilation regime in the case of finite number of point masses (discrete velocities) for initial data due to strong cumulative correlations. We will therefore in the sequel always consider the case of $d$-dimensional hard spheres with $d \geq 2.$ 

\medskip

Contrary to granular gases, ballistic annihilation dissipates the density, thus, it does not have natural collision invariants.  As a consequence, the solution to the associated kinetic equation converges to $0$ as time goes to infinity.  We aim to answer the two questions raised before:
\begin{enumerate}[(Q1)]
\item What is the precise rate of decay towards zero of the macroscopic quantities as density and kinetic energy ? 
\item  Is the long-time behavior of the solution described by some suitable self-similar profile which would attract any solution to the associated equation after proper rescaling ? 
\end{enumerate}

We will focus on these questions in the regime when $\alpha$ is relatively small, but still order one.  This regime is interpreted as a system of elastic particles colliding many times before annihilating, that is, particles undergoing significant diffusion due to collisions before annihilating.  This is precisely the natural regime to search for self-similarity.  We prove that the model possesses an universal attractor related to the self-similarity equation, which leads to universal algebraic relaxation rates that can be explicitly computed.  Exact rates are quite expensive to compute as they demand the knowledge of the attractor, which requires solving the highly nonlinear integro-differential equation \eqref{steady}.  For this reason, the rate found in the limit $\alpha\rightarrow0$ is of key relevance.  In reference \cite{Trizac} was conjectured that, in particular, the mass of the solution $f(t,\cdot)$ to the kinetic equation  behaved in the long run as 
$$\int_{\R^{d}}f(t,v)\d v\sim t^{-4d/(4d+1)}, \qquad t \to \infty$$ and later in \cite{Cop04} numerical evidence was given supporting this fact.  A particular application of the analysis performed in this work is precisely the rigorous proof of such statement (see Theorem \ref{main:no-scaled} and comments below).  Interestingly, the pure annihilation case $\alpha=1$ does not enjoy attractors, and long time relaxation rates depend on the initial configuration as proven in the aforementioned references (for both continuous and discrete velocity initial data).  Furthermore, it is unclear what happens with the system's dynamics in the regime where $\alpha$ is relatively large, however, reference \cite{Cop04} shows numerical evidence that seems to indicate existence of attractors as long as $\alpha<1$.
\medskip

Before discussing in details the results and answers to the above queries, let us precisely describe the model we are dealing with.

\subsection{The equation at stake}

In a spatially homogeneous framework, the density of particles  $f(t,v)$   with velocity $v \in \R^d$ $(d \geq 2)$ at time $t \geq 0,$  satisfies the following
\begin{equation}\label{BE}\begin{cases}
\partial_t f(t,v)&=\mathbb{B}_\alpha(f,f)(t,v):=(1-\alpha)\Q(f,f)(t,v) -\alpha \Q_-(f,f)(t,v) \qquad t > 0\\
f(0,v)&=f_0(v) \end{cases}\end{equation}
where $\Q$ is the quadratic Boltzmann collision operator defined by
 \begin{equation*}\label{bilin}
 \Q(g,f)(v) = \int _{\R^d \times \S^{d-1}} \left|v-v_*\right|
         \left( g' f_* ' - g f_* \right) \,\d v_* \, \d\sigma,
 \end{equation*}
where we have used the shorthands $g=g(v)$, $g'=g(v')$, $f_*=f(v_*)$ and
$f'_*=f(v'_*)$ with post-collisional velocities $v'$ and $v'_*$  parametrized by
 \begin{equation}\label{eq:rel:vit}
 v'  =  \frac{v+v_*}{2} + \frac{|v-v_*|}{2}\;\s,   \qquad
v'_* =  \frac{v+v_*}{2} - \frac{|v-v_*|}{2}\;\s,   \qquad \s  \in  {\S}^{d-1}.
 \end{equation}
 Here above, $\d\sigma$ denotes the \emph{normalized} Lebesgue measure over $\S^{d-1}$, i.e 
 $\int_{\S^{d-1}}\d\sigma=1.$\medskip

The above collision operator $\Q(f,f)$ splits as $\Q(f,f)=\Q_+(f,f)-\Q_-(f,f)$ where the gain part $\Q_+$ is given by
$$ \Q_+(f,f)(v) = \int _{\R^d \times {\S}^{d-1}}|v-v_*| f'_* f'\, \d v_* \,  \d\sigma ,$$
while the loss part $\Q_-$ is defined as
\begin{equation}\label{Q-}
\Q_-(f,f)(v)=f(v)\Sigma_{f}(v), \qquad \text{ with } \qquad \Sigma_{f}(v)=\int_{\R^d}|v-v_*|f_*\d v_*.\end{equation}
 
The Cauchy theory for the above equation has been investigated in a previous contribution \cite{jde}, and we refer to the \textit{op. cit.} for related questions.\\

In all the present paper, we shall assume that $f_0 \in L^1_{3}(\R^d)$ is a nonnegative initial datum and that $f(t,\cdot) \in L^1_3(\R^d)$ is the associated solution to \eqref{BE} for a given parameter $\alpha \in (0,1)$. As explained, such solution $f(t,\cdot)$ is expected to converge to zero  as $t \to \infty$  and, before reaching such degenerate state, the solution is expected to become close to a self-similar solution of the form
\begin{equation}\label{autosim}
f_H(t,v)=\lambda(t)\,\psi_\alpha(\beta(t)v),
\end{equation}
for some scaling functions $\lambda(t)$ and $\beta(t)$ and for a given \emph{self-similar profile} $\psi_\alpha$ (depending clearly on the choice of the parameter $\alpha$). One can then show, see \cite{Piasecki,Trizac,jde}, that such a self-similar profile is a solution to the following stationary Boltzmann equation
\begin{equation}\label{steady}
\mathbf{A}_\alpha \psi_\alpha(\xi)+\mathbf{B}_\alpha\, \xi \cdot \nabla_\xi \psi_\alpha(\xi)
=(1-\alpha)\Q(\psi_\alpha,\psi_\alpha)(\xi) -\alpha \Q_-(\psi_\alpha,\psi_\alpha)(\xi),
\end{equation}
where
\begin{equation}\label{Aalpha}
\mathbf{A}_\alpha=-\frac{\alpha}{2} \int_{\R^d}
\left(\frac{d+2}{\int_{\R^d} \psi_\alpha(\xi_*)\, \d\xi_*}
-\frac{d\,|\xi|^2}{\int_{\R^d} \psi_\alpha(\xi_*)\, |\xi_*|^2\, \d\xi_*}\right)
\Q_-(\psi_\alpha ,\psi_\alpha )(\xi)\d \xi,
\end{equation}
and
\begin{equation}\label{Balpha}
\mathbf{B}_\alpha =-\frac{\alpha}{2}\int_{\R^d}
\left(\frac{1}{\int_{\R^d} \psi_\alpha(\xi_*)\, \d\xi_*}
-\frac{|\xi|^2}{\int_{\R^d} \psi_\alpha(\xi_*)\, |\xi_*|^2\, \d\xi_*}\right)
\Q_-(\psi_\alpha ,\psi_\alpha )(\xi)\d \xi.
\end{equation}
Existence of solutions to \eqref{steady} has been proven in \cite{jde} for any $\alpha$ smaller than some explicit threshold value. Moreover, borrowing techniques already used for similar questions in the study of granular gases \cite{MiMo3}, uniqueness of the self-similar profile has been established in \cite{jde2} for a smaller range of parameters $\alpha$. Namely, 
\begin{theo}{\textit{\textbf{(Existence and uniqueness of the self-similar profile}} \cite{jde,jde2})}\phantomsection\label{theo:exist-unique}
There is some explicit $0 < \alpha_{0} < 1$
such that for any $\alpha\in(0,\alpha_{0})$, for any given $\varrho>0$ and  
$E>0$, there exists a unique  solution $\psi_\alpha$ to \eqref{steady} with mass $\varrho$, energy $E$ and zero momentum, i.e.
\begin{equation*}\int_{\R^d} \psi_\alpha(\xi)\left(\begin{array}{c}
                                                           1 \\
                                                           \xi \\
                                                           |\xi|^2
                                                         \end{array}\right)
\,\d\xi=\left(\begin{array}{c}
          \varrho \\
          0 \\
          E \end{array}\right)
.
\end{equation*}
Moreover, $\psi_{\alpha}$ is smooth and radially symmetric.
\end{theo}

By a simple scaling argument, there is no loss of generality in considering the special case in which $\varrho=1$ and $E=d/2$ and, from now on, we will denote by $\psi_\alpha$ the \emph{unique 
solution to \eqref{steady}} that satisfies 
\begin{equation}\label{init}\int_{\R^d} \psi_\alpha(\xi)\left(\begin{array}{c}
                                                           1 \\
                                                           \xi \\
                                                           |\xi|^2
                                                         \end{array}\right)
\,\d\xi=\left(\begin{array}{c}
          1 \\
          0 \\
          \dfrac{d}{2}
        \end{array}\right)
.
\end{equation}
We denote by $\M(\xi)$ the Maxwellian distribution with same first moments as $\psi_{\alpha}$, i.e.
\begin{equation}\label{M}
\M(\xi)=\pi^{-d/2}\, \exp(-|\xi|^2), \qquad \forall \xi \in\R^{d}.\end{equation}

\subsection{{Self similar variable}}

Let us consider a solution $f=f(t,v)$ to \eqref{BE} for some nonnegative initial datum $f_{0} \in L^{1}_{3}(\R^{d}).$ Let us introduce the following $\psi(\tau,\xi)$ through
\begin{equation}\label{scalingPsi}
f(t,v)=n_{f}(t){(2T_{f}(t))^{-d/2}}\psi\left(\tau(t),\frac{v-\bm{u}_{f}(t)}{\sqrt{2T_{f}(t)}}\right)
\end{equation} 
for some suitable scaling function $\tau\::\:\R^{+}\to\R^{+}$ and with
\begin{align}\label{eq:momf}
\begin{split}
n_{f}(t)=\int_{\R^{d}}f(t,v)\d v, &\qquad n_{f}(t)\bm{u}_{f}(t)=\int_{\R^{d}}f(t,v)v\d v, \\
&d\,n_{f}(t)T_{f}(t)=\int_{\R^{d}}f(t,v)|v-\bm{u}_{f}(t)|^{2}\d v, \qquad t \geq0.\end{split}\end{align}
Notice that the choice of the scaling enforces the following
\begin{equation}\label{eq:conserve}
\int_{\R^{d}}\psi(\tau,\xi)\left(\begin{array}{c}1\\\xi \\ {|\xi|^{2}}\end{array}\right)\d\xi=\left(\begin{array}{c}1 \\0 \\\frac{d}{2}\end{array}\right) \qquad \forall \tau \geq0.\end{equation}
which ensures the self-similar function $\psi(\tau)$ to share the same mass, momentum and energy of the steady profile $\psi_{\alpha}.$ With such a scaling, straightforward computations, see Section \ref{sec:scaling} for details, combined with the uniqueness of the solutions to both Cauchy problems \eqref{BE} and \eqref{rescaBE} yield to the following proposition.
\begin{prop}\phantomsection\label{prop:cauc} Let $f_{0} \in L^{1}_{3}(\R^{d})$ be a nonnegative initial datum with positive mass and temperature
$$n_{f_{0}}=\int_{\R^{d}}f_{0}(v)\d v >0, \quad \qquad T_{f_{0}}=\frac{1}{dn_{f_{0}}}\int_{\R^{d}}f_{0}(v)|v-\bm{u}_{f_{0}}|^{2}\d v> 0,$$
where $\bm{u}_{f_{0}}=\frac{1}{n_{f_{0}}}\int_{\R^{d}}f_{0}(v)\,v\d v \in \R^{d}.$ Let $f(t,v)$ denote the unique solution to \eqref{BE} associated to the initial datum $f_{0}.$ Then, $\min\left(n_{f}(t),T_{f}(t)\right) >0$ for all $t \geq 0$.  In addition, introducing the scaling function 
\begin{equation}\label{eq:tau}
\tau(t)=\sqrt{2}\int_{0}^{t} n_{f}(s)\sqrt{T_{f}(s)}\d s, \qquad \forall t \geq 0,\end{equation}
and defining $\psi(\tau,\xi)$ by \eqref{scalingPsi}, it holds that $\psi(\tau,\xi)$ 
is the unique solution to 
\begin{align}\label{rescaBE}
\begin{split}
\partial_{\tau}\psi(\tau,\xi) + \big(\mathbf{A}_{\psi}(\tau) - &d \mathbf{B}_{\psi}(\tau)\big)\,\psi(\tau,\xi) + \mathbf{B}_{\psi}(\tau)\mathrm{div}_{\xi}\big(\left({\xi}-\bm{v}_{\psi}(\tau)\right)\psi(\tau,\xi))\\
&= (1-\alpha)\Q(\psi,\psi)(\tau,\xi)-\alpha\Q_{-}(\psi,\psi)(\tau,\xi)\end{split}\end{align}
with initial datum $\psi(0,\xi)=(2T_{f_{0}})^{d/2}n_{f_{0}}^{-1}\;f_{0}\left(\sqrt{2T_{f_{0}}}\,\xi+\bm{u}_{f_{0}}\right)$ and 
where $\mathbf{A}_{\psi}(\cdot),\mathbf{B}_{\psi}(\cdot)$ and $\bm{v}_{\psi}(\cdot)$ are defined by 
\begin{equation}\label{eq:BAV}\begin{cases}
\mathbf{A}_\psi(\tau)&=\displaystyle -\frac{\alpha}{2}  \int_{\mathbb{R}^d}
\left( d+2 
-2|\xi|^2\right) \Q_-(\psi ,\psi )(t,\xi)\text{d} \xi\\
\\
\mathbf{B}_\psi(\tau) &=-\displaystyle \frac{\alpha}{2}\int_{\mathbb{R}^d}
\left(1-\frac{2}{d}|\xi|^2\right)\Q_-(\psi ,\psi )(\tau,\xi)\text{d} \xi\\
\\
\mathbf{B}_{\psi}(\tau)\bm{v}_{\psi}(\tau)&=-\displaystyle  \alpha\int_{\R^{d}}\xi\,\Q_{-}(\psi,\psi)(\tau,\xi)\d\xi \in \R^{d}, \qquad \forall \tau \geq 0.
\end{cases}\end{equation}
\end{prop}\phantomsection

From the previous results, one sees that $\psi_{\alpha}$ is a steady solution to \eqref{rescaBE} -- independent of the time variable $\tau$ -- and for which
$$\mathbf{A}_{\alpha}:=\mathbf{A}_{\psi_{\alpha}}, \qquad \mathbf{B}_{\alpha}:=\mathbf{B}_{\psi_{\alpha}}, \qquad \text{ and } \quad \bm{v}_{\psi_{\alpha}}=0$$
since, $\psi_{\alpha}$ being radially symmetric so is $\Q_{-}(\psi_{\alpha},\psi_{\alpha}).$

\subsection{Notations} For all $r >0$, we denote by $\mathbb{D}(r)$ the open disc of $\mathbb{C}$ with radius $r$, i.e. $\mathbb{D}(r)=\{z \in \mathbb{C}\;;\;|z| < r\}.$ Given two Banach spaces $X$ and $Y$, we denote by $\mathscr{B}(X,Y)$
the set of linear bounded operators from $X$ to $Y$ and by
$\|\cdot\|_{\mathscr{B}(X,Y)}$ the associated operator norm. If $X=Y$,
we simply denote $\mathscr{B}(X):=\mathscr{B}(X,X)$. We denote then by
$\mathscr{C}(X)$ the set of closed, densely defined linear operators
on $X$ and by $\mathscr{K}(X)$ the set of all compact operators in
$X$. For $A \in \mathscr{C}(X)$, we write $\D(A) \subset X$ for the
domain of $A$, $\mathscr{N}(A)$ for the null space of $A$ and
$\mathrm{Range}(A) \subset X$ for the range of $A$. The spectrum of
$A$ is then denoted by $\mathfrak{S}(A)$ and the resolvent set is
$\varrho(A)$. For $\lambda \in \rho(A)$, 
$\mathcal{R}(\lambda,A)=(\lambda\mathbf{Id} -A)^{-1}$ denotes the resolvent of $A$.

Let us introduce some useful notations for function spaces. For any nonnegative weight function $m\::\:\R^{d}\to \R^{+}$,  we define, for all $p \geq1$ and $q \geq0$ the space $L^{p}_{q}(m)$ through the norm
$$\|f\|_{L^{p}_{q}(m)}:=\left(\int_{\R^{d}}|f(\xi)|^{p}\langle \xi\rangle^{pq}m(\xi)\d\xi\right)^{1/p},$$
i.e. $L^{p}_{q}(m)=\{f\::\R^{d} \to \R\;;\,\|f\|_{L^{p}_{q}(m)} < \infty\}.$ We also define,  for $k \in \N$,
$$\W^{k,p}_{q}(m)=\left\{f \in L^{p}_{q}(m)\;;\;\partial_{\xi}^{\beta}f \in L^{p}_{q}(m) \:\forall |\beta| \leq k\right\}$$
with the usual norm,
$$\|f\|_{\W^{k,p}_{q}(m)}^{p}=\sum_{|\beta| \leq k}\|\partial_{\xi}^{\beta}f\|_{L^{p}_{q}(m)}^{p}.$$
For $m \equiv 1$, we simply denote the associated spaces by $L^{p}_{q}$ and $\W^{k,p}_{q}$.

\subsection{Main results} Let us recall that $\mathbf{A}_{\alpha}:=\mathbf{A}_{\psi_{\alpha}},$ $\mathbf{B}_{\alpha}:=\mathbf{B}_{\psi_{\alpha}}$. We introduce the nonnegative quantities 
\begin{equation}\label{abalpha}
\mathbf{a}_{\alpha}:=\frac{d \mathbf{B}_{\alpha}- \mathbf{A}_{\alpha}}{\alpha}, \qquad \mathbf{b}_{\alpha}:=\frac{(d+2) \mathbf{B}_{\alpha}- \mathbf{A}_{\alpha}}{\alpha}, \qquad \alpha \in (0,\alpha_{0}).
\end{equation}
The following is the main result of the paper.
\begin{theo}\phantomsection\label{main:no-scaled} 
Assume that $f_0$ is nonnegative, with positive mass and temperature and such that {
$$f_{0} \in  H^{\frac{(5-d)^{+}}{2}}_{\eta}(\R^{d}) \cap L^{1}_\kappa(\R^d) $$}
for some $\eta > 4+d/2$ and some $\kappa>\max\{ 4+d/2, d(d-2)/(d-1)\}$, $(d \geq 3)$. We also assume that $f_{0}$ has finite entropy and Fisher information, i.e.
$$\int_{\R^{d}}f_{0}(v)\log f_{0}(v)\d v < \infty \qquad \text{ and } \qquad \int_{\R^{d}}\big|\nabla \sqrt{f_{0}(v)}\big|^{2}\d v < \infty.$$
Let $f(t,v)$ be the unique solution to \eqref{BE} associated to the initial datum $f_{0}$. Then, there exists some $\overline{t}>0$ and some explicit $A>0$ such that for any $a\in(0,A/2)$, for any $\varepsilon >0$  there exist some explicit  $\alpha_{c} \in (0,\alpha_{0})$ such that, for all $\alpha \in (0,\alpha_{c})$, there is some $C=C_{\alpha,\varepsilon,f_{0}} >0$ depending on $f_{0}$ through $n_{f_{0}},$ $\bm{u}_{f_{0}}$ and $T_{f_{0}}$ and such that
$$\int_{\R^{d}}\left|f(t,v)- \bm{f}_{\alpha}(t,v)\right|\exp\left(a\,\frac{|v-\bm{u}_{f}(t)|}{\sqrt{2T_{f}(t)}}\right)
\d v \leq  C\left(1+t\right)^{-\vartheta} \qquad \forall t \geq \overline{t}$$
where $\vartheta:=\frac{2}{\alpha(\mathbf{a}_{\alpha}+\mathbf{b}_{\alpha})}\left(\alpha\mathbf{a}_{\alpha}+ \mu_{\star}-\varepsilon\right)$, $\mu_{\star}$ denotes the spectral gap of the linearized operator $\mathscr{L}_{0}$ associated to the elastic Boltzmann equation in $L^{2}(\M^{-1})$, 
$$\bm{f}_{\alpha}(t,v)=n_{f}(t)(2T_{f}(t))^{-d/2}\psi_{\alpha}\left(\frac{v-\bm{u}_{f}(t)}{\sqrt{2T_{f}(t)}}\right), \qquad t \geq 0$$
with the moments $n_{f}(t),$ $T_{f}(t)$ and $\bm{u}_{f}(t)$ satisfying
\begin{equation}
\label{eq:rate}
\log n_{f}(t)  \simeq -2\frac{ \mathbf{a}_{\alpha}}{\mathbf{a}_{\alpha}+\mathbf{b}_{\alpha}}\log t, \qquad  \qquad \log T_{f}(t) \simeq -\frac{4\mathbf{B}_{\alpha}}{\alpha(\mathbf{a}_{\alpha}+\mathbf{b}_{\alpha})}\log t\qquad \text{ for } t \to \infty\end{equation}
and 
{$$\lim_{t \to \infty}\bm{u}_{f}(t)=\bm{u}_{f_{0}}+ \sqrt{2T_{f_0}}\int_{0}^{\infty}\mathbf{B}_{\psi}(s)\bm{v}_{\psi}(s)\exp\left(-\int_{0}^{s}\mathbf{B}_{\psi}(r)\d r\right)\d s.$$}
Moreover, for fixed $\alpha$, the aforementioned rates for $n_{f}$ and $T_f$ are universal.\end{theo}

The above Theorem provides a satisfying answer to the queries (Q1) and (Q2) above:
\begin{itemize}
\item[--]   The precise rate of convergence of the density and temperature is described by \eqref{eq:rate} for any $\alpha \in (0,\alpha_{c})$. Notice that this rate is sharp in the regime of small annihilation since one has (see Remark \ref{imporem}) for $\alpha \to 0$
$$\frac{2 \mathbf{a}_{\alpha}}{\mathbf{a}_{\alpha}+\mathbf{b}_{\alpha}} \simeq \frac{4d}{4d+1} , \qquad \frac{4\mathbf{B}_{\alpha}}{\alpha(\mathbf{a}_{\alpha}+\mathbf{b}_{\alpha})}=\frac{2(\mathbf{b}_{\alpha}-\mathbf{a}_{\alpha})}{\mathbf{a}_{\alpha}+\mathbf{b}_{\alpha}} \simeq \frac{2}{4d+1},$$
which results, for small values of $\alpha$, in   
$$n_{f}(t) \simeq t^{-\frac{4d}{4d+1}}, \qquad T_{f}(t) \simeq t^{-\frac{2}{4d+1}} \qquad \text{ as } t \to \infty\,.$$
These results match the rate of convergence conjectured  by physicists in \cite{Trizac,Cop04} described in Section \ref{sec:phys}. 
\item[--]   For an initial datum with little regularity requirement, any solution to \eqref{BE} is asymptotically close to the self-similar profile $\bm{f}_{\alpha}(t,\cdot)$.  Notice that our statement is quantitative in the sense that explicit rate of convergence toward zero for the difference $f(t,\cdot)-\bm{f}_{\alpha}(t,\cdot)$ is provided.  {Such rate is algebraic and, interestingly, is related to the mass, momentum and energy of the profile $\psi_\alpha$ as well as to the spectral gap of the classical (elastic) Boltzmann linearized operator.}  Observe also that the convergence is established in $L^{1}$-space with exponential weight, but such strong tail is not demanded for the initial datum.  This improvement in weight from polynomial to exponential is obtained by exploiting the instantaneous appearance of exponential moments for Boltzmann-like equation associated to hard potentials.
\item[--]   Notice that the answers to both queries (Q1) and (Q2) are related.  Indeed, we are not able to obtain \emph{in a direct way} the behavior of the moments $n_{f}(t),\bm{u}_{f}(t)$ and $T_{f}(t)$ by inspecting just the moments equations associated to \eqref{BE}. Surprisingly, the inspection of these moments equations just allows us to get the decay of the product $n_{f}(t)\sqrt{T_{f}(t)}$
but not the decay of each term.  We are able to determine the long-time behavior of such moments after exploiting the convergence of the whole solution $f(t,v)$, see Section \ref{sec:original}.\end{itemize}

\subsection{Strategy of the proof and novelty of the current approach}
It appears convenient along the proof of Theorem \ref{main:no-scaled} to rather investigate the solution of the rescaled equation \eqref{rescaBE} because it is conservative.  For such rescaled equation, the main result can be formulated as follows.
\begin{theo}\phantomsection\label{theo:main-rescaled} 
Under the Assumptions of Theorem \ref{main:no-scaled} on the initial datum $f_{0}$, let $f(t,v)$ be the unique solution to \eqref{BE} associated to the initial datum $f_{0}$ and let $\psi(\tau,\xi)$ be the associated rescaled function given by \eqref{scalingPsi}. Then, there exists some explicit $A>0$ such that for any $a \in (0,A/2)$, for any $\varepsilon >0$  there exist some explicit  $\alpha_{c} \in (0,\alpha_{0})$ such that, for all $\alpha \in (0,\alpha_{c})$
$$\int_{\R^{d}}\left|\psi(\tau,\xi)-\psi_{\alpha}(\xi)\right|\exp(a|\xi|) \d\xi  \leq C_{\varepsilon}(\alpha)\exp(-(\mu_{\star}-\varepsilon)\tau) \qquad \forall \tau \geq 1, $$
where $\mu_{\star}$ denotes the spectral gap of the linearized operator $\mathscr{L}_{0}$ associated to the elastic Boltzmann equation in $L^{2}(\M^{-1})$ and $C_{\varepsilon}(\alpha)$ is a positive explicit constant depending on $\alpha$ and $\varepsilon.$
\end{theo}

\begin{nb}\label{rem:rate}
There are two noticeable facts in Theorem \ref{theo:main-rescaled}:
\begin{enumerate}[(1)]
\item the rate of convergence is nearly optimal, being as close as desired to the rate of convergence to equilibrium for the elastic Boltzmann equation  $\mathcal{O}(e^{-\mu_{\star}t}).$ This is an important contrast with respect to the results obtained so far in the context of granular gases \cite{MiMo3,Tr} for which the rate of convergence to self-similarity is not continuous with respect to the elastic limit. To be more precise, in \cite{MiMo3,Tr}, if $\eta \in (0,1)$ denotes the inelasticity parameter, then the decay to the self-similar profile is  $\mathcal{O}(e^{-c\,(1-\eta)\,t})$ for some explicit $c >0$.  As a consequence, the elastic limit $\eta \to 1$ yields no relaxation at all,  whereas it is well-known that the solution to the elastic Boltzmann equation converges exponentially fast to equilibrium. In Theorem \ref{theo:main-rescaled}, in the elastic limit $\alpha \to 0$ one exactly recovers the optimal rate of convergence to equilibrium of the elastic Boltzmann equation. 
\item The rate of convergence of the self-similar problem is independent of $\alpha$.  That is, all corresponding relaxation related to annihilation is hidden in the rescaling.  Thus, the self-similarity rescaling becomes a tool that decouples the annihilation dynamics from the elastic collision dynamics.  This is a powerful tool for analysis.\end{enumerate} \end{nb}

The strategy we adopt to prove the results combines the spectral analysis of the linearized operator and the entropy-entropy production method.  The introduction of the linearized operator in the rescaled equation may seem, at first sight, as a bad idea since the rescaled problem \eqref{rescaBE} is \textit{\textbf{non-autonomous}}.   However, it reveals to be very efficient because, essentially, the rescaled equation is \textit{\textbf{conservative}}. 

Let us try to describe more precisely our approach.  In the weighted space
$$\X_{0}=L^{1}(\m), \qquad \m(\xi)=\exp(a|\xi|)$$
where $a>0$ is some suitable number, we can introduce the linearized operator around the profile $\psi_{\alpha}$ as follows.
\begin{defi}\phantomsection \label{defi:linear}
For any  {$\alpha \in (0,\alpha_{0})$}, introduce the linear operator $\LL\::\:\D(\LL) \subset \X_{0} \to \X_{0}$ by
\begin{align*}
\LL h(\xi)=(1-\alpha)\big[\Q(h,\psi_\alpha)(\xi)+&\Q(\psi_{\alpha},h)(\xi)\big] - \alpha \big[\Q_-(\psi_\alpha,h)(\xi)+\Q_{-}(h,\psi_{\alpha})(\xi)\big]\\
& - \mathbf{A}_\alpha h(\xi)-\mathbf{B}_\alpha\, \xi \cdot \nabla_\xi h(\xi), \qquad \forall h \in \D(\LL)\end{align*}
with domain $\D(\LL)$ given by $\D(\LL)=\W^{1,1}_{1}(\m)$. We also denote by $\mathscr{L}_{0}$ the \emph{elastic} linearized operator $\mathscr{L}_{0}\::\:\D(\mathscr{L}_{0}) \subset \X_{0} \to \X_{0}$ given by
$$\mathscr{L}_{0}h=\Q(h,\M)+\Q(\M,h), \qquad \forall h \in \D(\mathscr{L}_{0})$$
with $\D(\mathscr{L}_{0})=L^{1}_{1}(\m)$ and where $\M$ is the unique Maxwellian with same mass, momentum and energy as $\psi_{\alpha}$ given by \eqref{M}.
\end{defi}

Then, one can prove that, for $\alpha$ small enough, $(\LL,\D(\LL))$ generates a $C_{0}$-semigroup $\{\mathcal{S}_{\alpha}(t)\,;\,t \geq 0\}$ in $\X_{0}$ (see Theorem \ref{theo:gen} for a precise statement) with the following spectral properties and decay.
\begin{theo}\phantomsection\label{theo:decayX0} Let us fix ${\nu}_{*}'\in (0,\mu_\star)$. There exists $\alpha^{\star} \in (0,\alpha_{0})$ such that, for any $\alpha \in (0,\alpha^{\star})$ the operator $\LL\::\:\D(\LL) \subset \X_{0} \to \X_{0}$ satisfies:
\begin{enumerate}[1)]
\item The spectrum $\mathfrak{S}(\LL)$ is such that
$$\mathfrak{S}(\LL) \cap \{z \in \mathbb{C}\;;\,\mathrm{Re}z\geq -{\nu}_{*}'\}=\{\mu_{\alpha}^{1},\ldots,\mu_{\alpha}^{d+2}\}$$
where $\mu_{\alpha}^{1},\ldots,\mu_{\alpha}^{d+2}$ are eigenvalues of $\LL$ (not necessarily distinct) of finite algebraic multiplicity.  
\item For all $\mu \in (0,{\nu}_{*}') $, there is $C_{\mu} >0$ such that
$$\left\|\mathcal{S}_{\alpha}(t)\big(\mathbf{Id}-\mathbf{P}_{\alpha}\big)\right\|_{\mathscr{B}(\X_{0})} \leq C_{\mu}\exp(-\mu\,t) \qquad \forall t \geq 0$$
where $\mathbf{P}_{\alpha}$ denotes the spectral projection associated to $\{\mu_{\alpha}^{1},\ldots,\mu_{\alpha}^{d+2}\}$ in $\X_{0}$.
\end{enumerate}
\end{theo}
\begin{nb} The approach to prove the above result is reminiscent to the recent contributions \cite{GMM,MiMo3,Tr} and consists in a perturbation argument around the elastic limit combined with some abstract enlargement and factorization arguments as developed in \cite{GMM}. 
\end{nb}
\begin{nb}
As will be seen later on, the sign of the eigenvalues $\mu_{\alpha}^{i}$ $(i=1,\ldots,d+2)$ \emph{do not play any role in our subsequent analysis} which is an important contrast with respect to the analysis performed in \cite{MiMo3} and \cite{Tr}. On this point, it is an interesting open question to determine the sign of the eigenvalues $\mu_{\alpha}^{i}$. It seems to be a non trivial problem and the fine asymptotics of $\mu_{\alpha}^{i}$ for $\alpha \simeq 0$  would provide an interesting complement of the above result.\end{nb} 
  
Considering the fluctuations around the equilibrium 
$$h(t,\xi)=\psi(t,\xi)-\psi_{\alpha}(\xi), \qquad t \geq 0,$$
it can be shown that $h$ satisfies the following quasi-linear equation in mild form, see Section \ref{sec:stabi7} for details,
\begin{equation}\label{eq:duham1}
    h(t) =
    \mathcal{S}_\alpha(t)h(t_{0})
    + \int_{t_{0}}^t \mathcal{S}_\alpha({t-s})
    \mathbf{G}_\alpha(s) \,\d s, \qquad \forall t   \geq t_{0} \geq 0,
  \end{equation}
where, roughly speaking,
$$\mathbf{G}_{\alpha}(s)=\mathbb{B}_{\alpha}(h(s),h(s)) + \mathcal{O}(\alpha).$$
This is where the entropy-entropy production approach enters the game.  It is well-known that for elastic interactions the dissipation of entropy forces, by some kind of La Salle's principle, the solution of the Boltzmann equation to become close to equilibrium.  An important breakthrough in the study of the Boltzmann equation has been to make this idea quantitative by using some version of the so-called Cercignani's conjecture \cite{vill}.   This results in explicit estimates on the time needed for any solution to the Boltzmann equation to fall into a vicinity of the equilibrium.\smallskip

Even though the above equations \eqref{BE} and \eqref{rescaBE} \emph{do not exhibit} any dissipation of (relative) entropy properties,   we expect the persistence of the above behavior in the elastic limit. This idea is made rigorous in Section \ref{sec:stabi7} using in a crucial way the fact that the rescaled equation is conservative.  We are led to an estimate of the type: there exists some $\alpha^{\ddagger}$ small enough such that, for $\alpha \in (0,\alpha^{\ddagger})$ it holds 
$$\left\|\psi(t)-\psi_{\alpha}\right\|_{\X_{0}} \leq \ell(\alpha) \qquad \forall t > T(\alpha)$$ 
for some explicit time $T(\alpha) >0$ large enough  and some function $\ell(\alpha)$ with $\lim_{\alpha\to0}\ell(\alpha)=0.$ This allows to sharpen our estimate on $\mathbf{G}_{\alpha}(s)$ yielding
$$\|\mathbf{G}_{\alpha}(s)\|_{\X_{0}} \leq \varepsilon(\alpha)\,\|h(s)\|_{\X_{0}}, \qquad \forall s \geq t_{0} \geq T(\alpha),$$
where $\varepsilon(\alpha) \to 0$ as $\alpha\to0.$ Unfortunately, this is not enough to obtain the convergence of $h(t)$ to $0$ in the Duhamel representation \eqref{eq:duham1} since the semigroup $\{\mathcal{S}_{\alpha}(t)\,;\,t \geq 0\}$ \emph{does not decay to zero in full generality} (recall we do not know the sign of the eigenvalues $\mu_{\alpha}^{i}$). From Theorem \ref{theo:decayX0}, the decay happens only when acting on the range of $\mathbf{Id-P}_{\alpha}$. Because of the highly dissipative behavior of $\LL$, the precise expression of the projection $\mathbf{P}_{\alpha}$ seems difficult to obtain. At this point, a crucial role is played by the fact that the scaling we choose is \emph{exactly} the one which makes \eqref{rescaBE} \emph{conservative}. Because of this additional property, the fluctuation $h(t,\xi)$ have zero mass, momentum and kinetic energy and, as such, satisfies
$$\mathbf{P}_{0}h(t)=0 \qquad \forall t \geq 0,$$
where $\mathbf{P}_{0}$ is the spectral projection on the kernel of the elastic operator $\mathscr{L}_{0}$. This obvious but fundamental property together with the fact that, in some sense,
$$\mathbf{P}_{\alpha}-\mathbf{P}_{0}=\mathcal{O}(\alpha)$$
allows us to prove that, for $\alpha$ small enough, 
\begin{equation}\label{eq:PP}
\|h(t)\|_{\X_{0}} \leq C\,\|(\mathbf{Id-P}_{\alpha})h(t)\|_{\X_{0}} \qquad \forall t \geq 0.\end{equation}
In other words, \textit{\textbf{it suffices to study the dynamic of Eq. \eqref{rescaBE} in the ``orthogonal'' space $\mathrm{Range}(\mathbf{Id-P}_{\alpha})$}}. However, it is important to emphasize the contrast here with the classical elastic Boltzmann equation: for such a problem, as well-documented, the \emph{nonlinear dynamic} occurs  exclusively on the ``orthogonal'' $\mathrm{Range}(\mathbf{Id}-\mathbf{P}_{0})$. Here, this is not the case, some part of the nonlinear dynamic still occurs on the space $\mathrm{Range}(\mathbf{P}_{\alpha})$ but according to the estimate \eqref{eq:PP}, such a dynamic is \emph{controlled} by the one occuring in $\mathrm{Range}(\mathbf{Id-P}_{\alpha})$. \medskip

The combination of these two approaches -- spectral analysis and entropy method -- is reminiscent of the work \cite{MiMo3} on granular gases and strongly relies on the understanding of the elastic problem corresponding to $\alpha=0$. However, the approach we follow is novel in different aspects:\medskip

\indent 1. Our approach is \emph{global in essence}. This contrasts with the approach of \cite{MiMo3} (see also \cite{CaLo}) where local stability estimates (in which exponential convergence is proven for small perturbations of the equilibrium) are first established and then suitable entropy estimates are used as a tool to pass from local to global stability. Here, even if we fully exploit the spectral properties of the linearized operator and the decay of the associated semigroup, our approach does not rely \emph{at all} on the study of close-to-equilibrium solutions to \eqref{rescaBE}. We directly prove the \emph{global} stability without proving first the \emph{local} one. We insist here in particular on the fact that the sign of the eigenvalues $\mu_{\alpha}^{1},\ldots,\mu_{\alpha}^{d+2}$ in Theorem \ref{theo:decayX0} do not play any role in our analysis (it is not completely clear actually whether these eigenvalues are nonnegative or not).\medskip

\indent 2. Related to this first point, our study of the global stability exploits in a crucial way the fact that the rescaled equation is fully \emph{conservative}. In the granular gases case studied in \cite{MiMo3}, the equation in rescaled variable does not preserve  energy. This is not the case here where \eqref{rescaBE} preserves mass, momentum and kinetic energy. The price to pay for obtaining a fully conservative equation is that this latter is \emph{non-autonomous}. As such, the linearization around steady solution is not completely natural but imposed. However, as explained previously, dealing with a conservative equation allows us to exploit -- in a crucial way -- the fact that the dynamic in the space $\mathrm{Range}(\mathbf{P}_{\alpha})$ is completely controlled by the dynamic in $\mathrm{Range}(\mathbf{Id-P}_{\alpha})$.\medskip

\indent 3. By virtue of the point 2, the rate of convergence to equilibrium for the rescaled equation is \emph{sharp} in the sense that it allows to recover, in the limit $\alpha \to 0$, the decay to equilibrium for the Boltzmann equation in $\mathcal{O}(e^{-\mu_{\star}t}).$ We already commented on this point in Remark \ref{rem:rate} explaining the contrast with the analysis in \cite{MiMo3,Tr}. Let us emphasize at this point that recovering the sharp decay rate is made possible again thanks to the conservative form of the rescaled equation and the method described in point 2 and in the previous paragraph.  Such novel approach is the main contribution of our paper which allows to understand in a better way the role of the linearized operator in the rescaled equation. Let us also mention that this method is robust enough and applies to the models of granular gases described earlier (at the price of performing the scaling which exactly preserves the energy).\medskip 

\indent 4. For the entropy-entropy production method, we follow a time-dependent approach initiated in \cite{AloLocmp} in the context of granular gases. With respect to this approach, one can see that the regularity assumptions made on the initial datum are minimal. This comes from an improvement of a well-known functional inequality obtained by C. Villani that relates the entropy production functional associated to $\Q(f,f)$ to the relative entropy. In \cite{vill}, an almost linear inequality is derived under some strong (high order) regularity on $f_{0}$. Here, we used a version of such an inequality -- obtained recently in \cite{alogam} -- where the functional inequality is far from being linear but for which the regularity on $f_{0}$ is drastically relaxed. Namely, we will resort on the following proposition.
\begin{prop}\phantomsection\label{theo:Vill} For a given function $f \in L^{1}_{2}(\R^{d}) \cap L^{2}(\R^{d})$, let $\M_f$ denote the Maxwellian
function with the same mass, momentum and energy as $f$. Assume that there exist $K_0 > 0$, $A_0 > 0$ and $q_0 \geq 2$ such that
\begin{equation}\label{villpoint}
f(v) \geq  K_0\,\exp\left(-A_0\,|v|^{q_0}\right) \qquad \forall v \in \R^{d}.
\end{equation}
Then, for all $\delta \in (0,1)$, there exists a constant $\lambda_{\delta}(f)$, depending on  $\delta$ and on $f$ only through its mass and energy and upper bounds on $A_0,$ $1/K_0$, $\|f\|_{2}$ and $\|f\|_{L^1_{s}}$, where $s = s(q_0) > 0$ such that
\begin{equation*}
\D(f) \geq \lambda_{\delta}(f) \left(\IR f(v)\log \left(\frac{f(v)}{\M_f(v)}\right)\d v \right)^{(1+\delta)(1+2/d)}
\end{equation*}
where $\D(f)$ is the  entropy dissipation functional associated to the elastic Boltzmann operator 
\begin{equation*}
\D(f)=-\IR \Q\big(f,f\big)(v)\log\left(\frac{f(v)}{\M_f(v)}\right)\d v.
\end{equation*}
\end{prop}

Notice that, in order to be able to apply the above Proposition to the solution $\psi(t)$ to \eqref{rescaBE}, we need first to prove the appearance of gaussian-like pointwise lower bound for such solutions, see Theorem \ref{a-l3}. \medskip

5. Finally, a novelty of our approach also lies in the control of the Fisher information ${I}(\psi(t))$ of the solution to \eqref{rescaBE}. Recall that, for a given nonnegative function $f$, the Fisher information of $f$ is defined as
\begin{equation}\label{eq:fisher}
{I}\big(f\big) =\frac{1}{4}\int_{\mathbb{R}^{d}}\frac{|\nabla f(\xi)|^{2}}{f(\xi)}\d\xi = \int_{\mathbb{R}^{d}}\big|\nabla\sqrt{f(\xi)}\big|^{2}\d\xi\,.
\end{equation}
It is very easy to see that, to obtain a uniform control of the solution $\psi(t)$ in spaces like $\W^{1,1}_{1}(\m)$, it is enough to prove that 
$$\sup_{t\geq 0}{I}(\psi(t)) < \infty.$$
We prove that such an estimate is true in Theorem \ref{theo:fisher} under minimal regularity on the initial datum $\psi_{0}$, which in dimension say $d=3$ is assumed to have finite Fisher information and to lie in $H^{1}(\R^{d})$ (with some algebraic moments). The uniform control of Fisher information for solutions to Boltzmann like equation seems to be completely new. We mention here the seminal work \cite{villa} dealing with the Boltzmann equation for Maxwell-like collision kernels and for which an algebraic growth of the Fisher information is obtained. Our approach relies in a heavy way on the appearance of gaussian-like pointwise lower bounds (Theorem \ref{a-l3}) and on the precise control on the way the various parameter in these lower bounds depend on time. We refer to Section \ref{sec:stabi} for more details on these new estimates. Again, the method we propose here seems robust enough to apply to a larger variety of kinetic models exhibiting the appearance of such pointwise lower bounds.

\subsection{Organization of the paper} 
The paper is organized as follows.  We describe in Section \ref{sec:moments} the evolution of the moments for the nonlinear equation in original variable \eqref{BE}. We are able, at this stage, to obtain only partial results yielding just the decay of the product $n_{f}(t)^2\,T_{f}(t)$.  However, this will turn of paramount importance since such a decay is actually governing the long-time behavior of the time scaling function $\tau(t)$ (see Eq. \ref{eq:tau}). The rest of Section \ref{sec:moments} makes rigorous the scaling performed earlier and provides the proof of Proposition \ref{prop:cauc}.

After these two Sections, the paper is divided into three parts:  Part \ref{part1} of the paper is devoted to the thorough study of the linearized operator $\LL$ and culminates with  the proof of Theorem \ref{theo:decayX0}. Our approach to Theorem \ref{theo:decayX0} is inspired by the one introduced in \cite{MiMo3} and revisited in \cite{GMM,Tr}.  It consists, roughly speaking, in a perturbation argument which exploits the spectral analysis of the linearized elastic Boltzmann operator $\mathscr{L}_{0}$. In a more precise way, we first use the fact that the spectrum of $\mathscr{L}_{0}$ is well localized, meaning that it admits a spectral gap in a large class of Sobolev spaces; second, we show that, for $\alpha$ small enough
$\LL-\mathscr{L}_{0}$ is of order $\mathcal{O}(\alpha)$ for some suitable norm; finally, to deduce the decay of the semigroup from the spectral structure of the generator, we need to use some abstract spectral mapping theorem established in \cite{MiSc}. The decay in $\X_{0}$ is then deduced from that in $\X_{1}$ thanks to an abstract enlargement and factorization argument as developed in \cite{GMM}. 

Part \ref{part2} of the paper is devoted to the stability analysis. In Section \ref{sec:entropy}, we develop the time-dependent entropy-entropy production method. In Section \ref{sec:stabi}, we first obtain uniform bounds on the solution $\psi(t,\xi)$ to \eqref{rescaBE} -- in particular obtaining the important estimate on the Fisher information $\mathcal{I}(\psi(t))$ and then prove Theorem \ref{theo:main-rescaled}. Finally, in Section \ref{sec:original}, we turn back to the original variable and prove Theorem \ref{main:no-scaled}. 

The final part of the paper is made of four Appendices which collect several technical results used in the main core of the paper. In particular, Appendix \ref{app:prooflemma} gives the proof of two technical results used in Part \ref{part1}. Appendix \ref{app:point} collect the main properties of the solutions to the rescaled equation \eqref{rescaBE} and, in particular, the appearance of pointwise lower bounds which is fundamental for the use of the above Proposition \ref{theo:Vill}. Recall here that, for such lower bound, it is important to get a control of the various constant with respect to time in order to perform our analysis of the Fisher information. In Appendix \ref{app:gener}, we prove that the linearized operator $(\LL,\D(\LL))$ is the generator of a $C_{0}$-semigroup in $\X_{0}$ exploiting well-known abstract generation results in $L^{1}$-spaces. 

\section{Evolution of the moments for the nonlinear equation}\label{sec:moments}

We consider here the evolution of macroscopic physically relevant quantities associated to the fully nonlinear Boltzmann equation that we recall here for convenience
\begin{equation}\label{BE1}\begin{cases}
\partial_t f(t,v)&=(1-\alpha)\Q(f,f)(t,v) -\alpha \Q_-(f,f)(t,v) \qquad t > 0\\
f(0,v)&=f_0(v)\,. \end{cases}\end{equation}

This kinetic equation has no conserved macroscopic quantities and density is decreasing to zero. To be more precise, let us {recall that}, for any $t \geq 0$, the density
$$n_{f}(t)=\int_{\R^d} f(t,v)\d v,$$
the momentum $\bm{u}_{f}(t) \in \R^{d}$ and the temperature $T_{f}(t) \geq 0$ {are} defined respectively by
$$n_{f}(t)\bm{u}_{f}(t)=\int_{\R^d} f(t,v)v \,\d v \in \R^d\,\qquad \text{
and } \qquad d\,n_{f}(t)T_{f}(t) = \int_{\R^d}f(t,v)|v-\bm{u}_{f}(t)|^2\d v.$$

\subsection{Evolution of first moments}

 We aim here to deduce the precise rate of convergence to zero of the quantity
 $${E}_{f}(t)=dn_{f}(t)^{2}T_{f}(t), \qquad \forall t \geq 0,$$  and our main result is the following
\begin{theo}\phantomsection\phantomsection\label{main} There exists some explicit $\alpha_\star \in (0,1)$ such that, for any $\alpha \in (0,\alpha_\star)$ and any nonnegative $f_{0} \in L^{1}_{3}(\R^{d})$, the associated solution $f(t,v)$  to \eqref{BE1}  satisfies the following:
$$ \left(c_0+ 2\,t\right)^{-2} \leq d\,n_{f}(t)^{2}  \,T_{f}(t) \leq  \,\left(c_1+ \frac{\alpha}{2}\,t\right)^{-2} \qquad \forall t \geq 0$$
for positive constants $c_0,c_1 > 0$ depending only on the initial distribution $f_0$ and not on $\alpha$, 
$$c_{0}:=\int_{\R^{d}}f_{0}(v)|v-\bm{u}_{f}(0)|\d v, \qquad c_{1}:=\frac{1}{\sqrt{E_{f}(0)}}.$$
\end{theo}\phantomsection
 Let us introduce the moments
$$ \qquad M_{k}(t)=\int_{\R^d} f(t,v)|v-\bm{u}_{f}(t)|^{k}\d v \qquad t \geq 0, \qquad (k \geq 0).$$
With such notations, $n_{f}(t)=M_0(t)$ and $dn_{f}(t)T_{f}(t) =M_2(t)$ and $E_{f}(t)=M_{0}(t)M_{2}(t).$

In all the sequel, we consider a nonnegative initial datum $f_{0} \in L^{1}_{3}(\R^{d})$ and denote by $f(t)$, $t \geq 0$ the associated solution to \eqref{BE1}.  One has the following
\begin{lem}\phantomsection\phantomsection\label{lem1}  One has
\begin{equation}\label{M0M2}
\dfrac{\d}{\d t}M_{0}(t) \leq -\alpha \,M_{0}(t)\,M_{1}(t) \qquad \text{ and } \qquad \dfrac{\d}{\d t}M_{2}(t) \leq -\alpha\,M_{0}(t)\,M_{3}(t) \quad \forall t\geq 0.\end{equation}
As a consequence
$$E_{f}(t) \leq \left(\frac{1}{\sqrt{E_{f}(0)}}+\frac{\alpha}{2}t\right)^{-2}, \qquad \forall t \geq0.$$
\end{lem}
\begin{proof} The proof follows from integration of \eqref{BE1} and the fact that both density and kinetic energy are conserved by the Boltzmann operator $\Q$. Therefore, we get
\begin{equation}\label{eq:nfEf1}
\dfrac{\d}{\d t} n_{f}(t)=-\alpha \int_{\R^d}\Q_-(f,f)(t,v)\d v,\qquad \dfrac{\d}{\d t}M_{2}(t)=-\alpha \int_{\R^d}\Q_-(f,f)(t,v)|v-\bm{u}_{f}(t)|^2\d v,
\end{equation}
so that
\begin{equation*}\begin{split}
\dfrac{\d}{\d t}n_{f}(t)&=-\alpha \int_{\R^{2d}} f(t,v)f(t,\vb)|v-\vb|\d \vb\,\d v \\
\dfrac{\d}{\d t}M_{2}(t)&=-\alpha\int_{\R^{2d}} f(t,v)f(t,\vb)\,|v-\vb|\,|v-\bm{u}_{f}(t)|^2\d v\d\vb.\end{split}\end{equation*}
According to Jensen's inequality  one has
$$\int_{\R^d} f(t,\vb)|v-\vb|\d\vb \geq n_{f}(t) |v-\bm{u}_{f}(t)|  \qquad \forall t \geq 0.$$
Therefore
{\small $$\dfrac{\d}{\d t} n_{f}(t) \leq -\alpha\,n_{f}(t)\int_{\R^d} f(t,v)|v-\bm{u}_{f}(t)|\d v \quad \text{and} \quad\dfrac{\d}{\d t}M_{2}(t) \leq -\alpha\,n_{f}(t) \int_{\R^d}f(t,v)|v-\bm{u}_{f}(t)|^3\d v$$}
from which \eqref{M0M2} follows. To deduce from this the decay of $E_{f}(t)$, we simply notice that, thanks to \eqref{M0M2},
$$\dfrac{\d}{\d t}E_{f}(t)=\,M_{2}(t)\,\dfrac{\d}{\d t}M_{0}(t)\,+M_{0}(t)\frac{\d}{\d t}M_{2}(t) \leq -\alpha\,M_{0}(t)^{2}M_{3}(t) \leq -\alpha\,E_{f}(t)^{3/2}$$
where we used that $M_{3}(t) \geq M_{0}^{-1/2}(t)M_{2}(t)^{3/2}$ thanks to H\"older's inequality. The result follows.
\end{proof}
One sees therefore that, to capture the asymptotic behavior of both $n_{f}(t)$ and $T_{f}(t)$, it will be necessary to understand the behavior of larger order moments (typically $M_{3}(t)$). One begins with recalling the Povzner's estimates obtained in \cite{jde}. For low order moments,  one has the following which comes from a combination of \cite[Lemma 3.1]{jde} and \cite[Lemma 3.7]{jde}.
\begin{lem}\phantomsection\label{Mk<1} For any $k \in (0,1)$ and any nonnegative mapping $\Psi\::\:\R^{d}\to \R$ one has
\begin{multline*}
-\int_{\R^d}\mathbb{B}_\alpha(\Psi,\Psi)(v)\,|v|^{2k}\d v \leq -\dfrac{\beta_k(\a)}{2}\int_{\R^{2d}}\Psi(v)\,\Psi(v_*)|v-\vb| \,\left(|v|^2+|\vb|^2\right)^k\d v\d\vb\\
+\dfrac{1}{2}\int_{\R^{2d}}\Psi(v)\Psi(\vb)|v-\vb| \,\left(|v|^{2k}+|\vb|^{2k}\right)\d v\d\vb,\end{multline*}
where $\beta_k(\a)=(1-\a)\varrho_k$ with
\begin{equation}\label{varrhoK}
\varrho_k =\int_{\S^{d-1}}\left[ \left(\dfrac{1+ \hat{U} \cdot \sigma}{2}\right)^k+\left(\dfrac{1- \hat{U} \cdot \sigma}{2}\right)^k\right] \d\sigma=\dfrac{|\S^{d-2}|}{|\S^{d-1}|}2^{1-k} \,\int_{-1}^1 (1+t)^k\,\left( {1-t^2}\right)^{\tfrac{d-3}{2}}\d t.
\end{equation}
\end{lem}\phantomsection

To prove Theorem \ref{main}, we first need to compute the evolution of $M_{1}(t)$:
\begin{lem}\phantomsection\phantomsection \label{lem:M1/2} 
There exists $\alpha_\star=\frac{\varrho_{1/2}-1}{\varrho_{1/2}+1} \in (0,1)$ such that, for all $\alpha \in (0,\alpha_{*})$ the unique solution to \eqref{BE} satisfies
$$\left(\frac{1}{M_{1}(0)}+2t\right)^{-1} \leq M_{1}(t) \qquad \forall t \geq 0.$$
\end{lem}\phantomsection
\begin{proof} The proof resumes some of the arguments of \cite[Lemma 3.7]{jde}. Precisely, multiplying Eq. \eqref{BE} by $|v-\bm{u}_{f}(t)|$ and integrating over $\R^d$ one has
\begin{equation}\label{eq:M1t}
\begin{split}
\dfrac{\d}{\d t}M_{1}(t)&=\int_{\R^d} \mathbb{B}_\alpha(f,f)(t,v)\,|v-\bm{u}_{f}(t)|\d v + \int_{\R^{d}}f(t,v)\partial_{t}|v-\bm{u}_{f}(t)|\d v\\
&=\int_{\R^{d}}\mathbb{B}_\alpha(f,f)(t,v)\,|v-\bm{u}_{f}(t)|\d v -\dot{\bm{u}}_{f}(t) \cdot \int_{\R^{d}}f(t,v)\frac{v-\bm{u}_{f}(t)}{|v-\bm{u}_{f}(t)|}\d v.\end{split}\end{equation}
Using now that 
$$\dfrac{\d}{\d t}(n_{f}(t)\bm{u}_{f}(t))=-\alpha\int_{\R^{d}}\Q_{-}(f,f)(t,v)\,v\,\d v=\dot{n}_{f}(t)\bm{u}_{f}(t)+n_{f}(t)\dot{\bm{u}}_{f}(t)$$
together with \eqref{eq:nfEf1} we easily get that
$$\dot{\bm{u}}_{f}(t)=-\frac{\alpha}{n_{f}(t)}\int_{\R^{d}}\Q_{-}(f,f)(t,v)(v-\bm{u}_{f}(t))\d v.$$
Consequently, 
\begin{equation*}
\left|\dot{\bm{u}}_{f}(t) \cdot \int_{\R^{d}}f(t,v)\frac{v-\bm{u}_{f}(t)}{|v-\bm{u}_{f}(t)|}\d v \right|\leq |\dot{\bm{u}}_{f}(t)| n_{f}(t) \leq \alpha\int_{\R^{d}}\Q_{-}(f,f)(t,v)|v-\bm{u}_{f}(t)|\d v.\end{equation*}
Using that
\begin{multline*}
\int_{\R^{d}}\Q_{-}(f,f)(t,v)|v-\bm{u}_{f}(t)|\d v=\int_{\R^{2d}}f(t,v)f(t,\vb)|v-\vb|\,|v-\bm{u}_{f}(t)|\d v\d\vb \\
\leq \int_{\R^{2d}}f(t,v)f(t,\vb)\left(|v-\bm{u}_{f}(t)|+|\vb-\bm{u}_{f}(t)|\right)|v-\bm{u}_{f}(t)|\d v\d\vb
=M_{0}(t)M_{2}(t)+M_{1}(t)^{2}\end{multline*}
we get  from \eqref{eq:M1t},
\begin{equation}\label{eqM1}
\dfrac{\d}{\d t}M_{1}(t) \geq \int_{\R^{d}}\mathbb{B}_\alpha(f,f)(t,v)\,|v-\bm{u}_{f}(t)|\d v - \alpha\left(M_{0}(t)M_{2}(t)+M_{1}(t)^{2}\right).\end{equation}
Using Lemma \ref{Mk<1} with $k=1/2$ and $\Psi(v)=f(t,v+\bm{u}_{f}(t))$, we obtain that
\begin{align*}
\int_{\R^d} \mathbb{B}_\alpha(f,f)(t,v)&\,|v-\bm{u}_{f}(t)|\d v=\int_{\R^{d}} \mathbb{B}_{\alpha}(f(t,\cdot+\bm{u}_{f}(t)),f(t,\cdot+\bm{u}_{f}(t)))(v)|v|^{2k}\d v\\
& \geq  \frac{1}{2}\int_{\R^{2d}}f(t,v+\bm{u}_{f}(t))f(t,\vb+\bm{u}_{f}(t)) \mathcal{J}(v,\vb)\d v \d\vb
\end{align*}
where
$$\mathcal{J} (v,\vb)=\beta_{1/2}(\a)|v-\vb| \,\left(|v|^2+|\vb|^2\right)^{1/2} -|v-\vb| \,\left(|v| +|\vb| \right).$$
Since $\left|\,|v|-|\vb|\,\right| \leq |v-\vb|\leq |v|+|\vb|$ and $\left(|v|^2+|\vb|^2\right)^{1/2} \geq \left|\,|v|-|\vb|\,\right|$ one gets that
\begin{align*}
\mathcal{J}(v,\vb) \geq \beta_{1/2}(\a)\,&\left|\,|v|-|\vb|\,\right|^2 -\left(|v| +|\vb| \right)^2\\
&=\left(\beta_{1/2}(\a)-1\right)\,\left(|v|^2+|\vb|^2\right) -2\left(\beta_{1/2}(\a)+1\right)\,|v|\,|\vb|\,.
\end{align*}
Since
\begin{align*}
\int_{\R^{2d}}f(t,v+\bm{u}_{f}(t))f(t,\vb&+\bm{u}_{f}(t))\left(|v|^{2}+|\vb|^{2}\right)\d v\d \vb\\
& = 2\int_{\R^{2d}}f(t,v)f(t,\vb)|v-\bm{u}_{f}(t)|^{2}\d v\d\vb=2M_{0}(t)M_{2}(t)\end{align*}
while
$$\int_{\R^{2d}}f(t,v+\bm{u}_{f}(t))f(t,\vb+\bm{u}_{f}(t))|v|\,|\vb|\d v\d\vb =\left(\int_{\R^{d}}f(t,v)|v-\bm{u}_{f}(t)|\d v\right)^{2}=M_{1}(t)^{2},$$
we get
\begin{equation}\label{eq:BaM1}
\int_{\R^d} \mathbb{B}_\alpha(f,f)(t,v)\,|v-\bm{u}_{f}(t)|\d v \geq \left(\beta_{1/2}(\a)-1\right)M_{0}(t)M_{2}(t)
-\left(\beta_{1/2}(\a)+1\right)M_{1}(t)^{2}.\end{equation}
Combining this with \eqref{eqM1} we finally obtain
\begin{equation}\label{eqM1-1}\dfrac{\d}{\d t}M_{1}(t) \geq \left(\beta_{1/2}(\a)-\alpha-1\right) M_{0}(t)M_{2}(t) -\left(\beta_{1/2}(\a)+1-\alpha\right)\,M_{1}(t)^2.\end{equation}
Now, setting $\alpha_\star=\frac{\varrho_{1/2}-1}{\varrho_{1/2}+1},$
one sees that, for any $0 < \alpha < \alpha_\star$, $\beta_{1/2}(\a) > 1+\alpha$. Moreover, Cauchy-Schwarz inequality ensures that $M_{1}(t)^2 \leq M_{0}(t)M_{2}(t)$ so that \eqref{eqM1-1} reads
\begin{equation*}
\dfrac{\d}{\d t}M_{1}(t) \geq \left(\beta_{1/2}(\a)-\alpha-1-(\beta_{1/2}(\a)+1-\alpha)\right)M_{1}(t)^2=-2M_{1}(t)^2 \qquad \forall t \geq 0.\end{equation*}
Integrating this differential inequality gives the result.
\end{proof}

The above inequality yields the optimal rate of convergence.
\begin{proof}[Proof of Theorem \ref{main}] Let $\alpha \in (0,\alpha_\star)$ be fixed. Using again that $M_{0}(t)M_{2}(t) \geq M_{1}^{2}(t)$, we deduce from Lemma \ref{lem:M1/2} the lower bound
$$M_{0}(t)M_{2}(t)\geq \left(\frac{1}{M_{1}(0)}+2t\right)^{-2}\qquad \forall t \geq 0 $$
which gives the conclusion thanks to Lemma \ref{lem1}.
\end{proof}

A direct consequence of  Theorem \ref{main} and Lemma \ref{lem:M1/2} is that
$$M_{1}(t) \propto (1+t)^{-1} \qquad \text{ as } t\to \infty.$$
More precisely, one has the following result.
\begin{cor}\phantomsection\phantomsection\label{corM12} There exists some explicit $\alpha_\star \in (0,1)$ such that, for any $\alpha \in (0,\alpha_\star)$, any nonnegative solution $f(t,v)$  to \eqref{BE1} associated to a nonnegative initial datum $f_{0} \in L^{1}_{3}(\R^{d})$ satisfies the following:
$$ \left(c_0+ 2\,t\right)^{-1} \leq M_{1}(t) \leq  \,\left(c_1+ \frac{\alpha}{2}\,t\right)^{-1} \qquad \forall t \geq 0$$
for positive constants $c_0,c_1 > 0$ depending on  the initial distribution $f_0$.
\end{cor}\phantomsection
\begin{proof} The lower bound comes from Lemma \ref{lem:M1/2} while the upper bound comes from the corresponding upper bound for $M_{0}(t)M_{2}(t)$ in Theorem \ref{main} together with the fact that $M_{1}(t) \leq \sqrt{M_{0}(t)M_{2}(t)}.$
\end{proof}

\subsection{Scaling and self-similarity}\label{sec:scaling}

Let us recall that we introduced in \eqref{scalingPsi} the following rescaled function $\psi(\tau,\xi)$ through
\begin{equation*}\label{scalingf}
f(t,v)=n_{f}(t)(2T_{f}(t))^{-d/2}\psi\left(\tau(t),\frac{v-\bm{u}_{f}(t)}{\sqrt{2T_{f}(t)}}\right), \qquad \forall t \geq 0.
\end{equation*}
where $n_{f}(t),T_{f}(t)$ and $\bm{u}_{f}(t)$ denote the first moments of $f(t,\cdot)$ defined by \eqref{eq:momf}. We give briefly here the proof of Proposition \ref{prop:cauc}  which asserts that, under such scaling, $\psi(\tau,\xi)$ is the unique solution to \eqref{rescaBE}. 

Using \eqref{scalingPsi}, one gets that, for $\tau=\tau(t)$ and $\xi=\frac{v-\bm{u}_{f}(t)}{\sqrt{2T_{f}(t)}}$,
\begin{multline*}
\partial_{t}f(t,v)=n_{f}(t)(2T_{f}(t))^{-d/2}\dot{\tau}(t)\partial_{\tau}\psi(\tau,\xi)\\
+\left(\dot{n}_{f}(t)(2T_{f}(t))^{-d/2}-d\dot{T}_{f}(t)n_{f}(t)(2T_{f}(t))^{-1-d/2}\right)\psi(\tau,\xi)\\
-n_{f}(t)(2T_{f}(t))^{-\frac{d+2}{2}}\dot{T}_{f}(t) \xi \cdot \nabla_{\xi} \psi(\tau,\xi)- n_{f}(t)(2T_{f}(t))^{-\frac{1+d}{2}} \dot{\bm{u}}_{f}(t)\cdot \nabla_{\xi}\psi(\tau,\xi),\end{multline*}
where the dot symbol denotes derivative with respect to $t$. Moreover, using the scaling properties of $\Q_{\pm}(f,f)$, one has
$$\Q_{\pm}(f,f)(t,v)=n_{f}(t)^{2}(2T_{f}(t))^{\frac{1-d}{2}}\Q_{\pm}(\psi,\psi)(\tau,\xi),$$
so that $\psi(\tau,\xi)$ satisfies the following equation 
\begin{multline*}
\mathbb{B}_{\alpha}(\psi,\psi)(\tau,\xi)=n_{f}(t)^{-1}(2T_{f}(t))^{-\frac{1}{2}}\dot{\tau}(t)\partial_{\tau}\psi(\tau,\xi)\\
+\left(\dot{n}_{f}(t)n_{f}(t)^{-2}(2T_{f}(t))^{-1/2}-d\dot{T}_{f}(t)n_{f}^{-1}(t)(2T_{f}(t))^{-3/2}\right)\psi(\tau,\xi)\\
-\dot{T}_{f}(t)n_{f}^{-1}(t)(2T_{f}(t))^{-3/2}\xi \cdot \nabla_{\xi}\psi(\tau,\xi) - n_{f}(t)^{-1}(2T_{f}(t))^{-1}\dot{\bm{u}}_{f}(t) \cdot \nabla_{\xi}\psi(\tau,\xi)\end{multline*}
for $\tau=\tau(t),$ $\xi=\frac{v-\bm{u}_{f}(t)}{\sqrt{2T_{f}(t)}}$.  
One sees then that choosing  the time scaling function $\tau$ in such a way that
\begin{equation*}\label{eq:scalingtau}
\dot{\tau}(t)=n_{f}(t)\sqrt{2T_{f}(t)}, \qquad \qquad \forall t \geq 0,\end{equation*}
we obtain, finally
$$
\partial_{\tau}\psi(\tau,\xi)  + \mathbf{A}_{\psi}(\tau)\,\psi(\tau,\xi) + \mathbf{B}_{\psi}(\tau)\xi\cdot\nabla_{\xi}\psi(\tau,\xi) + \mathbf{V}_{\psi}(\tau) \cdot \nabla_{\xi}\psi(\tau,\xi)
=
\mathbb{B}_{\alpha}(\psi,\psi)(\tau,\xi)$$
with
\begin{equation}\begin{cases}\label{eq:ABsca}
\mathbf{A}_{\psi}(\tau(t))&=\left(\dot{n}_{f}(t)n_{f}(t)^{-2}(2T_{f}(t))^{-1/2}-d\dot{T}_{f}(t)n_{f}^{-1}(t)(2T_{f}(t))^{-3/2}\right) \in \R\\
\\
\mathbf{B}_{\psi}(\tau(t))&=-\dot{T}_{f}(t)n_{f}^{-1}(t)(2T_{f}(t))^{-3/2}=n_{f}^{-1}(t)\dfrac{\d}{\d t}(2T_{f}(t))^{-1/2} \in \R\\
\\
\mathbf{V}_{\psi}(\tau(t))& =-n_{f}(t)^{-1}(2T_{f}(t))^{-1}\dot{\bm{u}}_{f}(t) \in \R^{d}, \qquad \forall t \geq 0.\end{cases}\end{equation} 
Introducing 
$$\alpha{\mathbf{a}}_{\psi}(\tau)=d\mathbf{B}_{\psi}(\tau)-\mathbf{A}_{\psi}(\tau), \qquad \bm{v}_{\psi}(\tau)=-\frac{1}{\mathbf{B}_{\psi}(\tau)}\mathbf{V}_{\psi}(\tau) \in \R^{d}$$ allows to write the above equation satisfied by $\psi(\tau,\xi)$ in divergence form
$$\partial_{\tau}\psi(\tau,\xi)  - \alpha\mathbf{a}_{\psi}(\tau)\,\psi(\tau,\xi) + \mathbf{B}_{\psi}(\tau)\mathrm{div}_{\xi}\big(\left({\xi}-\bm{v}_{\psi}(\tau)\right)\psi(\tau,\xi)\big) = (1-\alpha)\Q(\psi,\psi)(\tau,\xi)-\alpha\Q_{-}(\psi,\psi)(\tau,\xi).$$
Also, conservation of mass implies that
$${\mathbf{a}}_{\psi}(\tau)=\int_{\R^{d}}\Q_{-}(\psi,\psi)(\tau,\xi)\d\xi=\int_{\R^{d}\times\R^{d}}|\xi-\xi_{*}|\psi(\tau,\xi)\psi(\tau,\xi_{*})\d\xi\d\xi_{*} \geq 0.$$
The zero momentum assumption on $\psi(\tau,\xi)$ reads 
$$\mathbf{B}_{\psi}(\tau)\bm{v}_{\psi}(\tau)=-\alpha\int_{\R^{d}}\xi\,\Q_{-}(\psi,\psi)(\tau,\xi)\d\xi \in \R^{d}, \qquad \forall \tau \geq 0,$$
while conservation of kinetic energy yields
$$\alpha\mathbf{a}_{\psi}(\tau)+2\mathbf{B}_{\psi}(\tau)=\big((d+2)\mathbf{B}_{\psi}(\tau)-\mathbf{A}_{\psi}(\tau)\big)=\frac{2\alpha}{d}\int_{\R^{d}}|\xi|^{2}\Q_{-}(\psi,\psi)(\tau,\xi)\d\xi.$$
One sees easily then that this yields the expressions for $\mathbf{B}_{\psi},\mathbf{A}_{\psi}$ and $\bm{v}_{\psi}$ given by \eqref{eq:BAV} and the mapping $\psi(\tau,\xi)$ is a solution to \eqref{rescaBE}. Notice that a variant of Eq. \eqref{rescaBE} has been introduced and studied in \cite{jde} and  we  can deduce from  \cite[Theorem  1.10]{jde} that $\psi(\tau,\xi)$ is the \emph{unique nonnegative solution}, belonging to $\C([0,\infty),L^1_2(\R^d))\cap L^1_{\mathrm{loc}}((0,\infty), L^1_3(\R^d))$ to \eqref{rescaBE} with initial condition $\psi_0$. 

\begin{nb}\phantomsection\label{nb:1.3}
Notice that the coefficients $\mathbf{A}_{\psi}(\tau),$ and $\mathbf{B}_{\psi}(\tau)$ do not have definite sign. Furthermore, for the case of steady solution $\psi_{\alpha}$ for which we recall that 
$$\mathbf{A}_{\psi_{\alpha}}=:\mathbf{A}_{\alpha}, \qquad \mathbf{B}_{\psi_{\alpha}}=:\mathbf{B}_{\alpha},$$
it is not clear whether $\mathbf{A}_{\alpha}$ and $\mathbf{B}_{\alpha}$ have a sign. However
\begin{eqnarray*}
\mathbf{a}_{\psi}(\tau) & := & \frac{d \mathbf{B}_{\psi}(\tau)- \mathbf{A}_{\psi}(\tau)}{\alpha} 
\; = \; \int_{\R^d} \Q_-(\psi,\psi)(\tau,\xi) \, \d\xi,  \\
  \mathbf{b}_{\psi}(\tau)& := &\frac{ (d+2)\mathbf{B}_{\psi}(\tau)- \mathbf{A}_{\psi}(\tau)}{\alpha}
\;= \;\frac{2}{d} \int_{\R^d} \Q_-(\psi,\psi)(\tau,\xi) \, |\xi|^2 \d\xi
\end{eqnarray*}
are both nonnegative for any $\tau \geq 0.$ Again, we use the shorthand notations $\mathbf{a}_{\alpha}=\mathbf{a}_{\psi_{\alpha}},$ $\mathbf{b}_{\alpha}=\mathbf{b}_{\psi_{\alpha}}$ for the steady solution $\psi_{\alpha}.$
\end{nb}

\begin{nb}\phantomsection\label{imporem} As far as steady solution $\psi_{\alpha}$ is concerned,  we recall that, according to \cite[Theorem 3.1]{jde2}, $\psi_\alpha$ converges to $\M$ defined by \eqref{M}
as  $\alpha \to 0$. In particular, using the notations $\mathbf{A}_{\alpha}$ instead of $\mathbf{A}_{\psi_{\alpha}}$ and similar notations $\mathbf{B}_{\alpha},\mathbf{a}_{\alpha}$ and $\mathbf{b}_{\alpha}$ we see that\begin{equation*}\begin{split}
\lim_{\alpha\to 0} \mathbf{a}_\alpha & =\mathbf{a}_{0}:= 
\int_{\R^d} \Q_-(\M,\M)(\xi) \, \d\xi 
\; = \; \sqrt{2\pi}\; \frac{|\S^{d-1}|}{|\S^{d}|} \\
\lim_{\alpha\to 0} \mathbf{b}_\alpha & =\mathbf{b}_{0}:= 
\frac{2}{d}\int_{\R^d} \Q_-(\M,\M)(\xi) \, |\xi|^2\, \d\xi 
\;= \;\frac{2d+1}{2d} \;\sqrt{2\pi} \;\frac{|\S^{d-1}|}{|\S^{d}|}\,.
\end{split}\end{equation*}
Hence, 
$$ \lim_{\alpha\to0}\; \frac{2\mathbf{a}_\alpha}{\mathbf{a}_\alpha+\mathbf{b}_\alpha} = \frac{4d}{4d+1}, \qquad \mbox{ and } \qquad  \lim_{\alpha\to0}\;\frac{2\mathbf{b}_\alpha}{\mathbf{a}_\alpha+\mathbf{b}_\alpha} = \frac{4d+2}{4d+1}.$$
Introducing also
{\small $$A_{0}:=-\frac{1}{2}\int_{\R^{d}}\left(d+2-2|\xi|^{2}\right)\Q_{-}(\M,\M)(\xi)\d\xi, \quad
 B_{0}=-\frac{1}{2}\int_{\R^{d}}\left(1-\frac{2}{d}|\xi|^{2}\right)\Q_{-}(\M,\M)(\xi)\d\xi\,,$$}
we see that $$\lim_{\alpha\to0^{+}}\frac{\mathbf{A}_{\alpha}-\alpha\,A_{0}}{\alpha}=0 \qquad \text{ and } \qquad \lim_{\alpha\to0^{+}}\frac{\mathbf{B}_{\alpha}-\alpha\,B_{0}}{\alpha}=0.$$
In particular, $\mathbf{a}_{0}=dB_{0}-A_{0}$ and $\mathbf{b}_{0}=(d+2)B_{0}-A_{0}.$ Notice also that, since $b_{0} \geq a_{0}$, we get $B_{0} > 0$ and $\mathbf{B}_{\alpha} >0$ for $\alpha$ small enough. We will also use repeatedly in the sequel the fact that there exist $C>0$ such that
$$\left|\mathbf{A}_{\alpha}\right| + \left|\mathbf{B}_{\alpha}\right| \leq C\alpha, \qquad \qquad \forall \alpha \in (0,\alpha_{0}),$$
which can be  easily deduced from the fact that $\sup_{\alpha\in (0,\alpha_{0})}\|\psi_{\alpha}\|_{L^{1}_{3}(\R^{d})} < \infty$.\end{nb}

Notice that, by virtue of Theorem \ref{main}, $\tau(t)$ behaves for large time like $\log(1+t)$. 
Of course, the main interest of the above result is that, in order to deduce the rate of convergence to $\bm{f}_{\alpha}$ for the solution $f(t,v)$, it ``suffices'' to prove the rate of convergence to $\psi_{\alpha}$ of the solution $\psi(t,\xi)$. Since Equation \eqref{rescaBE} conserved both mass and kinetic energy, it will be possible to exploit \emph{entropy-entropy production} methods. 

Let us now explicit the first order moments of $f(t,v)$ in terms of quantities involving $\psi(\tau,\xi)$.

\begin{lem}\phantomsection\label{lem:nEtau} Under the assumptions and notations of Proposition \ref{prop:cauc}, it holds
\begin{equation*}
n_{f}(t)=n_{f_{0}}\exp\left(-\alpha\,\int_{0}^{\tau(t)}\mathbf{a}_{\psi}(s)\d s\right),\qquad
T_{f}(t)=T_{f_{0}}\exp\left(-2\int_{0}^{\tau(t)}\mathbf{B}_{\psi}(s)\d s\right),
\qquad \forall t\geq 0.\end{equation*}
In particular, the time scaling $\tau(\cdot)\::\:\R^{+}\to \R^{+}$ is the unique solution with $\tau(0)=0$ to the following differential equation
$$\dfrac{\d}{\d t}\tau(t)=n_{f_{0}}\sqrt{2T_{f_{0}}}\exp\left(-\frac{\alpha}{2}\int_{0}^{\tau(t)}\left(\mathbf{a}_{\psi}(s)+\mathbf{b}_{\psi}(s)\right)\d s\right),\qquad t \geq 0,$$ 
where we recall that $\alpha\,\mathbf{a}_{\psi}(s)=d\mathbf{B}_{\psi}(s)-\mathbf{A}_{\psi}(s)$ while  $\alpha\,\mathbf{b}_{\psi}(s)=(d+2)\mathbf{B}_{\psi}(s)-\mathbf{A}_{\psi}(s)\geq 0$ for any $s \geq 0.$ 
Finally, one has
\begin{align*}
\frac{1}{\sqrt{2T_{f}(t)}}\bm{u}_{f}(t)=&\frac{1}{\sqrt{2T_{f_{0}}}}\exp\left(\int_{0}^{\tau(t)}\mathbf{B}_{\psi}(s)\d s\right)\bm{u}_{f_{0}}\\
&+\int_{0}^{\tau(t)}\mathbf{B}_{\psi}(s)\bm{v}_{\psi}(s)\exp\left(\int_{s}^{\tau(t)}\mathbf{B}_{\psi}(r)\d r\right)\d s, \qquad \forall t \geq 0.\end{align*}
\end{lem}
\begin{proof} The proof resorts on the equation \eqref{eq:ABsca} where the evolution of the moments $n_{f}(t),T_{f}(t)$ and $\bm{u}_{f}(t)$ is related to the definition of $\mathbf{A}_{\psi}(\tau(t)), \mathbf{B}_{\psi}(\tau(t))$ and $\bm{v}_{\psi}(\tau(t))$. Namely, setting for simplicity $\beta(t)=\tfrac{1}{\sqrt{2T_{f}(t)}}$, the first and second identity in \eqref{eq:ABsca} imply that
$$\mathbf{A}_{\psi}(\tau(t))=\frac{\dot{n}_{f}(t)}{n_{f}(t)^{2}}\beta(t) + \frac{d}{n_{f}(t)}\dot{\beta}(t), \qquad \mathbf{B}_{\psi}(\tau(t))=\frac{1}{n_{f}(t)}\dot{\beta}(t).$$
From this, $\frac{\dot{n}_{f}(t)}{n_{f}(t)^{2}}\beta(t)=-\alpha\mathbf{a}_{\psi}(\tau(t))$, and since $\frac{1}{n_{f}(t)}\beta(t)=\frac{1}{n_{f}(t)\sqrt{2T_{f}(t)}}=\frac{1}{\dot{\tau}(t)},$ we get that
$$\log \frac{n_{f}(t)}{n_{f_{0}}}=-\alpha\,\int_{0}^{t}\mathbf{a}_{\psi}(\tau(s))\dot{\tau}(s)\d s=-\alpha\,\int_{0}^{\tau(t)}\mathbf{a}_{\psi}(s)\d s,$$
which gives the desired expression for $n_{f}(t)$.  Similarly, since 
$$\mathbf{B}_{\psi}(\tau(t))=-\dot{T}_{f}(t)n_{f}^{-1}(t)(2T_{f}(t))^{-3/2}, \qquad t \geq 0,$$ 
we easily obtain that
$2\dot{\tau}(t)\mathbf{B}_{\psi}(\tau(t))=-\frac{\dot{T_{f}}(t)}{T_{f}(t)}$, 
which gives the expression of $T_{f}(t).$ Finally, using again that $\dot{\tau}(t)=n_{f}(t)\sqrt{2T_{f}(t)}$ we get the desired differential equation for the time scaling. Introduce now $\bm{z}(t)=\frac{1}{\sqrt{2T_{f}(t)}}\bm{u}_{f}(t)=\beta(t)\bm{u}_{f}(t)$. According to the third identity in \eqref{eq:ABsca},
$$\beta(t)\dot{\bm{u}}_{f}(t)=-\dot{\tau}(t)\mathbf{V}_{\psi}(\tau(t))$$
so that
$$\dot{\bm{z}}(t)=\dot{\beta}(t)\bm{u}_{f}(t)+\beta(t)\dot{\bm{u}}_{f}(t)=\frac{\dot{\beta}(t)}{\beta(t)}\bm{z}(t)-\dot{\tau}(t)\mathbf{V}_{\psi}(\tau(t))=\dot{\tau}(t)\mathbf{B}_{\psi}(\tau(t))\bm{z}(t)-\dot{\tau}(t)\mathbf{V}_{\psi}(\tau(t)),$$
where we used that $\frac{\dot{\beta}(t)}{\beta(t)}=\frac{n_{f}(t)}{\beta(t)}\mathbf{B}_{\psi}(\tau(t))=\dot{\tau}(t)\mathbf{B}_{\psi}(\tau(t))$. Thus,
$$\dfrac{\d}{\d t}\left[\exp\left(-\int_{0}^{\tau(t)}\mathbf{B}_{\psi}(s)\d s\right)\bm{z}(t)\right]=-\dot{\tau}(t)\mathbf{V}_{\psi}(\tau(t))\exp\left(-\int_{0}^{\tau(t)}\mathbf{B}_{\psi}(s)\d s\right)$$
which gives the result.
%
\end{proof}
\begin{nb}
Notice that, since $2\mathbf{B}_{\psi}(s)+\alpha\,\mathbf{a}_{\psi}(s)=\alpha\,\mathbf{b}_{\psi}(s)$ for any $s \geq 0$, we get
$$\int_{\R^{d}}f(t,v)|v-\bm{u}_{f}(t)|^{2}\d v=dn_{f}(t)T_{f}(t)=dn_{f_{0}}T_{f_{0}}\exp\left(-\alpha\int_{0}^{\tau(t)}\mathbf{b}_{\psi}(s)\d s\right), \qquad \forall t \geq 0.$$
\end{nb}

In all the sequel, we shall assume $f_{0} \in L^{1}_{3}(\R^{d})$ to be given and satisfy the assumptions of Proposition \ref{prop:cauc} and will denote by $f(t,v)$ and $\psi(\tau,\xi)$ the associated unique solutions to \eqref{BE} and \eqref{rescaBE} provided by Proposition \ref{prop:cauc}. 

\part{Spectral analysis of the linearized operator}\label{part1}

The scope of this part is to prove Theorem \ref{theo:decayX0}. We  shall consider in the sequel the weight
\begin{equation}\label{eq:weight}
\m(\xi)=\exp(a|\xi|), \qquad a >0.\end{equation}
Inspired by \cite{Tr}, we work on  the following \emph{scales of Banach spaces}:
$$\X_{2} \subset \X_{1} \subset \X_{0}$$
where 
\begin{equation}\label{eq:defX}
\X_{0}=L^{1}(\m), \qquad \X_{1}=\W^{1,1}_{1}(\m), \qquad \X_{2}=\W^{2,1}_{2}(\m).\end{equation}
Recall that the  linearized operator associated to $\mathbb{B}_\alpha$ around the \textit{unique} steady state $\psi_{\alpha}$ has been defined in Definition \ref{defi:linear}. We notice that, for any $\alpha \in (0,\alpha_{0})$
$$\X_{1}=\D(\LL), \qquad \X_{2}=\D(\LL^{2}).$$

\section{Properties of the linearized operators $\LL$ and $\mathscr{L}_{0}$}\label{sec:ll}

We investigate in this section general properties of the linearized operators $\LL$ and $\mathscr{L}_{0}$ in general weighted spaces $\W^{k,1}_{q}(\m)$\footnote{\emph{To avoid too heavy notations, we shall still denote by $\mathscr{L}_{\alpha}$ and $\mathscr{L}_{0}$ the restriction of the above defined operators in the spaces $\X_{1}$ and $\X_{2}$. We adopt the same convention for the associated semigroups and spectral projections in those different spaces. However, one should always keep in mind the underlying space on which one considers such operators.}}. We should keep in mind that we are mainly interested in the properties of the operators in the Banach spaces $\X_{i}$, $i=0,1,2$ and shall restrict ourselves to these spaces at some point.

\subsection{Elastic limit} A crucial role in our analysis will be played by the fact that, in some suitable sense, $\LL$ is close to the elastic linearized operator $\mathscr{L}_{0}$ for $\alpha \simeq 0.$ 
Let us begin with the following lemma.
\begin{lem}\phantomsection\label{prop:psi} There exists some explicit $\overline{a} >0$ such that, for all $k \in \N$, $q \geq 0$, there exists a explicit  function $\eta_{k,q}\::\:(0,\alpha_{0}) \to \R^{+}$ with $\lim_{\alpha\to 0^{+}}\eta_{k,q}(\alpha)=0$ such that
$$\|\psi_{\al} -\M\|_{\W^{k,1}_{q}(\m)} \leq \eta_{k,q}(\alpha) \qquad \forall \alpha \in (0,\alpha_{0}),$$ where the weight function $\m$ is given by $\m(\xi)=\exp(a|\xi|)$, $a \in (0,\overline{a}).$
\end{lem}\phantomsection
\begin{proof} 
According to \cite[Theorem 3.1]{jde2}, for all $k \in \N$, $q \geq 0$,
\begin{equation}\label{eq:lim}
\lim_{\alpha\to0^{+}}\left\|\psi_{\alpha}-\M\right\|_{\W^{k,2}_{q}}=0\end{equation}
with some explicit rate of convergence, while, according to \cite[Corollary 3.3]{jde2},  there is $A >0$ such that
$$\lim_{\alpha\to0^{+}}\left\|\psi_{\alpha}-\M\right\|_{L^{1}_{q}(m_{b})}=0, \quad \forall q \geq 0, b \in [0,A/2)$$
where $m_{b}(\xi)=\exp(b|\xi|)$.  Using the following interpolation inequality (see \cite[Lemma B.1]{MiMo3} where we recall that $\m=m_{a}$)
$$\|f\|_{\W^{k,1}_{q}(\m)} \leq C\,\|f\|_{\W^{8k+7(1+d/2),2}}\,\|f\|_{L^{1}(m_{12a})}^{1/8}\,\|f\|^{3/4}_{L^{1}(\m)}$$
valid for all $f \in\W^{8k+7(1+d/2),2}(\R^{d})  \cap L^{1}(m_{12a})$, 
we deduce easily the conclusion. Notice that the above rate of convergence can be made explicit.\end{proof}

\noindent $\bullet$\textit{\textbf{ From now on, we always assume the weight $\m$ to be given by \eqref{eq:weight} for $a \in (0,\overline{a}).$}} \medskip

On the underlying space $\W^{k,1}_{q}(\m)$, introduce the operator 
$T_{\alpha}\::\:\D(T_{\alpha}) \subset \W^{k,1}_{q}(\m) \to \W^{k,1}_{q}(\m)$ defined by $\D(T_{\alpha})=\W^{k+1,1}_{q+1}(\m)$ and 
$$T_{\alpha}h=-\mathbf{B}_{\alpha}\mathrm{div}(\xi\,  h(\xi)), \qquad h \in \D(T_{\alpha}).$$
One sees that the operator $T_{\alpha}$ is the one responsible for the discrepancy between the domain of $\mathscr{L}_{0}$ and $\LL$. Because of this, we set
$$\P_{\alpha}^{0}\::\:\D(\P_{\alpha}^{0}) \subset \W^{k,1}_{q}(\m) \to \W^{k,1}_{q}(\m)$$
as $\P^{0}_{\alpha}=\mathscr{L}_{0}-\LL + T_{\alpha}$ with domain
$$\D(\P_{\alpha}^{0})=\D(\mathscr{L}_{0})=\W^{k,1}_{1+q}(\m).$$
One has then the following Proposition 
\begin{prop}\phantomsection\label{prop:converLLL0} For any $k \in \N$ and any $q \geq 0$, there exists some explicit function $\ep_{k,q}\::\:(0,\alpha_{0}) \to \R^{+}$ with 
$\lim_{\alpha\to 0^{+}}\ep_{k,q}(\alpha)=0$ and such that
\begin{equation}\label{eq:llXk0}
\|\P_{\alpha}^{0}h\|_{\W^{k,1}_{q}(\m)} \leq \ep_{k,q}(\alpha)\,\|h\|_{\W^{k,1}_{1+q}(\m)} \qquad \forall h \in \W^{k,1}_{1+q}(\m).\end{equation}
As a consequence, 
\begin{equation}\label{eq:llXk}
\|\LL h -\mathscr{L}_{0}h\|_{\W^{k,1}_{q}(\m)} \leq \ep_{k,q}(\alpha)\,\|h\|_{\W^{k+1,1}_{1+q}(\m)} \qquad \forall h \in \W^{k+1,1}_{1+q}(\m).\end{equation}
\end{prop}\phantomsection
\begin{proof} The proof is based upon the well-known estimate for the operators $\Q_{\pm}$ associated to hard-potentials (see Lemma \ref{lem:estimatQ} in Appendix \ref{app:point} for a simple proof): for any $q \geq 0$, there is some universal positive constant $C_{q} >0$ such that
\begin{equation}\label{eq:qpm}
\|\Q_{\pm}(g,f)\|_{L^{1}_{q}(\m)} \leq C_{q}\|g\|_{L^{1}_{q+1}(\m)}\,\|f\|_{L^{1}_{q+1}(\m)}, \qquad \forall f,g \in L^{1}_{q+1}(\m).\end{equation}
Then, since
\begin{multline}\label{eq:LLL0}
\LL h(\xi)-\mathscr{L}_{0}h(\xi)=\Q(h,\psi_{\alpha}-\M)(\xi) + \Q(\psi_{\alpha}-\M,h)(\xi) \\
- \alpha\,\left[\Q_{+}(h,\psi_{\alpha})(\xi)+\Q_{+}(\psi_{\alpha},h)(\xi)\right]
-\alpha\mathbf{a}_{\alpha}h(\xi)-\mathbf{B}_{\alpha}\mathrm{div}(\xi h(\xi))\end{multline}
(where we used that $\alpha\mathbf{a}_{\alpha}=d\mathbf{B}_{\alpha}-\mathbf{A}_{\alpha}$), 
one deduces from \eqref{eq:qpm} that
$$\|\P_{\alpha}^{0}h\|_{L^{1}_{q}(\m)} \leq 2C_{q}\|h\|_{L^{1}_{1+q}(\m)}\,\|\psi_{\alpha}-\M\|_{L^{1}_{1+q}(\m)} 
+2C_{q}\alpha\,\|h\|_{L^{1}_{1+q}(\m)}\,\|\psi_{\alpha}\|_{L^{1}_{1+q}(\m)} +  \alpha\mathbf{a}_{\alpha} \|h\|_{L^{1}_{q}(\m)}.$$
Using the fact that $c_{q}:=\sup_{\alpha \in (0,\alpha_{0})}\|\psi_{\alpha}\|_{L^{1}_{1+q}(\m)} < \infty$ while there exists $\mathbf{a} >0$ such that $\sup_{\alpha \in (0,\alpha_{0})}\mathbf{a}_{\alpha}=\mathbf{a}<\infty$ we get that
$$\|\P_{\alpha}^{0}h\|_{L^{1}_{q}(\m)} \leq \left(2C_{q}\,\eta_{0,1+q}(\alpha) + 2C_{q}\,c_{q}\alpha + \alpha\,\mathbf{a}\right)\|h\|_{L^{1}_{1+q}(\m)}, \qquad \forall h \in L^{1}_{1+q}(\m)$$
where $\eta_{0,1+q}(\alpha)$ is provided by Lemma \ref{prop:psi}. This proves \eqref{eq:llXk0} for $k=0$ with 
$$\ep_{0,q}(\alpha)=\left(2C_{q}\,\eta_{0,1+q}(\alpha) + 2C_{q}\,c_{q}\alpha + \alpha\mathbf{a}\right).$$ 

To prove the result for higher-order derivatives, say for $k=1$, one argues as before using the fact that 
$$\nabla \Q_{\pm}(g,f)=\Q_{\pm}(\nabla g,f) + \Q_{\pm}(g,\nabla f).$$ 
One obtains then easily from \eqref{eq:LLL0} that
\begin{multline*}
\|\P_{\alpha}^{0}h\|_{\W^{1,1}_{q}(\m)} \leq 2C_{q}\|h\|_{\W^{1,1}_{q+1}(\m)}\,\|\psi_{\alpha}-\M\|_{\W^{1,1}_{1+q}(\m)} 
+2C_{q}\alpha \|h\|_{\W^{1,1}_{q+1}(\m)}\,\|\psi_{\alpha}\|_{\W^{1,1}_{1+q}(\m)}\\
+\alpha\mathbf{a}_{\alpha} \|h\|_{W^{1,1}_{q}(\m)}.\end{multline*}
Setting $\ep_{1,q}(\alpha)=\left(2C_{q}\,\eta_{1,1+q}(\alpha) + 2C_{q}\,c_{1,q}\alpha + \mathbf{a}\alpha\right)$ where $\eta_{1,1+q}(\alpha)$ is given in Lemma \ref{prop:psi} and $c_{1,q}=\sup_{\alpha\in (0,\alpha_{0})}\|\psi_{\alpha}\|_{\W^{1,1}_{1+q}(\m)} <\infty,$ we get \eqref{eq:llXk} for $k=1$. The proof for $k > 1$ follows along the same paths. One deduces then \eqref{eq:llXk} from \eqref{eq:llXk0} using the obvious estimate $\|T_{\alpha}h\|_{\W^{k,1}_{q}(\m)} \leq |\mathbf{B}_{\alpha}|\|h\|_{\W^{k+1,1}_{1+q}(\m)}$.\end{proof}

\subsection{Splitting of $\LL$}\label{sec:hypo} Let us now recall the following splitting of $\mathscr{L}_{0}$ introduced in \cite{Mo,Tr}. For any $\delta\in (0,1)$, we consider 
$\Theta_\delta = \Theta_\delta(\xi,\xi_*, \sigma) \in\C^{\infty}(\R^d\times \R^d\times \S^{d-1})$
which is bounded by $1$, which equals $1$ on 
$$J_{\delta}:=\left\{(\xi,\xi_{*},\sigma)\in \R^d\times \R^d\times \S^{d-1},\;\;|\xi|\leq \delta^{-1}\;;\;\: 2\delta \leq |\xi-\xi_*|\leq \delta^{-1}\,;\;\;
 |\cos\theta| \leq 1-2\delta \right\}$$
and whose support is included in $J_{\delta/2}$ (here above $\cos \theta=\langle \frac{\xi-\xi_{*}}{|\xi-\xi_{*}|},\sigma\rangle$). We then set 
\begin{equation}\label{eq:L0Rdelta}\begin{split}
\mathscr{L}_{0}^{S,\delta}h(\xi)& =\int_{\R^d\times \S^{d-1}} 
[\M(\xi'_*)h(\xi') +\M(\xi')h(\xi'_*)-\M(\xi)h(\xi_*)] 
\:|\xi-\xi_*|\,\Theta_\delta (\xi,\xi_{*},\sigma) \d\xi_*\d\sigma  \\ 
\mathscr{L}_{0}^{{R,\delta}}h(\xi)&=  \int_{\R^d\times \S^{d-1}}
[\M(\xi'_*)h(\xi') +\M(\xi')h(\xi'_*)-\M(\xi)h(\xi_*)] 
\,|\xi-\xi_*|\,(1-\Theta_\delta(\xi,\xi,\sigma) ) \d\xi_*\d\sigma
\end{split}\end{equation}
so that 
$$\mathscr{L}_{0}= \mathscr{L}_{0}^{{S,\delta}}+ \mathscr{L}_{0}^{{R,\delta}}-\Sigma_\M$$
where $\Sigma_{\M}$ denotes both the mapping 
$$\Sigma_{\M}(\xi)=\int_{\R^{d}}\M(\xi_{*})|\xi-\xi_{*}|\d\xi_{*}, \qquad \xi \in \R^{d}$$
and the associated multiplication operator. We define then 
\begin{equation*} {\mathcal A}_\delta (h):=  \mathscr{L}_{0}^{S,\delta}(h)\qquad \text{ and } \qquad
{\mathcal B}_{0,\delta} (h) := \mathscr{L}_{0}^{{R,\delta}}-\Sigma_\M
\end{equation*}
so that $\mathscr{L}_{0}=\mathcal{A}_{\delta}+ \mathcal{B}_{0,\delta}$. Let us recall \cite[Lemma 4.16]{GMM}:
\begin{lem}\phantomsection\label{prop:hypo1}For any $k \in \N$ and $\delta >0,$ there are two positive constants $C_{k,\delta} >0$ and $R_{\delta} >0$ such that
$\mathrm{supp}\left(\mathcal{A}_{\delta}f\right)\subset B(0,R_{\delta})$
and
\begin{equation}\label{eq:Adelta}
\|\mathcal{A}_{\delta}f\|_{\W^{k,2}} \leq C_{k,\delta}\|f\|_{L^{1}_{1}}, \qquad \forall  f \in L^{1}_{1}(\R^{d})\end{equation}
\end{lem}
This leads to the following splitting of $\LL$:
$$\mathscr{L}_{\alpha}=\mathcal{B}_{\alpha,\delta} + \mathcal{A}_{\delta}$$
where $\mathcal{B}_{\alpha,\delta}=\mathcal{B}_{0,\delta}+\left[\LL-\mathscr{L}_{0}\right]$. One has the following properties of $\mathcal{B}_{\alpha,\delta}$ (see \cite[Lemma 2.7, 2.8, 2.9]{Tr} for a similar result)
\begin{prop}\phantomsection\label{prop:hypo}
For any $k, q \geq 0,$ there exists $\alpha^{\dagger}_{k,q} >0$, $\delta_{k,q} >0$ and ${\nu}_{k} >0$ such that
$$\mathcal{B}_{\alpha,\delta} + {\nu}_{k} \qquad \text{ is hypo--dissipative in } \: \W^{k,1}_{q}(\m), \qquad \forall \alpha \in (0,\alpha^{\dagger}_{k,q}), \:\:\delta \in (0,\delta_{k,q})$$
with $\D(\mathcal{B}_{\alpha,\delta})=\W^{k+1,1}_{q+1}(\m)$ and
$$\mathcal{B}_{\alpha,\delta}h=\mathcal{B}_{0,\delta}h-\P_{\alpha}^{0}h+T_{\alpha}h\,, \qquad h \in \W^{k+1,1}_{q+1}(\m).$$
\end{prop}\phantomsection
\begin{nb}\phantomsection Notice that, for $k=0$, the hypo--dissipativity of $\mathcal{B}_{\alpha,\delta}$ simply reads
$$\int_{\R^{d}}\mathcal{B}_{\alpha,\delta}f(\xi)\mathrm{sign}f(\xi)\,\langle \xi\rangle^{q}\m(\xi)\d\xi \leq -{\nu}_{0}\|f\|_{L^{1}_{q+1}(\m)}, \qquad \forall f \in L^{1}_{q+1}(\m)$$
which means that, for $k=0$, $\mathcal{B}_{\alpha,\delta}$ is actually dissipative. For $k \geq 1,$ there exists a norm -- denoted by $\llbracket\cdot\rrbracket$ -- which is equivalent to the $\|\cdot\|_{\W^{k,1}_{q}(\m)}$ norm (for which $\llbracket\cdot\rrbracket_{\star}$ denotes the norm on the dual space $\left(\W^{k,1}_{q}(\m)\right)^{\star}$) and such that for all  $f \in \D(\mathcal{B}_{\alpha,\delta}),$  there exists $\bm{u}_{f} \in \left(\W^{k,1}_{q}(\m)\right)^{\star}$ such that
\begin{equation*}
\langle \bm{u}_{f},f\rangle=\llbracket f\rrbracket^{2}=\llbracket\bm{u}_{f}\rrbracket_{\star}^{2}\qquad \text{ and } \quad
\mathrm{Re}\langle \bm{u}_{f},\mathcal{B}_{\alpha,\delta}f\rangle \leq -{\nu}_{k}\llbracket f\rrbracket^{2}
\end{equation*}
where here $\langle\cdot\;,\;\cdot\rangle$ denote the duality bracket between $\left(\W^{k,1}_{q}(\m)\right)^{\star}$ and $\W^{k,1}_{q}(\m)$.
\end{nb}\phantomsection
\begin{proof} Notice that the analysis performed in \cite{MiMo3} and \cite{Tr} (in the spatially inhomogeneous case) proves that, for any $k, q \geq 0,$ there exist $\delta >0$ and ${\nu}_{k} >0$ such that
$$\mathcal{B}_{0,\delta} + {\nu}_{k} \qquad \text{ is hypo--dissipative in } \: \W^{k,1}_{q}(\m).$$
It would be possible to simplify the proof we give using such an estimate. We prefer to give a direct and full proof of the result. Notice that our proof is a technical adaptation of the one given in \cite{Tr}. We first consider the case $k=0$. We write 
${\mathcal B}_{\alpha,\delta} (h)=\sum_{i=1}^{4}C_{i}(h)$ with 
$$C_{1}(h)=-\P_{\alpha}^{0}h,\qquad  C_{2}(h)=\mathscr{L}_{0}^{R,\delta}(h), \quad C_{3}(h)=-\mathbf{B}_{\alpha}\mathrm{div}(\xi\,h(\xi)), \quad C_{4}(h)=-\Sigma_{\M}h$$ and correspondingly,
\begin{equation*}
\int_{\R^d} {\mathcal B}_{\alpha,\delta} (h)(\xi) \, 
\mbox{sign}(h(\xi)) \, \langle\xi\rangle^q \, \m(\xi)\, \d\xi 
 =:\sum_{i=1}^{4}I_{i}(h).\end{equation*}
First, it follows from Proposition \ref{prop:converLLL0} that 
$$I_1(h)\leq \|\P_{\alpha}^{0}h\|_{L^1_{q}(\m)} \leq \ep_{0,q}(\alpha) \|h\|_{L^1_{q+1}( \m)},$$
with $ \displaystyle \lim_{\alpha\to 0^+} \ep_{0,q}(\alpha)=0$. Now, as in \cite[Eq. (2.10)]{Tr}, one has 
$$I_2(h)\leq \|\mathscr{L}_{0}^{R,\delta}(h)\|_{L^1_{q}(\m)}
\leq \tau(\delta) \|h\|_{L^1_{q+1}(\m)},$$
with $\displaystyle \lim_{\delta\to 0} \tau(\delta)=0$. Then, since $h \nabla \mbox{sign}h=0$, one has 
$$I_3(h) = -\mathbf{B}_{\alpha}  \int_{\R^d} \mathrm{div}(\xi |h(\xi)|) 
 \, \langle\xi\rangle^q \, \m(\xi)\, \d\xi  
=  \mathbf{B}_{\alpha}  \int_{\R^d}|h(\xi)|  \xi \cdot \nabla\left(
 \langle\xi\rangle^q \, \m(\xi)\right)\, \d\xi $$
Since $\xi \cdot \nabla\left(
 \langle\xi\rangle^q \, \m(\xi)\right)=q\,|\xi|^{2}\langle \xi\rangle^{q-2}\m(\xi)+a\langle \xi\rangle^{q}|\xi|\m(\xi)$, it is not difficult to see then that there is $C >0$ such that
$$I_3(h) \leq  C \frac{\alpha}{2} \; (\mathbf{a}_{\alpha}+ \mathbf{b}_{\alpha} )
\|h\|_{L^1_{q+1}(\m)} \leq \alpha C\|h\|_{L^{1}_{q+1}(\m)}.$$

Finally, it is well-known that there exists some constants $\sigma_0, \sigma_1>0$ such that, for any $\xi\in\R^d$, 
\begin{equation}\label{nu}
 0<\sigma_0\leq \sigma_0\langle\xi\rangle\leq  \Sigma_\M(\xi)\leq \sigma_1\langle\xi\rangle, 
\end{equation}
which leads to 
$$I_4(h) \leq -\sigma_0  \|h\|_{L^1_{q+1}( \m)}.$$
Gathering the previous estimates, one obtains 
\begin{equation}\label{eq:Bal}
\int_{\R^d} {\mathcal B}_{\alpha,\delta} (h)(\xi) \, 
\mbox{sign}(h(\xi)) \, \langle\xi\rangle^q \, \m(\xi)\, \d\xi 
\leq (\ep_{0,q}(\alpha) + \alpha\,C +\tau(\delta) -\sigma_0)\|h\|_{L^1_{q+1}(\m)}.\end{equation}
Let $\alpha^{\dagger}_{0,q}\in(0,\alpha_{0})$ be such that $\ep_{0,q}(\alpha)+ \alpha C <\sigma_0$ for all $\alpha\in[0,\alpha^{\dagger}_{0,q}).$ We then choose $\delta_{0,q}$ small enough such that, for any $\delta \in (0, \delta_{0,q})$
  $${\nu}_0:= -\left( \tau(\delta) +  \ep_{0,q}(\alpha)+ \alpha C  -\sigma_0 \right) >0 $$
  for all $\alpha\in[0,\alpha^{\dagger}_{0,q})$ and get the result. We now investigate the case $k=1$. We consider the norm 
$$\llbracket h\rrbracket = \|h\|_{L^1_{q}(\m)}+\eta \|\nabla h\|_{L^1_{q}(\m)}, $$
for some $\eta>0$, the value of which shall be fixed later on. This norm is equivalent to the classical $W^{1,1}_{q} (\m)$-norm. We shall prove that for some ${\nu}_1>0$, ${\mathcal B}_{\alpha,\delta}+{\nu}_1$ is dissipative in $W^{1,1}_{q}(\m)$ for the norm $\llbracket \cdot\rrbracket$ and thus hypo-dissipative in  $W^{1,1}_{q}(\m)$.  To this end, we consider 
$$\int_{\R^d} \nabla ({\mathcal B}_{\alpha,\delta} h(\xi))  \cdot  
\mbox{sign}(\nabla h(\xi)) \, \langle\xi\rangle^q \, \m(\xi)\, \d\xi $$
where we used the shorthand notation $\mbox{sign}(\nabla h(\xi))=\left(\mbox{sign}(\partial_{\xi_{1}}h(\xi)),\ldots,\mbox{sign}(\partial_{\xi_{d}}h(\xi))\right).$
First, 
\begin{multline}
\label{eq:nabB}
\nabla ({\mathcal B}_{\alpha,\delta}h)=\nabla ({\mathcal B}_{0,\delta}h) -\nabla (\P_{\alpha}^{0}h) + \nabla (T_{\alpha}h)
=\nabla [\mathscr{L}_{0}^{R,\delta}h -\Sigma_\M(\xi) h ] -\nabla (\P_{\alpha}^{0}h) 
+  \nabla (T_{\alpha}h).\end{multline}
It then follows from Proposition \ref{prop:converLLL0} that 
\begin{equation}\label{eq:nab}
\|\nabla (\P_{\alpha}^{0}h) \|_{L^{1}_{q}(\m)} \leq \ep_{1,q}(\alpha)\|h\|_{\W^{1,1}_{q+1}(\m)}=\ep_{1,q}(\alpha)\|h\|_{L^{1}_{1+q}(\m)}+\ep_{1,q}(\alpha)\|\nabla h\|_{L^{1}_{1+q}(\m)}\end{equation}
with $\lim_{\alpha\to 0^{+}}\ep_{1,q}(\alpha)=0.$ Now, 
$$\nabla [\mathscr{L}_{0}^{{R,\delta}}h -\Sigma_\M(\xi) h ]
= \mathscr{L}_{0}^{R,\delta}(\nabla h) -\Sigma_\M(\xi) \nabla h 
+{\mathcal R} (h) , $$
where 
$${\mathcal R}( h )=\Q(h,\nabla \M)+ \Q(\nabla \M,h) 
- (\nabla{\mathcal A}_\delta)(h) - {\mathcal A}_\delta(\nabla h). $$
Again as in \cite[Eq. (2.10)]{Tr}, one has 
$$\|\mathscr{L}_{0}^{{R,\delta}}(\nabla h)\|_{L^1_{q}(\m)}
\leq \tau(\delta) \|\nabla h\|_{L^1_{q+1}( \m)},$$
with $\displaystyle \lim_{\delta\to 0} \tau(\delta)=0$. 
Then, by \eqref{nu}, 
\begin{equation*}\begin{split}
- \int_{\R^d} \Sigma_\M(\xi) \, \nabla h(\xi) \cdot  
\mbox{sign}(\nabla h(\xi)) \, \langle\xi\rangle^q \, \m(\xi)\, \d\xi
& = - \int_{\R^d} \Sigma_\M(\xi) \,| \nabla h(\xi)|\, \langle\xi\rangle^q \, 
\m(\xi)\, \d\xi \\
& \leq  -\sigma_0  \,\|\nabla h\|_{L^1_{q+1}(\m)}.
\end{split}\end{equation*}
Here above, $| \nabla h(\xi)|=\sum_{i=1}^d |\partial_ih(\xi)|$.
Still,  as in \cite[p. 1942]{Tr}, an integration by parts leads to 
$$\|  (\nabla{\mathcal A}_\delta)(h)\|_{L^1_{q}(\m)} 
+\| {\mathcal A}_\delta(\nabla h)\|_{L^1_{q}(\m)}
\leq   C_\delta \|h\|_{L^1_{q}(\m)}, $$
for some constant $C_\delta>0$.  
Hence,  
$$\| {\mathcal R}( h )\|_{L^1_{q}(\m)}  
\leq  C_\delta \|h\|_{L^1_{q+1}(\m)}. $$
Therefore,
\begin{equation}\label{eq:nabL0}
\|\nabla [\mathscr{L}_{0}^{{R,\delta}}(h) -\Sigma_\M\, h ]\|_{L^{1}_{q}(\m)} \leq C_\delta	\|h\|_{L^{1}_{q+1}(\m)}+\left(\tau(\delta)-\sigma_{0}\right) \|\nabla h\|_{L^1_{q+1}(\m)}\end{equation}
with $\lim_{\delta \to 0^{+}}\tau(\delta)=0.$
Finally, 
\begin{equation}\label{eq:div}\begin{split} 
 \int_{\R^d}  \nabla ( T_{\alpha}h(\xi)) \cdot  
&\mbox{sign}(\nabla h(\xi)) \, \langle\xi\rangle^q \, \m(\xi)\, \d\xi
= - (d+1) \, \mathbf{B}_\alpha \int_{\R^d} |\nabla h(\xi)| \, 
\langle\xi\rangle^q \, \m(\xi)\, \d\xi \\
&   +\mathbf{B}_\alpha \int_{\R^d} |\nabla h(\xi)| \, 
\nabla\cdot (\xi  \langle\xi\rangle^q \, \m(\xi))\, \d\xi 
\leq  \alpha\, C\, \|\nabla h\|_{L^1_{q+1}(\m)}.
\end{split}\end{equation}
Combining \eqref{eq:nabB} with the above estimates \eqref{eq:nab}--\eqref{eq:div}, one obtains  
\begin{multline*}
\int_{\R^d} \nabla ({\mathcal B}_{\alpha,\delta} (h)) \cdot  
\mbox{sign}(\nabla h(\xi)) \, \langle\xi\rangle^q \, \m(\xi)\, \d\xi \\
\leq (C_\delta +\ep_{1,q}(\alpha) )\| h\|_{L^1_{q+1}(\m)} + (\ep_{1,q}(\alpha)+\alpha\,C+\tau(\delta)-\sigma_0) \|\nabla h\|_{L^1_{q+1}(\m)}. 
\end{multline*}
Hence, combining this estimate with \eqref{eq:Bal}
\begin{multline*}
\int_{\R^d} {\mathcal B}_{\alpha,\delta} (h)(\xi) \, 
\mbox{sign}(h(\xi)) \, \langle\xi\rangle^q \, \m(\xi)\, \d\xi 
 +\eta \int_{\R^d} \nabla ({\mathcal B}_{\alpha,\delta} (h)(\xi))  \cdot  
\mbox{sign}(\nabla h(\xi)) \, \langle\xi\rangle^q \, \m(\xi)\, \d\xi \\
\leq  (\ep_{0,q}(\alpha) +  \tau(\delta)+ \alpha C -\sigma_0 
+\eta\, (C_\delta +\ep_{1,q}(\alpha) ) )\|h\|_{L^1_{q+1}(\m)} \\
+ \eta\, (\ep_{1,q}(\alpha)+\alpha\,C+\tau(\delta)-\sigma_0)\|\nabla h\|_{L^1_{q+1}(\m)}.
\end{multline*}
We now choose $\delta >0$ and $\alpha>0$ small enough so that 
$-\lambda:=\ep_{1,q}(\alpha) + \alpha\,C + \tau(\delta)-\sigma_0 < 0$.
Let then $\eta>0$ be small enough such that ${\nu}_0-\eta\, (C_\delta +\ep_{1,q}(\alpha) )>0$. We set 
${\nu}_1:=\min \left\{ {\nu}_0-\eta\, (C_\delta +\ep_{1,q}(\alpha) ), \; \lambda\right\}$ and we finally obtain 
\begin{multline*}
\int_{\R^d} {\mathcal B}_{\alpha,\delta} (h) \, 
\mbox{sign}(h(\xi)) \, \langle\xi\rangle^q \, \m(\xi)\, \d\xi 
 +\eta \int_{\R^d} \nabla ({\mathcal B}_{\alpha,\delta} (h))  \cdot  
\mbox{sign}(\nabla h(\xi)) \, \langle\xi\rangle^q \, \m(\xi)\, \d\xi \\
\leq  -{\nu}_1\left [\|h\|_{L^1_{q+1}(\m)}
+ \eta\,  \|\nabla h\|_{L^1_{q+1}(\m)}\right]\leq -{\nu}_1 \llbracket h \rrbracket, 
\end{multline*}
which means that ${\mathcal B}_{\alpha,\delta}+{\nu}_1$ is hypo-dissipative in $\W^{1,1}_{q}(\m)$. We prove the result for higher order derivatives in the same way.
\end{proof}

\begin{nb}\phantomsection\label{NB:size}
Notice that, for any $\varepsilon >0$ and any $k, q \geq 0,$ a careful reading of the above proof shows that one can chose
${\nu}_{k}=\sigma_{0}-\varepsilon$
up to choosing $\alpha^{\dagger}_{k,q} >0$ and $\delta_{k,q} >0$ small enough.
\end{nb}

\subsection{Properties on the scale of Banach spaces $\X_{i}$, $i=0,1,2$}

Let us from now on restrict ourselves to the scales of Banach spaces $\X_{2} \subset \X_{1} \subset \X_{0}$ introduced earlier. We begin this section by recalling the spectral properties of $\mathscr{L}_{0}$ in the spaces $\X_{i}$, referring to \cite{MiMo3} for details.

\begin{theo}\phantomsection\label{theo:spec} For $i=0,1,2$, the operator $\mathscr{L}_{0}\::\:\D(\mathscr{L}_{0}) \subset \X_{i} \to \X_{i}$ with domain 
$$\D(\mathscr{L}_{0})=\W^{i,1}_{i+1}(\m)$$ is such that $0$ is an eigenvalue of $\mathscr{L}_{0}$ associated to the null set 
$$\mathscr{N}(\mathscr{L}_{0})=\mathrm{Span}(\mathcal{M},\xi_1\mathcal{M},\ldots,\xi_d\mathcal{M},|\xi|^2\mathcal{M}).$$
Moreover $\mathscr{L}_{0}$ admits a positive spectral gap $\mu_\star>0$, i.e. 
$$\mathfrak{S}(\mathscr{L}_{0}) \cap \{\lambda \in \mathbb{C}\;;\;\mathrm{Re}\lambda > -\mu_\star \}=\{0\}$$
and $\mathscr{L}_{0}$ is the generator of a $C_{0}$-semigroup $\{\cS_{0}(t)\;;\;t\geq0\}$ in $\X_{i}$ for which there exists a positive constant $C_{0} >0$ such that
$$\left\|\cS_{0}(t)h-\mathbf{P}_{0}h\right\|_{\X_{i}} \leq C_{0}\exp(-\mu_\star \,t)\|h-\mathbf{P}_{0}h\|_{\X_{i}}, \qquad \forall t \geq 0, \qquad h \in \X_{i}$$
where $\mathbf{P}_{0}$ is the spectral projection of $\mathscr{L}_{0}$ associated to the eigenvalue $\{0\}$.  Moreover,  there exists $n_{0} \in \N$ and $C(n_{0}) >0$ such that 
\begin{equation}\label{eq:Lo}\left\|\mathcal{R}(\lambda,\mathscr{L}_{0})\right\|_{\mathscr{B}(\X_{2})} \leq \frac{C(n_{0})}{|\lambda|^{n_{0}}}, \qquad \forall \mathrm{Re}{\lambda} \geq -\mu_\star.\end{equation}
\end{theo}

\begin{nb}\phantomsection\label{remP0} Notice that the above projection operator $\mathbf{P}_{0}$ does not depend on the space $\X_{i}$ $(i=0,1,2)$, i.e. it acts in the same way in each of the spaces {$\W^{0,1}_{0}(\m),$ $\W^{1,1}_{1}(\m)$ and $\W^{2,1}_{2}(\m)$.} Indeed, setting $M_{0}=\M$, $M_{i}(\xi)=\xi_{i}\,\M(\xi)$ $(i=1,\ldots,d)$ and $M_{d+1}(\xi)=|\xi|^{2}\M(\xi)$, for any $i=0,1,2$ and any $h \in \X_{i}$, one has $\mathbf{P}_{0}h=\sum_{j=0}^{d+1}\eta_{j}(h)M_{j}$
for some $\eta_{j}(h) \in \R.$ Moreover $\mathrm{Range}(\mathbf{I}-\mathbf{P}_{0}) \subset \mathrm{Range}(\mathscr{L}_{0})$ from \cite[equation (6.34), p 180]{kato} so that 
$$\int_{\R^{d}}(\mathbf{I}-\mathbf{P}_{0})h(\xi)\left(\begin{array}{c}1 \\ \xi_{j} \\|\xi|^{2}\end{array}\right)\d \xi=\left(\begin{array}{c}0 \\0 \\0\end{array}\right)\qquad \forall j=1,\ldots,d.$$
 Little algebra, using standard Gaussian computations, allows to determine $\eta_{j}(h)$ and we get easily that
$$\eta_{0}(h)=\int_{\R^{d}}h(\xi)\left(\frac{d+2}{2}-|\xi|^{2}\right)\d\xi, \qquad \eta_{d+1}(h)=\int_{\R^{d}}h(\xi)\left(\frac{2}{d}|\xi|^{2}-1\right)\d\xi, $$
and $\eta_{j}(h)=2\ds\int_{\R^{d}}h(\xi)\xi_{j}\d\xi\,$ for $j=1,\ldots,d.$ 
Notice in particular that, since all the $M_{j}$ are smooth, it holds 
$\mathbf{P}_{0} \in \mathscr{B}(\X_{i},\X_{i+1}),$  for $i=0,1.$
\end{nb}

We have the following result whose proof is differed to Appendix \ref{app:gener}
\begin{prop}\phantomsection\label{prop:gen} For $i=0,1,2$, there exist some explicit $\delta_{*} >0$ and $\alpha^{\dagger} >0$ small enough such that, for all $\alpha \in (0,\alpha^{\dagger})$, $\delta \in (0,\delta_{*})$ the operator
$$\mathcal{B}_{\alpha,\delta} \::\:\D(\mathcal{B}_{\alpha,\delta}) \subset \X_{i} \to \X_{i}$$
is the generator of a $C_{0}$-semigroup $\{\mathcal{U}_{\alpha,\delta}(t)\;;\;t \geq 0\}$. Moreover, there exists ${\nu}_{*} \in (\mu_\star,\sigma_0)$ and $C >0$  such that
$$\left\|\,\mathcal{U}_{\alpha,\delta}(t)\,\right\|_{\mathscr{B}(\X_{i})} \leq C_{i}\exp(-{\nu}_{*}t) \qquad \forall t \geq 0, i=0,1,2.$$
\end{prop}\phantomsection
\begin{nb}\phantomsection With the notations of Proposition \ref{prop:hypo}, one notices simply that $C_{i}$ is a positive constant which relates the usual norm to the modified equivalent norm $\llbracket\cdot\rrbracket$ in $\X_{i}$. Moreover, using Remark \ref{NB:size}, one also sees that for $\alpha,\delta$ small enough, one has ${\nu}_{*}$ arbitrarily close   to $\sigma_{0}$ (with ${\nu}_{*} < \sigma_{0}$).  {By \cite{Mo}, we have $\mu_\star<\sigma_0$. We can thus assume that $\sigma_0 > {\nu}_*>\mu_\star$}.\end{nb}\phantomsection

Notice that, since $\mathcal{A}_{\delta}f$ is compactly supported for any $f \in L^{1}_{1}(\R^{d})$, one can deduce easily from \eqref{eq:Adelta} that $\mathcal{A}_{\delta} \in \mathscr{B}(\W^{k,1}_{q}(\m))$ for any $k,q \geq 0.$ In particular, from the bounded perturbation Theorem, one has the following 
\begin{theo}\phantomsection\label{theo:gen} With the notations of Proposition \ref{prop:gen}, for any $i=0,1,2$ and $\alpha \in (0,\alpha^{\dagger})$ the linearized operator
$$\LL\::\:\D(\LL) \subset \X_{i} \to \X_{i}$$
is the generator of a $C_{0}$-semigroup $\{\mathcal{S}_{\alpha}(t)\,;\,t \geq 0\}$ given by $\mathcal{S}_{\alpha}(t)=\sum_{n=0}^{\infty}\mathcal{V}^{(n)}_{\alpha}(t)$,
where\footnote{Notice that, for each $n \in \N$, the above Dyson-Phillips iterated $\mathcal{V}_{\alpha}^{(n)}(t)$ depends on $\delta.$ We do not explicitly show this dependence to avoid heavy notation.} $\mathcal{V}^{(0)}_{\alpha}(t)=\mathcal{U}_{\alpha,\delta}(t)$ and
$$\mathcal{V}_{\alpha}^{(n+1)}(t)=\int_{0}^{t}\mathcal{V}_{\alpha}^{(n)}(t-s)\mathcal{A}_{\delta}\mathcal{U}_{\alpha,\delta}(s)\d s, \qquad n \in \N, \qquad t \geq 0$$
where $\{\mathcal{U}_{\alpha,\delta}(t)\,;\,t \geq 0\}$ is defined in Proposition \ref{prop:gen} and the above series converges in ${\mathscr{B}}(\X_{i})$ $(i=0,1,2)$.
\end{theo}

For notations convenience, we introduce for any $n \in \N$,
$$\T_{\alpha}^{(n+1)}(t)=\mathcal{A}_{\delta}\mathcal{V}_{\alpha}^{(n)}(t), \qquad \forall t \geq 0.$$
Notice that, with the notations of \cite{GMM,Tr}, $\T_{\alpha}^{(n+1)}(t)=\left(\A_{\delta}\mathcal{U}_{\alpha,\delta}\right)^{*(n+1)}(t)$.

\begin{prop}\phantomsection\label{prop:TnT} Let $i=0,1$ be given. Let $\delta \in (0,\delta_{*})$ and $\alpha \in (0,\alpha^{\dagger})$ be given as in Proposition \ref{prop:gen}. Then, for any $n \in \N$, there is $C_{\delta,n} >0$  such that
\begin{equation*}
\|\mathcal{V}_{\alpha}^{(n)}(t)\|_{\mathscr{B}(\X_{i})} \leq C_{\delta,n}t^{n}\exp\left(-{\nu}_{*}\,t\right) \qquad (t >0),
\end{equation*}
and
\begin{equation*}
\|\mathcal{T}_{\alpha}^{(n+1)}(t)f\|_{\mathscr{B}(\X_{i},\X_{i+1})}\leq C_{\delta,n}\,t^{n}\,\exp\left(-{\nu}_{*}\,t\right) \qquad (t >0).
\end{equation*}  
\end{prop}\phantomsection
\begin{proof} The proof of the first point is easily obtained by induction. Indeed, it holds true for $n=0$ thanks to Proposition \ref{prop:gen}. Assume it holds true for some $n \in \mathbb{N}$. Since $\mathcal{A}_{\delta} \in \mathscr{B}(\X_{i})$ one has
\begin{multline*}
\|\mathcal{V}_{\alpha}^{(n+1)}(t)\|_{\mathscr{B}(\X_{i})} \leq \int_{0}^{t}\left\|\mathcal{V}_{\alpha}^{(n)}(t-s)\A_{\delta}\mathcal{U}_{\alpha,\delta}(s)\right\|_{\mathscr{B}(\X_{i})}\d s\\
\leq C_{\delta,n}\|\A_{\delta}\|_{\mathscr{B}(\X_{i})}\int_{0}^{t}(t-s)^{n}\exp(-{\nu}_{*}(t-s))\|\mathcal{U}_{\alpha,\delta}(s)\|_{\mathscr{B}(\X_{i})}\d s\\
\leq C_{i}C_{\delta,n}\|\A_{\delta}\|_{\mathscr{B}(\X_{i})}\exp(-{\nu}_{*}t)\int_{0}^{t}(t-s)^{n}\d s
\end{multline*}
so the result is true for $\mathcal{V}_{\alpha}^{(n+1)}(t)$ by setting $C_{\delta,n+1}:=\frac{1}{n+1}C_{i}C_{\delta,n}\|\A_{\delta}\|_{\mathscr{B}(\X_{i})}.$ 

Since $\mathcal{A}_{\delta}h$ has compact support for any $h$, we get $\mathcal{A}_{\delta} \in \mathscr{B}(\W^{k,1}_{q}(\m),\W^{k+1,1}_{1+q}(\m))$ and therefore $\|\mathcal{T}_{\alpha}^{(n+1)}(t)\|_{\mathscr{B}(\X_{i},\X_{i+1})} \leq \|\A_{\delta}\|_{\mathscr{B}(\X_{i},\X_{i+1})}\,\|\mathcal{V}_{\alpha}^{(n)}(t)\|_{\mathcal{B}(\X_{i})}$. The  result follows from the first point.
\end{proof}

A simple consequence of this is the following
\begin{lem}\phantomsection\label{lem:laplace} For any $n \in \mathbb{N}$, there exists $C(\delta,n) \geq 0$ such that, for all $i=0,1$ and all $\lambda \in \mathbb{C}$ with $\mathrm{Re}\lambda >-{\nu}_{*}$,  it holds 
\begin{equation}\label{eq:ARB}
\left\|\left[\mathcal{A}_{\delta} \mathcal{R}(\lambda,\mathcal{B}_{\alpha,\delta})\right]^{n}\right\|_{\mathscr{B}(\X_{i},\X_{i+1})}
\leq C(\delta,n)
\left(\mathrm{Re}\lambda+{\nu}_{*}\right)^{-n}.\end{equation}
\end{lem}
\begin{proof} Using the fact that, for $\lambda \in \mathbb{C}$ with $\mathrm{Re}\lambda >-{\nu}_{*}$, $\mathcal{R}(\lambda,\mathcal{B}_{\alpha,\delta})\left[\mathcal{A}_{\delta} \mathcal{R}(\lambda,\mathcal{B}_{\alpha,\delta})\right]^{k-1}$ is the Laplace transform of $\mathcal{V}_{\alpha}^{(k-1)}(t)$ (which is easily checked by induction argument), we obtain that
\begin{multline*}
\left[\mathcal{A}_{\delta} \mathcal{R}(\lambda,\mathcal{B}_{\alpha,\delta})\right]^{k}=\mathcal{A}_{\delta}\mathcal{R}(\lambda,\mathcal{B}_{\alpha,\delta})\left[\mathcal{A}_{\delta} \mathcal{R}(\lambda,\mathcal{B}_{\alpha,\delta})\right]^{k-1}
=\mathcal{A}_{\delta}\int_{0}^{\infty}\exp(-\lambda\,t)\mathcal{V}_{\alpha}^{(k-1)}(t)\d t\\
=\int_{0}^{\infty}\exp(-\lambda\,t)\mathcal{A}_{\delta}\mathcal{V}_{\alpha}^{(k-1)}(t)\d t
=\int_{0}^{\infty}\exp(-\lambda\,t)\mathcal{T}_{\alpha}^{(k)}(t)\d t.\end{multline*}
The result follows then directly from the previous Proposition.\end{proof}

\section{Spectral analysis in  $\X_{1}$}\label{sec:spectral}

\subsection{Spectral properties of $\LL$ in $\X_{1}$} \label{sec:spectral1}

 In all the sequel, we \emph{fix} $\delta \in (0,\delta_{*})$ and simply write
$$\A=\A_{\delta}, \qquad \B_{\alpha}=\B_{\alpha,\delta}, \qquad
\B_{0}=\B_{0,\delta} .$$
We obtain the following whose proof is the same as \cite[Lemma 2.16]{Tr} and is postponed to Appendix \ref{app:prooflemma}:
\begin{lem}\phantomsection\label{lem:inverse} {For all $\lambda \in \mathbb{C}\backslash\{0\}$ with $\mathrm{Re}\lambda > -\mu_\star$} and all $k \in \N$, let
$$\mathcal{J}_{\alpha,k}(\lambda)=\left(\LL -\mathscr{L}_{0}\right)\mathcal{R}(\lambda,\mathscr{L}_{0})\left[\A\,\mathcal{R}(\lambda,\B_{\alpha})\right]^{k}.$$
Then,  {for all ${\nu}_{*}' \in (0,\mu_\star)$}, there exists $r_{k}\::\:(0,\alpha^{\dagger}) \to \R^{+}$ with $\lim_{\alpha\to0^{+}}r_{k}(\alpha)=0$ and such that
\begin{equation}\label{eq:Jalk}
\left\|\mathcal{J}_{\alpha,k}(\lambda)\right\|_{\mathscr{B}(\X_{1})} \leq r_{k}(\alpha), \qquad \forall \lambda \in \bm{\Omega}_{k}(\alpha)\end{equation}
where $\bm{\Omega}_{k}(\alpha)=\{\lambda \in \mathbb{C}\;;\;\mathrm{Re}\lambda >-{\nu}_{*}' \text{ and } |\lambda| > r_{k}(\alpha)\}.$ Moreover, there exists $\underline{\alpha}_{k} \in (0,\alpha^{\dagger})$ such that $\mathbf{Id}-\mathcal{J}_{\alpha,k}(\lambda)$ and $\lambda-\LL$ are invertible in $\X_{1}$ for any $\lambda \in \bm{\Omega}_{k}(\alpha)$, $\alpha \in (0,\underline{\alpha}_{k})$  with
\begin{equation}\label{eq:reso}
\mathcal{R}(\lambda,\LL)=\Gamma_{\alpha,k}(\lambda)(\mathbf{Id}-\mathcal{J}_{\alpha,k}(\lambda))^{-1}, \qquad \lambda \in \bm{\Omega}_{k}(\alpha)\end{equation}
where $\Gamma_{\alpha,k}(\lambda)=\sum_{j=0}^{k-1}\mathcal{R}(\lambda,\B_{\alpha})\left[\A\,\mathcal{R}(\lambda,\B_{\alpha})\right]^{j}+\mathcal{R}(\lambda,\mathscr{L}_{0})\left[\A\,\mathcal{R}(\lambda,\B_{\alpha})\right]^{k}.$ Finally, there exists some positive constant $C_{k} >0$ such that
\begin{equation}\label{eq:estimR}
\|\mathcal{R}(\lambda,\LL)\|_{\mathscr{B}(\X_{1})} \leq \frac{C_{k}}{1-r_{k}(\alpha)}\sum_{j=0}^{k}\frac{1}{({\nu}_{*}-{\nu}_{*}')^{j}}, \qquad \forall \lambda \in \bm{\Omega}_{k}(\alpha), \:\:\alpha \in (0,\underline{\alpha}_{k}).\end{equation}
\end{lem}\phantomsection

{Let us fix ${\nu}_{*}' \in (0,\mu_\star)$ and $k \in \mathbb{N}$. There exists 
$\underline{\alpha}_k^\dagger\in (0,\underline{\alpha}_k)$ such that $r_k(\alpha)\leq{\nu}'_*$ for any $\alpha\in(0,\underline{\alpha}_k^\dagger)$. From the previous result, one gets in particular that,
$$\mathfrak{S}(\LL) \cap \{\lambda\in \mathbb{C}\,;\,\mathrm{Re}\lambda \geq -{\nu}_{*}'\} \subset \{z \in \mathbb{C}\,;\,|z| \leq r_{k}(\alpha)\}, \qquad \forall \alpha \in (0,\underline{\alpha}^\dagger_{k}).$$
We denote then by $\mathbb{P}_{\alpha}$ the spectral projection in $\X_{1}$ associated to the set
$$\mathfrak{S}_{\alpha}:=\mathfrak{S}(\LL) \cap \{\lambda\in \mathbb{C}\,;\,\mathrm{Re}\lambda \geq -{\nu}_{*}'\}=\mathfrak{S}(\LL) \cap  \{z \in \mathbb{C}\,;\,|z| \leq r_{k}(\alpha)\}.$$}
One can deduce then the following whose proof -- similar to that of \cite[Lemma 2.17]{Tr} -- is postponed to Appendix \ref{app:prooflemma}
\begin{lem}\phantomsection\label{lem:PaP0}\phantomsection For any $\alpha$ small enough, $\mathbb{P}_{\alpha} \in \mathscr{B}(\X_{1},\X_{2}).$ Moreover, there exists some explicit   $\ell_{0}\::\:(0,\underline{\alpha}^{\dagger}_k) \to \R^{+}$ such that $\lim_{\alpha\to0^{+}}\ell_{0}(\alpha)=0$ and 
\begin{equation}\label{eq:limproj}\left\|\mathbb{P}_{\alpha}-\mathbf{P}_{0}\right\|_{\mathscr{B}(\X_{1})} \leq \ell_{0}(\alpha).\end{equation}
\end{lem}
From the above result, there exists some explicit $\alpha_{1} \in (0,\alpha_{0})$ such that
$$\|\mathbb{P}_{\alpha}-\mathbf{P}_{0}\|_{\mathscr{B}(\X_{1})} < 1, \qquad \forall \alpha \in (0,\alpha_{1}).$$
According to \cite[Paragraph I.4.6]{kato} (see also \cite[Lemma 2.18]{Tr}), for all $\alpha \in (0,\alpha_{1})$
$$\mathrm{dim}\,\mathrm{Range}(\mathbb{P}_{\alpha})=\mathrm{dim}\,\mathrm{Range}(\mathbf{P}_{0})=d+2$$
where the last identity is deduced from Theorem \ref{theo:spec}. This leads to the following
\begin{prop}\phantomsection \label{cor:mu}
Let us fix ${\nu}_{*}'\in (0,\mu_\star)$.
There is some explicit $\alpha_{1} \in (0,\alpha_{0})$ such that, for all $\alpha \in (0,\alpha_{1})$, the linearized operator $\LL\::\:\D(\LL) \subset \X_{1}\to \X_{1}$ is such that,
$$\mathfrak{S}(\LL) \cap \{z \in \mathbb{C}\;;\,\mathrm{Re}z\geq -{\nu}_{*}'\}=\{\mu_{\alpha}^{1},\ldots,\mu_{\alpha}^{d+2}\} $$ 
where $\mu_{\alpha}^{1},\ldots,\mu_{\alpha}^{d+2}$ are eigenvalues of $\LL$ (not necessarily distinct) with $|\mu_{\alpha}^{j}| \leq r_{k}(\alpha)$ for $j=1,\ldots,d+2$.
\end{prop}

\subsection{Semigroup decay in $\X_{1}$}

Let us now deduce, from the above results, the decay of the semigroup associated to $\LL$ in the space $\X_{1}$. This is done thanks to the following quantitative spectral mapping Theorem which can be deduced from \cite[Theorem 2.1]{MiSc} (see more precisely \cite{MiSc2} for a slight correction on the assumptions).
\begin{theo}\phantomsection {\textbf{\textit{(Quantitative Spectral Mapping Theorem --\cite{MiSc})}}}\label{theo:SMT}
Let $X$ be a given Banach space and let $\Lambda\::\:\D(\Lambda) \subset X \to X$ be the generator of a $C_{0}$-semigroup $\{S_{\Lambda}(t)\,;\,t\geq 0\}$ in $X$. Assume that $\Lambda$ can be split as
$$\Lambda=\A +\B$$
where  $\B$ is the generator of a $C_{0}$-semigroup $\{S_{\B}(t)\,;\,t\geq0\}$ on $X$ and  $\A$ is $\B$-bounded. Assume moreover that
\begin{enumerate}
\item[H1)] There exists $a^{*} \in \R$ such that, for all $a > a^{*}$ and any $\ell \geq 0$, there exists $C=C_{a,\ell} >0$ such that
$$\left\|S_{\B} \ast \left(\A S_{\B}\right)^{(\ast \ell)}(t)\right\|_{\mathscr{B}(X)} \leq C\,\exp(a\,t) \qquad t \geq 0.$$
\item[H2)] There exists $\zeta \in (0,1]$, $s \in [0,\zeta)$ such that $\A \in \mathscr{B}(X_{s},X)$ and there exists $n \geq 1$ such that, for all $a > a^{*}$
$$\left\|\left(\A S_{\B}\right)^{(\ast n)}(t)\right\|_{\mathscr{B}(X,X_{\zeta})} \leq C_{n}\,\exp(a\,t) \qquad t \geq 0$$
for some positive constant $C_{n}$ depending only on $a,n,\zeta$ and $X_{s}$ denotes the abstract Sobolev space associated to $\Lambda$.
\item[H3)] The spectrum $\mathfrak{S}(\Lambda)$ satisfies
$$\mathfrak{S}(\Lambda) \cap \{z \in \mathbb{C}\,;\,\mathrm{Re}z > a^{*}\} \subset \{z \in \mathbb{C}\,;\,\mathrm{Re}z > a'\}$$
for some $a' >a^{*}.$
\end{enumerate}
Then, there exists a projector $\Pi \in \mathscr{B}(X)$ satisfying 
$$\Lambda\,\Pi=\Pi\,\Lambda, \qquad \Lambda_{1}=\Lambda\vert_{X_{1}} \in \mathscr{B}(X_{1}),\qquad \mathfrak{S}(\Lambda_{1}) \subset \{z \in \mathbb{C}\,;\,\mathrm{Re}z > a^{*}\}$$ 
where $X_{1}=\mathrm{Range}(\Pi)$ and, for any $a >a^{*}$, there exists some positive constant $C_{a} >0$ such that
$$\left\|S_{\Lambda}(t)(\mathbf{Id}-\Pi)\right\|_{\mathscr{B}(X)} \leq C_{a}\exp(a\,t), \qquad t \geq0.$$
\end{theo}\phantomsection

We deduce from this the following decay in $\X_{1}$
\begin{prop}\phantomsection\label{prop:decayX1} Let us fix ${\nu}_{*}'\in (0,\mu_\star)$. There exists $\alpha^{\star} \in (0,\alpha_{0})$ such that, for any $\alpha \in (0,\alpha^{\star})$ the $C_{0}$-semigroup $\{\mathcal{S}_{\alpha}(t)\;;\;t \geq 0\}$ in $\X_{1}$ generated by $\LL\::\:\D(\LL) \subset \X_{1} \to \X_{1}$ satisfies, for all $\mu \in (0,{\nu}_{*}')$
$$\left\|\mathcal{S}_{\alpha}(t)\big(\mathbf{Id}-\mathbb{P}_{\alpha}\big)\right\|_{\mathscr{B}(\X_{1})} \leq C_{\mu}\exp(-\mu\,t) \qquad \forall t \geq 0$$
for some positive constant $C_{\mu} >0$.
\end{prop}\phantomsection
\begin{proof} Let $\alpha_{1} \in (0,\alpha_{0})$ be such that Proposition \ref{cor:mu} holds true. Given $\alpha \in (0,\alpha_{1})$, we apply the above Theorem \ref{theo:SMT} with $X=\X_{1}$ and $\Lambda=\LL$. The splitting of $\LL$ has been established in Section \ref{sec:ll}. According to Proposition \ref{cor:mu}, if we set $a^{*}=-{\nu}_{*}'$, one sees that Hypothesis $H3)$ is met with $a^*<a'<\min(\mu_{\alpha}^{1},\ldots,\mu_{\alpha}^{d+2}).$ Notice also that, for all $n \in \N$, $S_{\B} \ast (\A S_{\B})^{(\ast n)}(t)$ is exactly $\mathcal{V}_{\alpha}^{(n+1)}(t)$ so that Assumption $H1)$ is met thanks to Proposition \ref{prop:TnT} since, for all $n \in \mathbb{N}$, $t^{n+1}\exp(-{\nu}_{*}t) \leq C\,\exp(-{\nu}_{*}'t),$ for all $t \geq 0$,
for some positive constant depending only on $n,{\nu}_{*},{\nu}_{*}'$. In the same way, for $\zeta=1$ so that $X_{\zeta}=\X_{2}$ and $s=0$ so that $X_{s}=\X_{1}$ one sees that $H2)$ is met thanks to Proposition \ref{prop:TnT}. This proves that there exists a projector $\mathbf{\Pi}_{\alpha}$ such that, { for all $\mu \in (0,{\nu}_{*}')$}
$$\left\|\mathcal{S}_{\alpha}(t)(\mathbf{Id}-\mathbf{\Pi}_{\alpha})\right\|_{\mathscr{B}(\X_{1})} \leq C_{\mu}\exp(-\mu\,t).$$
As well-known, this implies that the spectrum of the generator $\LL$ satisfies
$$\mathfrak{S}(\LL)=\mathfrak{S}(\LL\vert_{\mathrm{Range}(\mathbf{Id}-\mathbf{\Pi}_{\alpha})}) \cup \mathfrak{S}(\LL\vert_{\mathrm{Range}(\mathbf{\Pi}_{\alpha})})$$
and, since $\mathfrak{S}(\LL\vert_{\mathrm{Range}(\mathbf{\Pi}_{\alpha})}) \subset \{z \in \mathbb{C}\,;\,\mathrm{Re}z > -{\nu}_{*}'\}$ according to Theorem \ref{theo:SMT}, we see that it coincides with $\{\mu_{\alpha}^{1},\ldots,\mu_{\alpha}^{d+2}\}$ and therefore $\mathbf{\Pi}_{\alpha}=\mathbb{P}_{\alpha}.$
\end{proof}

\section{Stability in $\X_{0}$}\label{sec:hypo0}

We still denote here by  $\{\mathcal{S}_{\alpha}(t)\,;\,t \geq 0\}$ the $C_{0}$-semigroup in $\X_{0}$ generated by the linearized operator $\LL$. 
To deduce the decay of the associated semigroup from the above fine properties of the spectrum of $\LL$, we shall resort to the following enlargement result which ensures some suitable quantitative spectral mapping theorem from $\X_{1}$ to $\X_{0}$
\begin{theo}{\textbf{\textit{(Enlargement result -- \cite[Theorem 2.13]{GMM} )}}}\phantomsection\label{theo:GMM}
Let $E$, $\mathcal{E}$ be two Banach spaces with $E \subset \E$ dense with continuous embedding, and consider  $L \in  \mathscr{C}(E)$, $\mathcal{L} \in  \mathscr{C}(\E)$ with $\mathcal{L}\vert_{E} = L$ and $a \in \R$. Assume the following
\begin{enumerate}
\item[A1)] $L$ is the generator of a $C_{0}$-semigroup $\{U(t)\,;\,t\geq 0\}$ in $E$,
$$\mathfrak{S}(L) \cap \{\lambda\;;\;\mathrm{Re}\lambda \geq a\} = \{\xi_{1},\ldots,\xi_{k}\} \subset \mathfrak{S}_{d}(L)$$
and $L-a$ is hypo-dissipative on $\mathrm{Range}(\mathbf{Id}-\mathbf{\Pi}_{L,a})$ where $\mathbf{\Pi}_{L,a}$ is the spectral projection on $E$ associated to the above set of eigenvalues.
\item[A2)] The operator $\mathcal{L}$  can be written as
$$\mathcal{L}=\A + \B$$ 
with $\A, \B \in  \mathscr{C}(\E)$ where $\A \in \mathscr{B}(\E)$ and $\mathcal{L}$ generates a $C_{0}$-semigroup $\{\mathcal{S}(t)\,;\,t\geq 0\}$ in $\E$ and such that\begin{enumerate}
\item $(\B-a)$ is hypo-dissipative on $\E$ while $\A \in \mathscr{B}(\E)$ and $\A\vert_{E} \in \mathscr{B}(E)$;
\item there are constants $n \in \N$, $C_{a} \geq 1$,  such that
$$\left\|\left(\A\,\mathcal{S}\right)^{*n}(t)\right\|_{\mathscr{B}(\E,E)} \leq C_{a}\,\exp(a\,t), \qquad \forall t \geq 0.$$
\end{enumerate}
\end{enumerate}
Then, $\mathcal{L}$ is hypo-dissipative on $\E$ and  there exists some constructive constant $C_{a}' \geq 1$ such that
$$\left\|\mathcal{S}(t)(\mathbf{Id}-\bm{\Pi}_{\mathcal{L},a})\right\|_{\mathscr{B}(\E)} \leq C_{a}'\,t^{n}\,\exp(a\,t), \qquad \forall t \geq 0$$
where $\bm{\Pi}_{\mathcal{L},a}$ is the spectral projector of $\mathcal{L}$ associated to $\{\xi_{1},\ldots,\xi_{k}\}$ in $\E$. 
\end{theo}\phantomsection

We are now in position to prove our main result concerning the linearized operator $\LL$ in $\X_{0}$:
\begin{proof}[Proof of Theorem \ref{theo:decayX0}] We apply Theorem \ref{theo:GMM} with $\E=\X_{0}$, $E=\X_{1}$ and $L=\LL\vert_{\X_{1}}$. The spectral structure of $L$ in $\X_{1}$ is given by Proposition \ref{cor:mu}.  {From Proposition \ref{prop:decayX1}, we deduce that, for any $\mu$ satisfying $-\min(\mu_{\alpha}^{1},\ldots,\mu_{\alpha}^{d+2})<\mu<\nu'_{*}$, the operator $L-\mu$ is hypo-dissipative in $\mathrm{Range}(\mathbf{Id}-\mathbb{P}_{\alpha})$ }(see \cite[Theorem 2.9]{GMM} for the equivalence between hypo-dissipativity and decay of the semigroup). Again Proposition \ref{prop:TnT} shows that Assumption A2) is met and the conclusion follows.
\end{proof}

As already mentioned, it is not clear whether the above Lemma \ref{lem:PaP0} holds true in $\X_{0}$ or not. However, it appears important for our subsequent analysis to obtain suitable norms of $\mathbf{P}_{\alpha}$ in $\X_{0}$ for small values of $\alpha$. This will be done thanks to the following:

\begin{lem} With the notations of Lemma \ref{lem:PaP0}, one has
$\sup_{\alpha \in (0,\underline{\alpha}^{\dagger}_k)}\|\mathbf{P}_{\alpha}\|_{\mathscr{B}(\X_{0},\X_{1})} <\infty.$
\end{lem}
\begin{proof} According to \cite[Theorem 2.1]{GMM}, for any $\alpha \in (0,\underline{\alpha}^{\dagger}_k)$, the restriction of  projection operator  $\mathbf{P}_{\alpha}$ on $\X_{1}$ is exactly $\mathbb{P}_{\alpha}$ and, for all $j=1,\ldots,d+2$, 
$$\mathrm{Ker}(\LL-\mu_{\alpha}^{j})^{m_{j}}=\mathrm{Ker}(L-\mu_{\alpha}^{j})^{m_{j}}, \qquad j=1,\ldots,d+2$$
where $m_{j}$ is the algebraic multiplicity of $\mu_{\alpha}^{j}$ and, as in the proof of Theorem \ref{theo:decayX0}, we set $L=\LL\vert_{\X_{1}}$. In particular, the eigenfunctions of $\LL$ associated to $\mu_{\alpha}^{j}$ belongs to $\X_{1}.$ One gets therefore easily that  $\mathbf{P}_{\alpha} \in \mathscr{B}(\X_{0},\X_{1}).$  Using Lemma \ref{lem:PaP0} we have that, for all $h \in \X_{1}$, $\lim_{\alpha\to0}\|\mathbb{P}_{\alpha}h-\mathbf{P}_{0}h\|_{\X_{1}}=0$ while, according to Remark \ref{remP0}, $\mathbf{P}_{0}\in {\mathscr{B}(\X_{0},\X_{1})}$. Since $\X_{1}$ is dense in $\X_{0}$ and $\mathbf{P}_{\alpha}\vert_{\X_{1}}=\mathbb{P}_{\alpha}$, this implies that $\sup_{\alpha\in (0,\underline{\alpha}^{\dagger}_k)}\|\mathbf{P}_{\alpha}h\|_{\X_{1}}< \infty$ for any $h \in \X_{0}$ and we get the conclusion thanks to Banach-Steinhaus Theorem.\end{proof}

We deduce from this the following
\begin{lem}\phantomsection\label{lem:projnorm} There exists a mapping $\ell_{1}\::\:(0,\underline{\alpha}^{\dagger}_k) \to (0,1)$ with $\lim_{\alpha \to 0}\ell_{1}(\alpha)=0$ and 
\begin{equation}\label{eq:Pa2}
\|(\mathbf{P}_{\alpha}-\mathbf{P}_{0})^{2}\|_{\mathscr{B}(\X_{0})} \leq \ell_{1}(\alpha) \qquad \forall \alpha \in (0,\underline{\alpha}^{\dagger}_k).\end{equation}
In particular, there exists $\alpha^{\star}_{1}$ such that, for all $\alpha \in (0,\alpha^{\star}_{1})$, $(\mathbf{Id}-(\mathbf{P}_{\alpha}-\mathbf{P}_{0})$ is invertible in $\X_0$ and there exists $C >0$  -- independent of $\alpha$ -- such that
$$\|(\mathbf{Id}-(\mathbf{P}_{\alpha}-\mathbf{P}_{0}))^{-1}\|_{\mathscr{B}(\X_{0})} \leq C \qquad \forall \alpha \in (0,\alpha^{\star}_{1}).$$
\end{lem}
\begin{proof} Since, for all $\alpha \in (0,\underline{\alpha}^{\dagger}_k)$, $(\mathbf{P}_{\alpha}-\mathbf{P}_{0})^{2}=\left(\mathbf{P}_{\alpha}-\mathbf{P}_{0}\right)\mathbf{P}_{\alpha} +\left(\mathbf{P}_{0}-\mathbf{P}_{\alpha}\right)\mathbf{P}_{0}$,
and since, for any $h \in \X_{1}$, $\|h\|_{\X_{0}} \leq \|h\|_{\X_{1}}$
one gets
$$\|(\mathbf{P}_{\alpha}-\mathbf{P}_{0})^{2}\|_{\mathscr{B}(\X_{0})} \leq \|\mathbb{P}_{\alpha}-\mathbf{P}_{0}\|_{\mathscr{B}(\X_{1})}\|\mathbf{P}_{\alpha}\|_{\mathscr{B}(\X_{0},\X_{1})} +  \|\mathbb{P}_{\alpha}-\mathbf{P}_{0}\|_{\mathscr{B}(\X_{1})}\|\mathbf{P}_{0}\|_{\mathscr{B}(\X_{0},\X_{1})}$$
and, thanks to the previous Lemma and Lemma \ref{lem:PaP0}, we get \eqref{eq:Pa2} with $\ell_{1}(\alpha)=C\ell_{0}(\alpha)$ where $C=2\,\sup_{\alpha \in [0,\underline{\alpha}^{\dagger}_k)}\|\mathbf{P}_{\alpha}\|_{\mathscr{B}(\X_{0},\X_{1})}<\infty.$ Now, given $\delta \in (0,1)$, there exists $\alpha^{\star}_{1}$ such that $\ell_{1}(\alpha) \leq \delta$ for all $\alpha \in (0,\alpha^{\star}_{1})$. Then, from \eqref{eq:Pa2} and since $\|\mathbf{P}_{\alpha}-\mathbf{P}_{0}\|_{\mathscr{B}(\X_{0})} \leq 2$ for any $\alpha$, we get 
$$\|(\mathbf{P}_{\alpha}-\mathbf{P}_{0})^{n}\|_{\mathscr{B}(\X_{0})} \leq 2\delta^{\frac{n-1}{2}} \qquad \forall n \geq 2.$$
We deduce easily that  $(\mathbf{Id}-(\mathbf{P}_{\alpha}-\mathbf{P}_{0}))$ is invertible with $(\mathbf{Id}-(\mathbf{P}_{\alpha}-\mathbf{P}_{0}))^{-1}=\sum_{n=0}^{\infty}\left(\mathbf{P}_{\alpha}-\mathbf{P}_{0}\right)^{n}$ and $\|(\mathbf{Id}-(\mathbf{P}_{\alpha}-\mathbf{P}_{0}))^{-1}\|_{\mathscr{B}(\X_{0})} \leq 1+2+2\sum_{n=2}^{\infty}\delta^{\frac{n-1}{2}} \leq 3+\frac{\sqrt{\delta}}{1-\sqrt{\delta}}.$
This proves the result.
\end{proof}

\part{Stability analysis}\label{part2}

We establish here the main results concerning the long-time behavior of the solution to \eqref{rescaBE} that we recall here for the reader convenience:
\begin{multline*}
\partial_{t}\psi(t,\xi) +\left(\mathbf{A}_{\psi}(t)-d\mathbf{B}_{\psi}(t)\right)\,\psi(t,\xi) + \mathbf{B}_{\psi}(t)\mathrm{div}_{\xi}\big(\left({\xi}-\bm{v}_{\psi}(t)\right)\psi(t,\xi))\\
=(1-\alpha)\Q(\psi,\psi)(t,\xi)-\alpha\Q_{-}(\psi,\psi)(t,\xi)\end{multline*}
where we recall that $\psi(t,\xi)$ is obtained from the original solution $f(t,v)$ to \eqref{BE} through the scaling \eqref{scalingPsi}. As already said, our approach combines the entropy production method with the spectral analysis performed in the First Part. In all this part, $\psi(t,\xi)$ will be the unique solution to \eqref{rescaBE} obtained through the scaling \eqref{scalingPsi} in Proposition \ref{prop:cauc}.

\section{Entropy production method}\label{sec:entropy}

Introduce the time-dependent relative entropy
\begin{equation}\label{Ht}
\mathcal{H}(t)=\mathcal{H}\left(\psi(t)\,|\,\M\right):=\IR \psi(t,\xi)\log\left(\frac{\psi(t,\xi)}{\M(\xi)}\right)\d \xi, \qquad t \geq 0
\end{equation}
where we recall that $\M$ denotes the Maxwellian distribution with same mass, momentum and kinetic energy of $\psi(t,\cdot)$ and $\psi_{\alpha}$, that is,
\begin{equation*}\label{maxwellianM0}
\M(\xi)=\pi^{-d/2}\exp\left(-|\xi|^{2}\right), \qquad \xi \in \R^{d}.\end{equation*}

We also introduce the entropy production functional associated to the \emph{elastic} Boltzmann operator
\begin{equation}\label{D1}
\D_{0}(t)=-\IR \Q\big(\psi,\psi\big)(t,\xi)\log\left(\frac{\psi(t,\xi)}{\M(\xi)}\right)\d \xi\,.
\end{equation}
\begin{lem}\phantomsection\phantomsection\label{lem:dht} The evolution of $\mathcal{H}(t)$ is given by the following
\begin{equation}\label{dHt}
\dfrac{\d}{\d t}\mathcal{H}(t)+(1-\alpha)\D_{0}(t)=\left(d\mathbf{B}_{\psi}(t)-\mathbf{A}_{\psi}(t)\right)\mathcal{H}(t) + \mathcal{I}_1(t)  \,,\qquad \forall t \geq 0
\end{equation}
with
\begin{equation}\label{I1}
\mathcal{I}_{1}(t):=-\alpha \IR\Q_-\big(\psi,\psi\big)(t,\xi) \log\left(\frac{\psi(t,\xi)}{\M(\xi)}\right)\d \xi.
\end{equation}
\end{lem}\phantomsection
\begin{proof} The proof consists simply in multiplying \eqref{rescaBE} by $\log\left(\frac{\psi(t,\xi)}{\M(\xi)}\right)$ and integrating over $\R^{d}$.  This leads to
$$\dfrac{\d}{\d t}\mathcal{H}(t)+(1-\alpha)\D_{0}(t)=\left(d\mathbf{B}_{\psi}(t)-\mathbf{A}_{\psi}(t)\right)\mathcal{H}(t) + \mathcal{I}_1(t)+\mathbf{B}_{\psi}(t)\mathcal{I}_2(t)$$where
$$\mathcal{I}_{2}(t)=-\IR \nabla_{\xi}\cdot\big((\xi-\bm{v}_{\psi}(t))\,\psi(t,\xi) \big)\log\left(\frac{\psi(t,\xi)}{\M(\xi)}\right)\d \xi\,.$$
One checks, integrating by parts, that $\mathcal{I}_{2}(t)=0$ since $\int_{\R^{d}}|\xi|^{2}\psi(t,\xi)\d\xi=\frac{d}{2}=\frac{d}{2}\int_{\R^{d}}\psi(t,\xi)\d\xi$.  This shows the result.
\end{proof}

In order to estimate the term $\mathcal{I}_{1}(t)$ we need the propagation of the $3^{rd}$ moment and some $L^{p}$ Lebesgue norm.  We refer to Appendix \ref{app:point}. for a discussion on propagation and creation of moments and the proof of the following result (also, see \cite[Theorem 1.6 and Remark 1.7]{jde}).
\begin{lem}\phantomsection \label{lem:Lp} For any $\eta \geq 0$, there exists some explicit $\alpha^{\star}_{\eta} \in (0,1)$ such that for all $p \in (1,\infty)$ and any $\alpha \in (0,\alpha^{\star}_{\eta})$, if $f_{0} \in L^{1}_{3}(\R^{d})\cap L^{1}_{\eta+\frac{d-2}{1-\theta}}(\R^{d}) \cap L^{p}_{\eta}(\R^{d})$ with 
$$ \theta=\left\{\begin{array}{cc} \frac{1}{d} & \mbox{ if } p\in(1,2],\\
\\
\frac{d(p-2)+1}{d(p-1)}  & \mbox{ if } p\in [2,\infty),\end{array}\right.$$
then \begin{equation}\label{eq:boundLp}
\sup_{t \geq 0}\|\psi(t)\|_{L^p_{\eta}(\R^{d})} \leq  \max\left\{\|\psi_0\|_{L^p_\eta(\R^{d})},C_{p,\eta}(\psi_0)\right\}\end{equation}
for an explicit constant $C_{p,\eta}(\psi_0)>0$ depending only on  $p$, $d$, $\|\psi_{0}\|_{L^{1}_{3}(\R^{d})}$ and $\|\psi_{0}\|_{L^{1}_{\eta+\frac{d-2}{1-\theta}}(\R^{d})}$ but not on $\alpha$.
\end{lem} 

\begin{nb}\phantomsection Notice that the bound obtained in \cite[Theorem 1.6 \& Remark 1.7]{jde} actually depends on $\alpha$. However, a careful reading of the proof shows that it depends on $\alpha$ only through the parameter  $\mu_{\alpha}$ such that $\inf_{t\geq 0}\int_{\R^{d}}\psi(t,\xi)|\xi-\xi_{*}|\d\xi_{*}\geq \mu_{\alpha}\langle \xi\rangle$. 
Using that the upper bound on, say, the third-order moment of $\psi(t)$ is independent of $\alpha$, Lemma \ref{appB.2}, we deduce from \cite[Lemma 2.1]{alogam} that $\mu_{\alpha}$ is actually independent of $\alpha$, i.e. there exists $\kappa_{0} >0$ such that
\begin{equation}\label{eq:kapp0}
\int_{\R^{d}}\psi(t,\xi)|\xi-\xi_{*}|\d\xi_{*}\geq \kappa_{0}\langle \xi\rangle, \qquad \forall \xi \in \R^{d},\quad t \geq 0\end{equation}
and the bound in \eqref{eq:boundLp} turns out to be uniform with respect to $\alpha$. Notice also that the proof in \cite{jde} is done for $\eta=0$. It is straightforward to extend it to $\eta >0$. We provide in the Appendix a full proof in the case $p=2$, which is the one we use in the sequel.\end{nb}
We have all in hands to estimate the term  $\mathcal{I}_{1}(t)$   defined in Lemma \ref{lem:dht}. 
\begin{lem}\phantomsection\phantomsection
\label{lem:I1} Let $f_{0} \in L^{1}_{3}(\R^{d}) \cap L^{p}(\R^{d})$
for some $p >1$. Then, there exists a positive constant $C$ depending only on $\|\psi_{0}\|_{L^p(\R^{d})}$ and $\|\psi_{0}\|_{L^{1}_{3}(\R^{d})}$ such that, for all  $0 < \alpha < \min(\alpha_{\star},\alpha^{\star}_0)$, 
\begin{equation}\label{estimI1}
\left|\mathcal{I}_1(t)\right|  \leq C\,\alpha \qquad \qquad t \geq 0,
\end{equation}
\end{lem}\phantomsection
\begin{proof}  Clearly, there is some positive constant $C_{d}$ depending only on $d$ such that
$$\left|\int_{\R^{d}}\Q_{-}(\psi,\psi)(t,\xi)\log\frac{\psi(t,\xi)}{\M(\xi)}\d\xi\right| \leq \int_{\R^{d}}\Q_{-}(\psi,\psi)(t,\xi)\left|\log \psi(t,\xi)\right|\d\xi  + C_{d}\int_{\R^{d}}\Q_{-}(\psi,\psi)(t,\xi)(1+|\xi|^{2})\d\xi,$$
so that
$$\left|\mathcal{I}_{1}(t)\right| \leq \alpha\, 
\|\psi(t)\|_{L^{1}_{1}(\R^{d})}\int_{\R^{d}}\langle\xi\rangle\,\psi(t,\xi)\left|\log \psi(t,\xi)\right|\d\xi + \alpha C_{d}\|\psi(t)\|_{L^{1}_{1}(\R^{d})} \|\psi(t)\|_{L^{1}_{3}(\R^{d})}.$$
{
Now,
$$\int_{\R^{d}}\langle\xi\rangle\,\psi(t,\xi)\left|\log \psi(t,\xi)\right|\d\xi
= \int_{|\psi|< 1}\langle\xi\rangle\,\psi(t,\xi)\left|\log \psi(t,\xi)\right|\d\xi+ \int_{|\psi|\geq 1}\langle\xi\rangle\,\psi(t,\xi)\left|\log \psi(t,\xi)\right|\d\xi $$
On the one hand, setting  $C_p=\sup_{r\geq 1} r^{1-p}|\log r|^{2}$ , we deduce that 
\begin{eqnarray*}
 \int_{|\psi|\geq 1}\langle\xi\rangle\,\psi(t,\xi)\left|\log \psi(t,\xi)\right|\d\xi
&\leq &  \sqrt{1+\frac{d}{2}} \left(\int_{|\psi|\geq 1} \psi(t,\xi)\left|\log \psi(t,\xi)\right|^2\d\xi\right)^{1/2}\\
&\leq &   \sqrt{C_p} \; \sqrt{1+\frac{d}{2}}\; \|\psi(t)\|^{p/2}_{L^p(\R^d)}
\end{eqnarray*}
On the other hand, for $\beta\in(0,1)$, setting $D_\beta=\sup_{r\in(0,1)} r^\beta|\log r|$, we have 
\begin{eqnarray*}
\int_{|\psi|< 1}\langle\xi\rangle\,\psi(t,\xi)\left|\log \psi(t,\xi)\right|\d\xi
&\leq & D_\beta \int_{\R^{d}} \langle\xi\rangle^{2(1-\beta)-(1-2\beta)}\,\psi(t,\xi)^{1-\beta}\d\xi\\
&\leq & D_\beta  \; \left(1+\frac{d}{2}\right)^{1-\beta}\; \left(\int_{\R^{d}} \langle\xi\rangle^{-(1-2\beta)/\beta} \d\xi\right)^\beta
\end{eqnarray*}
The choice $\beta=\frac{1}{d+3}$, together with propagation of the third moment and Lemma \ref{lem:Lp} yield the result.
}
\end{proof}
The following technical lemma, refer to Appendix \ref{app:point} for a proof, proves the appearance of gaussian-like pointwise lower bound.  Because of the use we make such lower bound later, a precise estimate on the time rate appearance is needed.
\begin{theo}{\textbf{\textit{(Gaussian-like lower bound)}}}\phantomsection \label{a-l3}
Let  $\psi_{0} \in L^{1}_{3}(\R^{d}) \cap L^{p}(\R^{d})$
for some $p >1$. Let $0 < \alpha < \min(\alpha_{\star},\alpha_0^{\star})$ be given. Then, for any  $t_0 \in (0,1)$ and any $\varepsilon >0$ there exist some explicit constant $c_{0}(\alpha)$ and some integer $N \in \mathbb{N}$ depending on $\varepsilon,$ $\|\psi_{0}\|_{L^{1}_{3}}$ and $\|\psi_{0}\|_{L^{p}}$ and $\alpha$ (but not on $t_{0}$) such that  
\begin{equation}\label{eq:pointlower}
\psi(t,\xi)\geq c_{0}(\alpha)t_{1}^{N}\exp\left(-c_{0}(\alpha)\left(1 + \log\left(\tfrac{1}{t_{0}}\right)\right)\,|\xi|^{2+\varepsilon}\right)\,,\qquad t \geq t_{0}\,,\quad \xi \in \R^{d}.
\end{equation}
\end{theo}

\begin{nb}
It readily follows from Theorem \ref{a-l3} that, for any $\varepsilon>0$, 
$$|\log \psi(t,\xi)|\leq C_\varepsilon (1+\log^+(1/t))\langle\xi \rangle^{2+\varepsilon} +\psi(t,\xi), \qquad \xi\in\R^d,t>0$$
for some universal constant $C_\varepsilon>0$. Indeed, for $\psi(t,\xi)\geq1$, we clearly have $\log\psi(t,\xi)\leq \psi(t,\xi)$ whereas for $\psi(t,\xi)<1$, \eqref{eq:pointlower} ensures that 
$$-\log\psi(t,\xi)\leq C_\varepsilon (1+\log^+(1/t))\langle\xi \rangle^{2+\varepsilon}.$$
\end{nb}
 
\begin{theo}\phantomsection\phantomsection\label{prop:decayHt}
Assume that $f_{0} \in L^{1}_{3}(\R^{d}) \cap L^{2}(\R^{d})$.  Given $0 < \alpha < \min(\alpha_{\star},\alpha_0^{\star})$ where $\alpha_{\star}$ and $\alpha_0^{\star}$ are defined respectively in Theorem \ref{main} and Lemma \ref{lem:Lp}, the unique solution $\psi(t,\xi)$ to \eqref{rescaBE} satisfies, for all $t_{0} > 0$
\begin{equation}\label{eq:estimH}
\mathcal{H}(t)=\mathcal{H}(\psi(t)|\M) \leq C_{H}\left((1+t)^{-1/2}+\alpha^{1/3}\right) \qquad \forall\, t > t_{0}\end{equation}
where the positive constant $C_{H}$ depends explicitly on $t_0$, $\sup_{t\geq t_{0}}\|\psi(t)\|_{L^{1}_{s_{0}}(\R^{d})}$ (with $s_{0}\geq2$ large but explicit) and $\|\psi_{0}\|_{L^2(\R^{d})}$.
\end{theo}\phantomsection
\begin{proof} Since $f_{0} \in L^{1}_{3}(\R^{d}) \cap L^{2}(\R^{d})$, for all $t_{0} >0$, according to Proposition \ref{a-l3}, there exists $c_{0} >0$ such \eqref{eq:pointlower} holds true for, say, $\varepsilon=1$. Then, according to Proposition \ref{theo:Vill}, for all $t \geq t_{0}$, one has
$$\D_{0}(t) \geq \lambda(\psi(t))\,\mathcal{H}^{3}(t)$$
for $\lambda(\psi(t))$ depending only on $c_{0}$, $\|\psi(t)\|_{L^2}$ and $\|\psi(t)\|_{L^{1}_{s_{0}}}$ for some explicit $s_{0} > 0$ related to $c_{0}$.  By virtue of the creation of moments Lemma \ref{appB.2} and from Lemma \ref{lem:Lp},
$$\sup_{t\geq t_0}\|\psi(t)\|_{L^{1}_{s_0}} \leq C_{s_{0}}(t_{0}), \qquad \sup_{t \geq 0}\|\psi(t)\|_{L^2} \leq C_{L^2}$$
for $C_{2} >0$ depending only on $\|\psi_{0}\|_{L^2}$.  In other words, $\inf_{t \geq t_{0}}\lambda(\psi(t)) \geq \lambda_{0}$ for some positive $\lambda_{0}$ depending only on $c_{0}$, $C_{s_{0}}(t_{0})$, and $\|\psi_{0}\|_{L^2}$. This shows that
$$\dfrac{\d}{\d t}\mathcal{H}(t) +(1-\alpha)\lambda_{0}\mathcal{H}^{3}(t) \leq \left(d\mathbf{B}_{\psi}(t)-\mathbf{A}_{\psi}(t)\right)\mathcal{H}(t) +\mathcal{I}_{1}(t)  \qquad \forall t \geq t_{0}$$
which, thanks to Lemma \ref{lem:I1}, yields
$$\dfrac{\d}{\d t}\mathcal{H}(t) +(1-\alpha)\lambda_{0}\mathcal{H}^{3}(t) \leq \left(d\mathbf{B}_{\psi}(t)-\mathbf{A}_{\psi}(t)\right)\mathcal{H}(t) + C\alpha,\qquad \forall t \geq t_{0}.$$
Because, 
$$d\mathbf{B}_{\psi}(t)-\mathbf{A}_{\psi}(t)=\alpha\,\int_{\R^{d}}\Q_{-}(\psi,\psi)(t,\xi)\d\xi \leq \alpha\,\|\psi(t)\|^2_{L^{1}_{1}(\R^{d})}\leq C_{1}\alpha$$
for some $C_{1} >0$ depending only on $\|\psi_{0}\|_{L^{1}_{1}(\R^{d})}$ while $\mathcal{H}(t) \leq C_{3}\left(\|\psi(t)\|_{2}+\|\psi(t)\|_{L^{1}_{2}(\R^{d})}\right)$ for some positive constant $C_{3}$, we obtain, using Lemma \ref{lem:Lp} again, the following inequality satisfies by $\mathcal{H}(t)$,
$$\dfrac{\d}{\d t}\mathcal{H}(t) + \lambda_{0}\,(1-\alpha_{\star}) \mathcal{H}^{3}(t) \leq C\,\alpha$$
for some positive constant $C$ depending only on  $\|\psi_{0}\|_{L^2}$ and $\|\psi_{0}\|_{L^{1}_{s_{0}}(\R^{d})}$. Integration of this inequality yields the desired result.\end{proof}

\section{Stability result}\label{sec:stabi}

We begin this section collecting uniform estimates for $\psi(t)$ in the scale of Banach spaces $\X_{0},\X_{1}$.

\subsection{Uniform bounds on $\X_{0}$ and $\X_{1}$}

In all this section, we shall assume that the initial datum $f_0$ is nonnegative, with positive mass and temperature, and such that
$$f_{0} \in L^{1}_{3}(\R^{d}) \cap L^{2}(\R^{d}).$$
By \eqref{scalingPsi}, we have $\psi_0(\xi)=\frac{(2T_{f_{0}})^{d/2}}{n_{f_{0}}}\;f_{0}\left(\sqrt{2T_{f_{0}}}\, \xi+\mathbf{u}_{f_{0}}\right)$, so that $\psi_{0} \in L^{1}_{3}(\R^{d}) \cap L^{2}(\R^{d})$. 

The following result shows the appearance of exponential moments for the solutions to \eqref{rescaBE}. We refer to  Lemma \ref{appB.2} and subsequent discussion in the Appendix for a proof.
\begin{theo}\phantomsection\phantomsection \label{theo:tail}  For any $\alpha \in (0,\alpha_{\star})$, let $\psi(t,\xi)$ be the unique solution to \eqref{rescaBE} with initial datum $\psi_0$. Let $\beta>1$. Then, there exists $A >0$, $C >0$ explicit and depending on $\beta$, $d$ and $\int_{\R^d}\psi_0(\xi)|\xi|^3\,\d\xi$ such that
$$\int_{\R^{d}}\psi(t,\xi)\exp\left(a\min\{1,t^{\beta}\}|\xi|\right)\d\xi\leq C,\qquad \forall\, a \in (0,A).$$
In particular, for any $\alpha \in (0,\alpha_{\star})$,
$$\psi(t) \in \X_{0}\,,\quad \text{with}\quad \|\psi(t)\|_{\X_{0}}\leq C\,, \qquad \forall\, t \geq 1,$$
with exponential weight
\begin{equation}\label{Eq:varpi}
\m(\xi)=\exp(a|\xi|), \qquad 0 < a < A.\end{equation}
\end{theo}\phantomsection

It is more intricate to derive uniform bounds on the solution $\psi(t,\xi)$ in the weighted Sobolev space $W^{1,1}_{1}(\m)$.  Using the estimates on $\Q_{+}$ in weighted spaces provided in \cite[Section 4]{alogam}, it would be simple to prove the propagation of $H^{1}(\m)$ norms.  We adopt here a \emph{new} viewpoint which is based on the propagation of \emph{Fisher information} and relies on the pointwise lower bounds (Theorem \ref{a-l3}). We recall that Fisher information has been defined in \eqref{eq:fisher}.
\begin{theo}{\textit{\textbf{(Uniform bound on the Fisher information)}}}\phantomsection\label{theo:fisher} Assume, in additon, that
$$f_{0} \in L^{1}_{\eta}(\R^{d}) \cap H^{\frac{(5-d)^{+}}{2}}_{\eta}(\R^{d}) \qquad \text{ and } \qquad I(f_{0}) < \infty,$$
for some $\eta > 4+d/2$, $d \geq 2$.  Then, the unique solution $\psi(t)$ of \eqref{rescaBE} satisfies
$$\sup_{t\geq0}\,I(\psi(t)) \leq C$$
for some positive constant $C$ depending on $I(\psi_{0})$ and the $L^{1}_{\eta}\cap H^{\frac{(5-d)^{+}}{2}}_{\eta}$-norm of $\psi_{0}.$
\end{theo}
\begin{proof} Let us multiply Equation \eqref{rescaBE} by $\frac{1}{2\,\sqrt{\psi(t,\xi)}}$ to get
\begin{align*}
\partial_{t}\sqrt{\psi(t,\xi)} + \frac{1}{2}\mathbf{A}_{\psi}(t)\sqrt{\psi(t,\xi)} &+ \mathbf{B}_{\psi}(t)(\xi-\bm{v}_{\psi}(t))\cdot\nabla\sqrt{\psi(t,\xi)}\\
&=\dfrac{1-\alpha}{2\sqrt{\psi(t,\xi)}}\Q_{+}(\psi,\psi)(t,\xi) - \dfrac{1}{2}\sqrt{\psi(t,\xi)}R(\psi)(t,\xi)\,,
\end{align*}
where $R(\psi)(t,\xi)=\int_{\R^{d}}\psi(t,\xi_{*})|\xi-\xi_{*}|\d\xi_{*}.$
Now, given $i=1,\ldots,d$, let us define $g(t,\xi):=\partial_{\xi_{i}}\sqrt{\psi(t,\xi)},$ so that $\partial_{\xi_{i}}\big((\xi-\bm{v}_{\psi}(t))\cdot\nabla\sqrt{\psi}\big) = (\xi-\bm{v}_{\psi}(t))\cdot\nabla g + g$.  Then, $g(t,\xi)$ satisfies 
\begin{multline*}
\partial_{t}g(t,\xi) + \tfrac{1}{2}\big(\mathbf{A}_{\psi}(t)+2\mathbf{B}_{\psi}\big)g(t,\xi) + \mathbf{B}_{\psi}(t)(\xi-\bm{v}_{\psi}(t))\cdot\nabla g(t,\xi) =\\
=\frac{1-\alpha}{2}\bigg[\frac{\partial_{\xi_{i}}\Q_{+}(\psi,\psi)(t,\xi)}{\sqrt{\psi(t,\xi)}} - \frac{g(t,\xi)\,\Q_{+}(\psi,\psi)(t,\xi)}{\psi(t,\xi)}\bigg]\\
 - \frac{1}{2}g(t,\xi)\,R(\psi)(t,\xi) - \frac{1}{2}\sqrt{\psi(t,\xi)}\,\partial_{\xi_{i}}R(\psi)(t,\xi)\,.
\end{multline*}  
Multiplying this equation  by $g(t,\xi)$ and integrating over $\R^{d}$, it follows that
\begin{multline*}
\frac{1}{2}\frac{\d}{\d {t}}\|g(t)\|_{L^{2}}^{2}+ \frac{1}{2}\big(\mathbf{A}_{\psi}(t)+2\mathbf{B}_{\psi}(t)\big)\|g(t)\|_{L^{2}}^{2} + \mathbf{B}_{\psi}(t)\int_{\mathbb{R}^{d}}\big((\xi-\bm{v}_{\psi}(t))\cdot\nabla g(t,\xi)\big)g(t,\xi) \,\d\xi \\
 =\frac{1-\alpha}{2}\bigg[\int_{\mathbb{R}^{d}} \frac{g\,\partial_{\xi_{i}}\Q_{+}(\psi,\psi)}{\sqrt{\psi}}\,\d\xi - \int_{\mathbb{R}^{d}}\frac{g^{2}\,\Q_{+}(\psi,\psi)}{\psi}\,\d\xi \bigg] \\
-\frac{1}{2}\int_{\mathbb{R}^{d}} g^{2}\,R(\psi)\,\d\xi - \frac{1}{2}\int_{\mathbb{R}^{d}} g\sqrt{\psi}\,\partial_{\xi_{i}}R(\psi)\,\d\xi\,.
\end{multline*}
Now, integration by parts leads to $\int_{\mathbb{R}^{d}}((\xi-\bm{v}_{\psi}(t))\cdot\nabla g(t,\xi))g(t,\xi) \,\d\xi= -\frac{d}{2}\|g(t)\|_{L^{2}}^{2}$ and
\begin{equation*}
\int_{\mathbb{R}^{d}} \frac{g\,\partial_{\xi_{i}}\Q_{+}(\psi,\psi)}{\sqrt{\psi}}\d\xi= \frac{1}{2}\int_{\mathbb{R}^{d}} \big(\partial_{\xi_{i}}\log\psi\big)\partial_{\xi_{i}}\Q_{+}(\psi,\psi)\d\xi= -\frac{1}{2}\int_{\mathbb{R}^{d}} \log\psi\, \partial^{2}_{\xi_{i}}\Q_{+}(\psi,\psi)\,\d\xi\,.\end{equation*}
Moreover, as already observed (see \eqref{eq:kapp0}), $R(\psi)(t,\xi)\geq \kappa_{0}\langle \xi \rangle$ for some positive $\kappa_{0}$ so that
\begin{multline*}
\frac{1}{2}\frac{\d}{\d t}\|g(t)\|_{L^{2}}^{2}+ \frac{1}{2}\Big(\mathbf{A}_{\psi}(t)+(2-d)\mathbf{B}_{\psi}(t)\Big)\|g(t)\|_{L^{2}}^{2} +\frac{\kappa_{0}}{2}\|g(t)\|_{L^{2}}^{2}\\
\leq -\frac{1-\alpha}{4}\int_{\mathbb{R}^{d}} \log\psi(t,\xi)\,\partial^{2}_{\xi_{i}}Q_{+}(\psi,\psi)(t,\xi)\,\d\xi\\
  - \frac{1}{2}\int_{\mathbb{R}^{d}} g(t,\xi)\sqrt{\psi(t,\xi)}\,\partial_{\xi_{i}}R(\psi)(t,\xi)\,\d\xi\,.
\end{multline*}
Moreover, we have that $|\partial_{\xi_{i}}R(\psi)(t,\xi)|\leq \int_{\R^{d}}\psi(t,\xi)\d\xi=1$.  Thus, we get
\begin{equation*}
\left|\int_{\mathbb{R}^{d}} g(t,\xi)\sqrt{\psi(t,\xi)}\,\partial_{\xi_{i}}R(\psi)(t,\xi)\,\d\xi\right|\leq \|g(t)\|_{L^{2}}\,.
\end{equation*}
Furthermore, by the instantaneous appearance of an exponential lower bound (see Theorem \ref{a-l3} and the remark afterwards), we have for any $\varepsilon >0$
\begin{equation*}
|\log \psi(t,\xi)| \leq c_{\varepsilon}(t)\langle \xi \rangle^{2+\varepsilon} + \psi(t,\xi),\qquad \xi \in \R^{d}, \:\:t >0
\end{equation*}
where $c_{\varepsilon}(t) =  C_{\varepsilon}(1+\log^{+}(1/t))$ for some universal constant $C_{\varepsilon} >0$. Thus, using first Cauchy Schwarz inequality we get
\begin{equation*}\begin{split}
\int_{\mathbb{R}^{d}} \log\psi(t,\xi)\,\partial^{2}_{\xi_{i}}\Q_{+}&(\psi,\psi)(t,\xi)\,\d\xi \leq \int_{\mathbb{R}^{d}} \big(c_{\varepsilon}(t)\langle \xi \rangle^{2+\varepsilon} + \psi(t,\xi)\big)\Big|\partial^{2}_{\xi_{i}}\Q_{+}(\psi,\psi)(t,\xi)\Big|\,\d\xi \\
&\leq C_{d,\varepsilon}\Big(c_{\varepsilon}(t) + \|\psi(t)\|_{L^{2}}\Big)\Big\|\langle \xi \rangle^{2+3\varepsilon/2+d/2}\,\partial^{2}_{\xi_{i}}\Q_{+}(\psi,\psi)(t,\xi)\Big\|_{L^{2}} \qquad \forall\, t >0\end{split}\end{equation*}
for some universal positive constant depending on $d,\varepsilon$. Now, using Theorem \ref{regularite} we can estimate the last term as
\begin{equation*} 
\Big\|\langle \xi \rangle^{2+3\varepsilon/2+d/2}\,\partial^{2}_{\xi_{i}}\Q_{+}(\psi,\psi)(t,\xi)\Big\|_{L^{2}}\leq \left\|\Q_{+}(\psi,\psi)(t)\right\|_{H^{2}_{2+3\varepsilon/2+d/2}}\leq C\Big(\|\psi\|^{2}_{H^{s}_{\eta_{1}}} + \|\psi(t)\|^{2}_{L^{1}_{\eta_{2}}}\Big)\,
\end{equation*}
with
$$\eta_{1}:=\frac{8+d+3\varepsilon}{2}, \qquad \eta_{2}:=\frac{6+3\varepsilon+d}{2}, \qquad s=2 - \frac{(5-d)^{+}}{2}.$$
Therefore, using the uniform estimates on the $H^{s}_{\eta_{1}}(\R^{d})$ and moments we obtain that, for a suitable choice of $\varepsilon >0$ small enough and $\alpha \in (0,\alpha_0^{*})$, it holds
\begin{equation*}
\frac{\d }{\d t}\|g(t)\|_{L^{2}}^{2}  + \frac{\kappa_{0}}{2}\|g(t)\|_{L^{2}}^{2}  \leq C(1+\log^{+}(1/t))\|g(t)\|_{L^{2}}\, \qquad t >0,
\end{equation*}
or, equivalently
\begin{equation*} \frac{\d }{\d t}\bm{y}(t) + \frac{\kappa_{0}}{4}\bm{y}(t)  \leq \frac{C}{2}(1+\log^{+}(1/t)), \qquad t >0,
\qquad \text{ with } \quad\bm{y}(t):=\|g(t)\|_{L^{2}}\,. 
\end{equation*}
Using that the mapping $t\mapsto 1+\log^{+}(1/t)$ is integrable at $t=0$, simple integration of this differential inequality implies that $\sup_{t\geq0}\bm{y}(t)\leq C_{1}\bm{y}(0)+C_2<\infty$ for some explicit constants $C_{1}$ and $C_2.$ This proves the result.
\end{proof}
\begin{cor}\phantomsection\label{cor:fisher} Under the assumptions of Theorem \ref{theo:fisher}, the unique solution $\psi(t)$ to \eqref{rescaBE} satisfies the estimate
$\sup_{t \geq 1}\|\psi(t)\|_{\X_{1}} < \infty$,
where we recall that $\X_{1}=W^{1,1}_{1}(\m)$, with weight $\m$ having rate $a<A/2$. \end{cor}
\begin{proof} Using Cauchy-Schwarz inequality
$$\int_{\R^{d}}|\nabla \psi(t,\xi)|\exp(a/2|\xi|)\d\xi \leq 2\sqrt{I(\psi(t)}\left(\int_{\R^{d}}\psi(t,\xi)\exp(a|\xi|)\d\xi\right)^{1/2}\leq C\,,\quad\text{for any}\quad a<A\,,\quad t\geq1\,.$$
The boundedness in the last inequality is concluded thanks to Theorem \ref{theo:fisher} and Theorem \ref{theo:tail}.\end{proof}
\subsection{Stability estimate}\label{sec:stabi7}
Using the Csisz\'ar-Kullback inequality (see \cite[Theorem A.2, p. 131]{jungel}), we deduce from Theorem \ref{prop:decayHt} the following result.
{\begin{cor}\phantomsection\phantomsection \label{cor:Lal}
Assume that $0<f_{0} \in L^{1}_{3}(\R^{d})\cap L^{2}(\R^{d})$ and that $0<a<\min\{\overline{a},A/2\}$ where $\overline{a}$ and $A$ are defined respectively in Lemma \ref{prop:psi} and Theorem \ref{theo:tail}.  There exists some explicit function $\ell\::\:(0,\alpha^{\ddagger}] \to \mathbb{R}^{+}$ with $\lim_{\alpha \to 0}\ell(\alpha)=0$ and some constant $C>0$ both depending on  the $L^{1}_{3}\cap L^{2}$-norm of $\psi_{0}$ such that, for any $\alpha \in (0,\alpha^{\ddagger})$ the solution $\psi(t,\xi)$ to \eqref{rescaBE} satisfies
$$\left\|\psi(t)-\psi_{\alpha}\right\|_{L^{1}(m_a)} \leq \ell(\alpha)+C (1+t)^{-1/4} \qquad \forall\, t \geq 1.$$ 
where $m_{a}(\xi):=\exp(a|\xi|)$. 
\end{cor}}\phantomsection 
\begin{proof} Using both Csisz\'ar-Kullback and Cauchy-Schwarz inequalities, we obtain, for all $m_{a}(\xi):=\exp(a|\xi|)$
  \begin{align*}
    \|\psi(t)-\M\|_{L^{1}(m_{a})}^{2}
    &\leq \|\psi(t)-\M\|_{L^{1}(\R^{d})}\,\|\psi(t)-\M\|_{L^{1}(m_{2a})}
    \\
    &\leq \sqrt{2\,H(\psi(t)|\M)}\,\|\psi(t)-\M\|_{L^{1}(m_{2a})}.
  \end{align*}
Due to Theorem \ref{theo:tail}, choosing $a< A/2$, gives $\sup_{t \geq 1} \|\psi(t)-\M\|_{L^{1}(m_{2a})} < \infty$.  One deduces from \eqref{eq:estimH} that there exists some constant $C>0$ such that
$$\|\psi(t)-\M\|_{L^{1}(m_a)} \leq C \left((1+t)^{-1/4}+ \alpha^{1/6}\right)\,, \qquad \forall\, t \geq 1\,.$$
Using that $\|\psi(t)-\psi_{\alpha}\|_{L^{1}(m_a)} \leq \|\psi(t)-\M\|_{L^{1}(m_a)}+\|\M-\psi_{\alpha}\|_{L^{1}(m_a)}$, we obtain the conclusion invoking Lemma \ref{prop:psi}.
\end{proof}
Let us move to a perturbative setting. Set $h(t,\xi) := \psi(t,\xi) - \psi_{\alpha}(\xi)$, so that,
\begin{equation*}\begin{split}
    \partial_t h(t,\xi) &= \mathscr{L}_{\alpha}h(t,\xi)
    + \mathbb{B}_{\alpha}(h,h)(t,\xi) + \left[\mathbf{A}_{\alpha}-\mathbf{A}_{\psi}(t)\right]\psi(t,\xi)\\
    &\phantom{+++} + \left[\mathbf{B}_{\alpha}-\mathbf{B}_{\psi}(t)\right]\xi \cdot \nabla_{\xi} \psi(t,\xi) + \mathbf{B}_{\psi}(t)\bm{v}_{\psi}(t) \cdot \nabla_{\xi}\psi(t,\xi).
    \end{split}
  \end{equation*}
As already mentioned, defining $\bm{v}_{\alpha}$ through
$\mathbf{B}_{\alpha}\bm{v}_{\alpha}:=-\alpha\int_{\R^{d}}\xi\,\Q_{-}(\psi_{\alpha},\psi_{\alpha})(\xi)\d\xi$,
one sees that $\bm{v}_{\alpha}$ is equal to zero since $\Q_{-}(\psi_{\alpha},\psi_{\alpha})$ is radially symmetric. Therefore, one can rewrite the evolution as
  \begin{equation}\label{Eq:dth}\begin{split}
    \partial_t h(t,\xi) &= \mathscr{L}_{\alpha}h(t,\xi)
    + \mathbb{B}_{\alpha}(h,h)(t,\xi) + \left[\mathbf{A}_{\alpha}-\mathbf{A}_{\psi}(t)\right]\psi(t,\xi)\\
    &\phantom{+++} + \left[\mathbf{B}_{\alpha}-\mathbf{B}_{\psi}(t)\right]\xi \cdot \nabla_{\xi} \psi(t,\xi) + \left[\mathbf{B}_{\psi}(t)\bm{v}_{\psi}(t)-\mathbf{B}_{\alpha}\bm{v}_{\alpha}\right] \cdot \nabla_{\xi}\psi(t,\xi).\end{split}\end{equation}
Moreover, for any $t\geq1$ we have that $h(t,\xi) \in \X_{0}$, and 
 $$\int_{\R^{d}}h(t,\xi)\left(\begin{array}{c}1 \\ \xi \\|\xi|^{2}\end{array}\right)\d\xi=\left(\begin{array}{c}0 \\ 0 \\0\end{array}\right)\,,\qquad \forall\,t\geq0.$$
Let us introduce, for any $t \geq 0$
\begin{align*}
\mathbf{G}_{\alpha}(t,\xi)=\mathbb{B}_{\alpha}(h,h)&(t,\xi) + \left[\mathbf{A}_{\alpha}-\mathbf{A}_{\psi}(t)\right]\psi(t,\xi)\\
& + \left[\mathbf{B}_{\alpha}-\mathbf{B}_{\psi}(t)\right]\xi \cdot \nabla_{\xi} \psi(t,\xi) + \left[\mathbf{B}_{\psi}(t)\bm{v}_{\psi}(t)-\mathbf{B}_{\alpha}\bm{v}_{\alpha}\right] \cdot \nabla_{\xi}\psi(t,\xi).\end{align*}
As a consequence, using Duhamel's formula, where we recall that $\{\mathcal{S}_{\alpha}(t)\,;\,t\geq 0\}$ denotes the $C_{0}$-semigroup generated by $\mathscr{L}_{\alpha}$ in $\X_{0}$, we can write
{\begin{equation}\label{eq:duham}
h(t) = \mathcal{S}_\alpha(t-t_0)h(t_{0}) + \int_{t_{0}}^t \mathcal{S}_\alpha({t-s}) \mathbf{G}_\alpha(s) \,\d s, \qquad \forall\, t   \geq t_{0} > 0\,.\end{equation}}
\begin{lem}\label{eq:Ga}\phantomsection Assume the conditions of Theorem \ref{theo:fisher} for $f_0>0$ and take $a\in\left(0,\frac{1}{2}\min\{\overline{a},A/2\}\right)$ and any $t_{0} \geq  1$.  Then,

\begin{equation*}
\|\mathbf{G}_{\alpha}(s)\|_{\X_{0}} \leq C\|h(s)\|_{\X_{0}} \left(\alpha + \ell(\alpha)^{1/2}+(1+s)^{-1/8}\right) \qquad \forall\, s \geq t_{0}.\end{equation*}
The constant $C>0$ depends on $L^{1}_{\eta}\cap H^{(5-d)^{+}/2}_{\eta}$-norm of $\psi_{0}$, with $\eta>4+d/2$.
\end{lem}\phantomsection
\begin{proof} 
Denote by $C >0$ a constant that may depend on $L^{1}_{\eta}\cap H^{(5-d)^{+}/2}_{\eta}$-norm of $\psi_{0}$, with $\eta>4+d/2$, and that can change from line to line.  Using Lemma \ref{lem:estimatQ} in Appendix \ref{app:point}
$$\|\mathbb{B}_{\alpha}(h(s),h(s))\|_{\X_{0}}  \leq C\|h(s)\|_{L^{1}_{2}(\m)}\|h(s)\|_{L^{1}(\m)} \qquad \qquad \forall s\,\geq t_{0}.$$
{Moreover, using Cauchy-Schwarz inequality together with Corollary \ref{cor:Lal}, one sees that
$$\|h(s)\|_{L^{1}_{2}(\m)} =\|h(s)\|_{L^{1}_{2}(m_a)}\leq C\|h(s)\|_{L^1(m_{2a})}^{1/2} \leq C\left(\ell(\alpha)^{1/2} +(1+s)^{-1/8}\right)\qquad \forall s \geq t_{0},$$
where $m_{a}(\xi):=\exp(a|\xi|)$.
}
Therefore,
$$\|\mathbb{B}_{\alpha}(h(s),h(s))\|_{\X_{0}} \leq C\,\left(\ell(\alpha)^{1/2} +(1+s)^{-1/8}\right)\|h(s)\|_{\X_{0}} \qquad \forall s \geq t_{0}.$$
It is easy to check that 
\begin{equation*}\begin{split}
\left|\mathbf{A}_{\alpha}-\mathbf{A}_{\psi}(s)\right| &\leq \frac{\alpha(d+2)}{2}\,\left[\|\Q_{-}(h(s),\psi_{\alpha})\|_{L^{1}_{2}(\R^{d})} + \|\Q_{-}(\psi(s),h(s))\|_{L^{1}_{2}(\R^{d})}\right]\\
&\leq C\alpha\,\|h(s)\|_{L^{1}_{3}(\R^{d})}.
\end{split}\end{equation*}
In the same way, 
\begin{equation}\label{conv}
\left|\mathbf{B}_{\alpha}-\mathbf{B}_{\psi}(s)\right| + \left|\mathbf{B}_{\alpha}\bm{v}_{\alpha}-\mathbf{B}_{\psi}(s)\bm{v}_{\psi}(s)\right|\leq C\alpha\,\|h(s)\|_{L^{1}_{3}(\R^{d})}\,.
\end{equation}
Consequently,
$$\|\mathbf{G}_{\alpha}(s)\|_{\X_{0}} \leq C\|h(s)\|_{\X_{0}}\left(\ell(\alpha)^{1/2}+(1+s)^{-1/8}+\alpha\,\|\psi(s)\|_{\X_{0}} + \alpha \|\psi(s)\|_{\X_{1}}\right)\,,\qquad s\geq t_{0}\geq1\,.$$
Moreover, under our assumption on $f_{0}$, by Theorem \ref{theo:tail} and Corollary \ref{cor:fisher}, $\sup_{s\geq1}\|\psi(s)\|_{\X_{0}} \leq C$ and $\sup_{s\geq1}\|\psi(s)\|_{\X_{1}} \leq C$.\end{proof} 

The following lemma is crucial to the argument.  We use the notations of Lemma \ref{lem:projnorm}.
\begin{lem}\phantomsection\label{lem:I-P}\phantomsection
There exists some constant $C_{0} >0$ such that
$$ \|h(t)\|_{\X_{0}} \leq C_{0}\|(\mathbf{I}-\mathbf{P}_{\alpha})h(t)\|_{\X_{0}}, \qquad \forall\, t \geq 1, \qquad \forall \alpha \in (0,\alpha^{\star}_{1}),$$
where $\alpha^{\star}_{1}$ is defined in Lemma \ref{lem:projnorm}.
\end{lem}\phantomsection
\begin{proof} Since $\mathbf{P}_{0}h(t)=0$ for all $t > 0$, one has
$$\mathbf{P}_{\alpha}h(t)=(\mathbf{P}_{\alpha}-\mathbf{P}_{0})h(t)\,\quad\text{ and }\quad g(t):=(\mathbf{Id}-\mathbf{P}_{\alpha})h(t)=(\mathbf{Id}-(\mathbf{P}_{\alpha}-\mathbf{P}_{0}))h(t)\,.$$ Since $\mathbf{Id}-(\mathbf{P}_{\alpha}-\mathbf{P}_{0})$ is invertible for any $\alpha \in (0,\alpha^{\star}_{1})$ we get from Lemma \ref{lem:projnorm} that there exists $C_{0} >0$ independent of $\alpha$ such that
$$\|h(t)\|_{\X_{0}} \leq \|(\mathbf{Id}-(\mathbf{P}_{\alpha}-\mathbf{P}_{0}))^{-1}\|_{\mathscr{B}(\X_{0})}\|g(t)\|_{\X_{0}} \leq C_{0}\|g(t)\|_{\X_{0}}\,,\qquad \forall\,t\geq1\,,$$
for any $\alpha \in (0,\alpha^{\star}_{1}).$ This proves the result.\end{proof}
\begin{proof}[Proof of Theorem \ref{theo:main-rescaled}]
For any $\alpha \in (0,\alpha^{\star}_{1})$, introduce, as in aforementioned proof ,
$$g(t)=(\mathbf{I}-\mathbf{P}_{\alpha})h(t),\quad \forall\, t > 0.$$
Applying $(\mathbf{I}-\mathbf{P}_{\alpha})$ to Duhamel's formula \eqref{eq:duham} and using that $\mathbf{P}_{\alpha}$ commutes with $\mathcal{S}_{\alpha}(t)$ we get
$$g(t)=\mathcal{S}_\alpha(t-t_0)g(t_{0})
    + \int_{t_{0}}^t \mathcal{S}_\alpha({t-s})\left(\mathbf{I}-\mathbf{P}_{\alpha}\right)
    \mathbf{G}_\alpha(s) \,\d s, \qquad \forall\, t \geq t_{0}.$$
Using Theorem \ref{theo:decayX0} and Lemma \ref{eq:Ga},  for all $\mu \in (0,{\nu}_{*}') $ and any $t \geq t_{0}$ we have that
\begin{equation*}\begin{split}\|g(t)\|_{\X_{0}} &\leq C_{ {\mu}}\exp(- {\mu}(t-t_{0})\|h(t_{0})\|_{\X_{0}}  +C_{{\mu}}\int_{t_{0}}^{t}\exp(-{\mu}(t-s))\|G_{\alpha}(s)\|_{\X_{0}}\d s\\
&\leq C_{\mu}\exp(-\mu(t-t_{0}))\|h(t_{0})\|_{\X_{0}} + C_{\mu}\int_{t_{0}}^{t} \left(\alpha + \ell(\alpha)^{1/2}+(1+s)^{-1/8}\right)\exp(-\mu(t-s))\|h(s)\|_{\X_{0}}\d s. \end{split}\end{equation*}
Using Lemma \ref{lem:I-P}, this translates into
$$\|g(t)\|_{\X_{0}} \leq C_{\mu}\exp(-\mu(t-t_{0}))\|g(t_{0})\|_{\X_{0}}+ C_{0}\,C_{\mu}\int_{t_{0}}^t\left(\alpha + \ell(\alpha)^{1/2}+(1+s)^{-1/8}\right)\exp(-\mu(t-s))\|g(s)\|_{\X_{0}}\d s.$$
Thanks to Gronwall's Lemma, we obtain
$$\|g(t)\|_{\X_{0}} \leq C_{1}\,\exp(-\mu t) \exp\left( C_{0}\,C_{\mu}\left(\alpha t + \ell(\alpha)^{1/2}t + \frac{8}{7}(1+t)^{7/8}\right)   \right)\|h(t_{0})\|_{\X_{0}}, \qquad \forall\, t \geq t_{0}.$$
But, one has 
$(1+t)^{7/8} \leq \chi (1+t) +C_\chi$, for $\chi>0$. Hence,  
$$\|g(t)\|_{\X_{0}} \leq C\,\exp(-\mu_{\alpha}t)\|h(t_{0})\|_{\X_{0}}, \qquad \forall\, t \geq t_{0}$$
with $\mu_{\alpha}=\mu-C_{0}C_{\mu}\left(\alpha + \ell(\alpha)^{1/2} +\chi\right).$ Recall from theorem \ref{theo:decayX0} that $\mu$ may be chosen arbitrarily close to $\mu_\star$. Consequently, $\mu_{\alpha}$ may be chosen arbitrarily close to $\mu_\star$ for $\alpha$ small enough. Using again Lemma \ref{lem:I-P}, this gives
$$\|h(t)\|_{\X_{0}} \leq C\exp(-\mu_{\alpha}t)\|h(t_{0})\|_{\X_{0}} \qquad \forall\, t \geq t_{0},$$
achieving the proof.
\end{proof}

\section{Back to the original variable}\label{sec:original}

Let us now explain how the above convergence result can be translated in the original variable. Recall that, from \eqref{scalingPsi}, the link between the original unknown $f(t,v)$ and the rescaled function $\psi(\tau,\xi)$ is given by 
\begin{equation*}
f(t,v)=n_{f}(t)(2T_{f}(t))^{-d/2}\psi\left(\tau(t),\frac{v-\bm{u}_{f}(t)}{\sqrt{2T_{f}(t)}}\right)
\end{equation*} 
where $n_{f}(t),\bm{u}_{f}(t)$ and $T_{f}(t)$ denote respectively the mass, momentum and  temperature of $f(t,\cdot).$ Then, one obtains the following version of Theorem \ref{main:no-scaled}.
\begin{prop}\phantomsection\label{prop:main-noscal}
Under the assumption and notations of Theorem \ref{main:no-scaled}, for any $\varepsilon >0$ there exist some explicit $\alpha_{c} \in (0,\alpha_{0})$ and $C_{\varepsilon} >0$ such that, for any $\alpha \in (0,\alpha_{c})$
$$\int_{\R^{d}}\left|f(t,v)- \bm{f}_{\alpha}(t,v)\right|\exp\left(a\,\frac{|v-\bm{u}_{f}(t)|}{\sqrt{2T_{f}(t)}}\right)
\d v \leq  {C_{\varepsilon}}{n_{f}(t)}\exp\left(-(\mu_{\star}-\varepsilon)\tau(t)\right), \qquad \tau(t) \geq 1\,,$$
where
$$\bm{f}_{\alpha}(t,v)=n_{f}(t)(2T_{f}(t))^{-d/2}\psi_{\alpha}\left(\frac{v-\bm{u}_{f}(t)}{\sqrt{2T_{f}(t)}}\right).$$
\end{prop}
For Proposition \ref{prop:main-noscal} to be operant, we need to have a better understanding of the behavior, as $t \to \infty$, of the quantities $n_{f}(t),\bm{u}_{f}(t),T_{f}(t)$.  We mentioned in Section \ref{sec:moments} that this seems a difficult task, yet, we can profit from the exponential convergence of $\psi(\tau,\xi)$ towards $\psi_{\alpha}$ to obtain estimates for the long-time behavior of these macroscopic quantities.
\begin{lem}\phantomsection\label{lem:ABv}
With the notations of Theorem \ref{main:no-scaled},  for any $\varepsilon >0$ there exist some explicit $\alpha_{c} \in (0,\alpha_{0})$ and $C>0$ depending only on $\varepsilon$ and $f_{0}$ such that, for any $\alpha \in (0,\alpha_{c})$
\begin{equation*}
\big|\mathbf{A}_{\psi}(\tau) - \mathbf{A}_{\alpha}\big| + \big| \mathbf{B}_{\psi}(\tau) - \mathbf{B}_{\alpha}\big| + \big|\mathbf{B}_{\psi}(\tau)\,\bm{v}_{\psi}(\tau)\big|\leq C\,\alpha\,\exp(-(\mu_{\star}-\varepsilon)\tau) \qquad \forall\, \tau \geq 1.
\end{equation*}
where we recall that $\mathbf{A}_{\psi}(\tau), \mathbf{B}_{\psi}(\tau)$ and $\bm{v}_{\psi}(\tau)$ are defined in \eqref{eq:BAV}.\end{lem}
\begin{proof} The result was almost established in Lemma \ref{eq:Ga}. Namely, it was proved there that, for all $\tau \geq 0$
$$\big|\mathbf{A}_{\psi}(\tau) - \mathbf{A}_{\alpha}\big| + \left|\mathbf{B}_{\alpha}-\mathbf{B}_{\psi}(\tau)\right| + \left|\mathbf{B}_{\alpha}\bm{v}_{\alpha}-\mathbf{B}_{\psi}(\tau)\bm{v}_{\psi}(\tau)\right| \leq C\alpha\|h(\tau)\|_{L^{1}_{3}}=C\alpha\|\psi(\tau)-\psi_{\alpha}\|_{L^{1}_{3}}\,,$$
for some positive constant $C >0$ depending only on $f_{0}$. Since $\bm{v}_{\alpha}=0$ and
$$\|\psi(\tau)-\psi_{\alpha}\|_{L^{1}_{3}} \leq C \|\psi(\tau)-\psi_{\alpha}\|_{L^{1}(\m)} \leq C_{\varepsilon}\exp(-(\mu_{\star}-\varepsilon)\tau)\,\qquad\tau\geq1,$$
the result follows. 
\end{proof}

This, combined with Lemma \ref{lem:nEtau}, translates in the following result.
\begin{prop}\phantomsection\label{prop:rate}
Under the assumptions of Theorem \ref{main:no-scaled}, for all $\varepsilon >0$ and $\alpha \in (0,\alpha_{c})$  one has
$$\tau(t) \simeq \frac{2}{\alpha(\mathbf{a}_{\alpha}+\mathbf{b}_{\alpha})}\log t, \qquad \text{ as } \qquad t \to \infty$$
and
$$\log n_{f}(t)  \simeq -2\frac{ \mathbf{a}_{\alpha}}{\mathbf{a}_{\alpha}+\mathbf{b}_{\alpha}}\log t, \qquad \text{ and } \qquad \log T_{f}(t) \simeq -\frac{4\mathbf{B}_{\alpha}}{\alpha(\mathbf{a}_{\alpha}+\mathbf{b}_{\alpha})}\log t\qquad \text{ for } t \to \infty.$$
Finally, $\lim_{t \to \infty}\bm{u}_{f}(t)=\bm{u}_{f_{0}}+\sqrt{2T_{f_0}}\displaystyle\int_{0}^{+\infty}\mathbf{B}_{\psi}(s)\bm{v}_{\psi}(s)\exp\left(-\int_{0}^{s}\mathbf{B}_{\psi}(r)\d r\right)\d s$.
\end{prop}
\begin{proof} We notice that, from Lemma \ref{lem:ABv}, $\lim_{\tau\to\infty}\mathbf{a}_{\psi}(\tau)=\mathbf{a}_{\alpha}$ and $\lim_{\tau\to\infty}\mathbf{B}_{\psi}(\tau)=\mathbf{B}_{\alpha}$ so that, by a Cesaro-type argument (noticing that both mappings $s \mapsto \mathbf{a}_{\psi}(s)$ and $s \mapsto \mathbf{B}_{\psi}(s)$ are locally integrable),
$$\lim_{t\to\infty}\frac{1}{\tau(t)}\int_{0}^{\tau(t)}\mathbf{a}_{\psi}(s)\d s=\mathbf{a}_{\alpha}, \qquad \lim_{t\to\infty}\frac{1}{\tau(t)}\int_{0}^{\tau(t)}\mathbf{B}_{\psi}(s)\d s=\mathbf{B}_{\alpha}.$$
Then, a direct consequence of Lemma \ref{lem:nEtau} is that
\begin{equation}\label{eq:nfEf}
n_{f}(t) \simeq n_{f_{0}}\exp\left(-\alpha\,\mathbf{a}_{\alpha}\tau(t)\right), \qquad \text{ and } \qquad T_{f}(t) \simeq T_{f_{0}}\exp\left(-2\mathbf{B}_{\alpha}\tau(t)\right) \qquad \text{ for } t \to \infty\end{equation}
and 
\begin{equation}\label{lim_u}
\bm{u}_{f}(t) \simeq \bm{u}_{f_{0}}
+ \sqrt{2T_{f_0}}\int_{0}^{+\infty}\mathbf{B}_{\psi}(s)\bm{v}_{\psi}(s)\exp\left(-\int_{0}^{s}\mathbf{B}_{\psi}(r)\d r\right)\d s \qquad \text{ as } t \to \infty.
\end{equation}
{Let us note that the above integral converges, at least for $\alpha$ small enough. Indeed, we deduce from \eqref{conv} that there exists some constant $C_\alpha$ such that $$|\mathbf{B}_{\psi}(s)\bm{v}_{\psi}(s)|\leq C_\alpha, \qquad s\geq 0.$$ 
On the other hand, \eqref{conv} and Theorem \ref{theo:main-rescaled} imply that 
for fixed $\alpha$, $\lim_{s\to +\infty} \mathbf{B}_{\psi}(s)=\mathbf{B}_{\alpha}$. Moreover, by Remark \ref{imporem}, we have $\mathbf{B}_{\alpha}>0$ for $\alpha$ small enough. Thus, taking $\alpha$ small enough, there exists $\tau_0>0$ such that 
$$\mathbf{B}_{\psi}(s)\geq \frac{1}{2} \mathbf{B}_{\alpha}>0, \qquad s\geq \tau_0,$$whence the convergence of the integral in \eqref{lim_u}. 
}

The same reasoning as above also shows that 
$$\frac{\d}{\d t}\tau(t) \simeq c_{f_{0}}\exp\left(-\frac{\alpha}{2}(\mathbf{a}_{\alpha}+\mathbf{b}_{\alpha})\tau(t)\right) \qquad \text{ as } \, t \to\infty$$
where we set $c_{f_{0}}=n_{f_{0}}\sqrt{2T_{f_{0}}}$.  An application of de L'H\^{o}pital's rule shows that
$$\lim_{t\to\infty}\frac{1}{t}\exp\left(\alpha\frac{\mathbf{a}_{\alpha}+\mathbf{b}_{\alpha}}{2}\tau(t)\right)=c_{f_{0}}\alpha\frac{\mathbf{a}_{\alpha}+\mathbf{b}_{\alpha}}{2},$$
that is, $$\tau(t) \simeq \frac{2}{\alpha(\mathbf{a}_{\alpha}+\mathbf{b}_{\alpha})}\log\left(c_{f_{0}}\alpha\frac{\mathbf{a}_{\alpha}+\mathbf{b}_{\alpha}}{2}t\right) \simeq \frac{2}{\alpha(\mathbf{a}_{\alpha}+\mathbf{b}_{\alpha})}\log t \qquad \text{ as } \, t \to \infty.$$
This, combined with \eqref{eq:nfEf} gives the result.\end{proof}

\begin{proof}[Proof of Theorem \ref{main:no-scaled}]
The proof follows directly from Propositions \ref{prop:main-noscal} and \ref{prop:rate}. It only remains to show that the rates obtained for $n_f(t)$ and $T_f(t)$ in Proposition \ref{prop:rate} only depend on $\alpha$. Let us fix $\varrho>0$ and $E>0$ and let $\overline{\psi}_\alpha$ be the unique solution to \eqref{steady} that has mass $\varrho$, energy $E$ and zero momentum.  Let us denote by $\overline{\mathbf{A}}_\alpha, \overline{\mathbf{B}}_\alpha, 
\overline{\mathbf{a}}_\alpha$ and $\overline{\mathbf{b}}_\alpha$ the 
associated coefficients defined by \eqref{Aalpha}, \eqref{Balpha} and \eqref{abalpha} where $\psi_\alpha$ is replaced with $\overline{\psi}_\alpha$.
We deduce from \cite[Section 1.2]{jde} and Theorem \ref{theo:exist-unique} that 
$$ \overline{\psi}_\alpha(\xi)= \varrho \left(\frac{d\varrho}{2 \,E}\right)^{d/2} \; 
\psi_\alpha\left(\sqrt{\frac{d\varrho}{2 \,E}}\;\xi\right).$$
Consequently, the scaling properties of $\Q_-$ lead to
$$\overline{\mathbf{A}}_\alpha = \varrho \, \sqrt{\frac{2 \,E}{d\varrho}} \; \mathbf{A}_\alpha, \quad \quad
\overline{\mathbf{B}}_\alpha = \varrho \, \sqrt{\frac{2 \,E}{d\varrho}}\; \mathbf{B}_\alpha, \quad \quad 
\overline{\mathbf{a}}_\alpha = \varrho \, \sqrt{\frac{2 \,E}{d\varrho}}\; \mathbf{a}_\alpha \qquad \mbox{and} \qquad  \overline{\mathbf{b}}_\alpha = \varrho \, \sqrt{\frac{2 \,E}{d\varrho}} \;\mathbf{b}_\alpha.$$
{In particular, 
$$\frac{\overline{\mathbf{B}}_\alpha}{\overline{\mathbf{a}}_\alpha+\overline{\mathbf{b}}_\alpha}=\frac{\mathbf{B}_{\alpha}}{\mathbf{a}_{\alpha}+\mathbf{b}_{\alpha}}, \qquad \text{ and } \qquad  \frac{\overline{\mathbf{a}}_{\alpha}}{\overline{\mathbf{a}}_\alpha+\overline{\mathbf{b}}_\alpha}=\frac{\mathbf{a}_{\alpha}}{\mathbf{a}_{\alpha}+\mathbf{b}_{\alpha}}.$$ This proves that the rates in \eqref{eq:rate} depend only on $\alpha$ and not on the mass and energy of $\psi_\alpha$.}
\end{proof}
 
 \part{Appendices}

 \appendix

\section{Proofs of Lemma \ref{lem:inverse} and Lemma \ref{lem:PaP0}}\label{app:prooflemma}

We collect here the proofs of two fundamental results in Section \ref{sec:spectral1}. Notations are those introduced in Sections \ref{sec:ll} and \ref{sec:spectral}. 
\begin{proof}[Proof of Lemma \ref{lem:inverse}] One has clearly that, for all $\mathrm{Re}\lambda >-{\mu}_{\star}$, $\left[\A \mathcal{R}(\lambda,\B_{\alpha})\right]^{k} \in \mathscr{B}(\X_{1},\X_{2})$, $\mathcal{R}(\lambda,\mathscr{L}_{0}) \in \mathscr{B}(\X_{2})$ and $\LL-\mathscr{L}_{0} \in \mathscr{B}(\X_{2},\X_{1})$ for $\alpha \in (0,\alpha^{\dagger})$ from which
$$\left\|\mathcal{J}_{\alpha,k}(\lambda)\right\|_{\mathscr{B}(\X_{1})} \leq \left\|\LL-\mathscr{L}_{0}\right\|_{\mathscr{B}(\X_{2},\X_{1})}\,\|\mathcal{R}(\lambda,\mathscr{L}_{0})\|_{\mathscr{B}(\X_{2})}\,\left\|\left[\A\,\mathcal{R}(\lambda,\mathcal{B}_{\alpha})\right]^{k}\right\|_{\mathscr{B}(\X_{1},\X_{2})}$$
Using \eqref{eq:Lo}, \eqref{eq:ARB} and \eqref{eq:llXk}, this yields to a bound
$$\left\|\mathcal{J}_{\alpha,k}(\lambda)\right\|_{\mathscr{B}(\X_{1})}  \leq C_{k}\,\ep_{1,1}(\alpha)\,|\lambda|^{-n_{0}}, \qquad \forall \mathrm{Re}\lambda \geq -{\nu}_{*}'$$
for some explicit constant $C_{k} >0$.
Choosing then $r_{k}(\alpha)=\left(C_{k}\,\ep_{1,1}(\alpha)\right)^{\frac{1}{n_{0}+1}}$, we get \eqref{eq:Jalk} and clearly $\lim_{\alpha\to0^{+}}r_{k}(\alpha)=0.$ Clearly then, if $\underline{\alpha}_{k}$ is chosen in such a way that $r_{k}(\alpha) < 1$ for all $\alpha \in (0,\underline{\alpha}_{k})$, one sees that $\mathbf{Id}-\mathcal{J}_{\alpha,k}(\lambda)$ is invertible in $\X_{1}$ for all $\lambda \in \bm{\Omega}_{k}(\alpha)$ with 
$$(\mathbf{Id}-\mathcal{J}_{\alpha,k}(\lambda))^{-1}=\sum_{p=0}^{\infty}\left[\mathcal{J}_{\alpha,k}(\lambda)\right]^{p}.$$
Let us fix then $\alpha \in (0,\underline{\alpha}_{k})$ and $\lambda \in \bm{\Omega}_{k}(\alpha)$. The range of $\Gamma_{\alpha,k}(\lambda)$ is clearly included in $\D(\B_{\alpha})=\D(\LL)$. Then, writing $\LL=\A+\B_{\alpha}$ we get 
$$(\lambda-\LL)\Gamma_{\alpha,k}(\lambda)=\left(\lambda-\B_{\alpha}-\A\right) \sum_{j=0}^{k-1}\mathcal{R}(\lambda,\B_{\alpha})\left[\A\mathcal{R}(\lambda,\B_{\alpha})\right]^{j}
+(\lambda-\LL)\mathcal{R}(\lambda,\mathscr{L}_{0})\left[\A\,\mathcal{R}(\lambda,\B_{\alpha})\right]^{k}.$$
The first term on the right-hand side is equal to
\begin{multline*}
\left(\lambda-\B_{\alpha}-\A\right) \sum_{j=0}^{k-1}\mathcal{R}(\lambda,\B_{\alpha})\left[\A\mathcal{R}(\lambda,\B_{\alpha})\right]^{j}
=\sum_{j=0}^{k-1}\left[\A\mathcal{R}(\lambda,\B_{\alpha})\right]^{j}
-\sum_{j=0}^{k-1}\A\mathcal{R}(\lambda,\B_{\alpha})\left[\A\mathcal{R}(\lambda,\B_{\alpha})\right]^{j} \\
=\mathbf{Id}-\left[\A\mathcal{R}(\lambda,\B_{\alpha})\right]^{k}\end{multline*}
while, writing simply $(\lambda-\LL)$ as $(\lambda-\mathscr{L}_{0})+(\mathscr{L}_{0}-\LL)$ the second term is equal to
$$\left[\A\,\mathcal{R}(\lambda,\B_{\alpha})\right]^{k}+(\mathscr{L}_{0}-\LL)\mathcal{R}(\lambda,\mathscr{L}_{0})\left[\A\,\mathcal{R}(\lambda,\B_{\alpha})\right]^{k}=\left[\A\,\mathcal{R}(\lambda,\B_{\alpha})\right]^{k}-\mathcal{J}_{\alpha,k}(\lambda).$$
This proves that
$$(\lambda-\LL)\Gamma_{\alpha,k}(\lambda)=\mathbf{Id}-\mathcal{J}_{\alpha,k}(\lambda)$$
and shows that $\Gamma_{\alpha,k}(\lambda)(\mathbf{Id}-\mathcal{J}_{\alpha,k}(\lambda))^{-1}$ is a right-inverse of $(\lambda-\LL).$

To prove that $\lambda-\LL$ is invertible, it is therefore enough to prove that it is one-to-one. Consider then the eigenvalue problem
$$\LL h=\lambda\,h, \qquad h \in \D(\LL)=\W_{2}^{2,1}(\m).$$
Since $\D(\mathscr{L}_{0})=\W^{1,1}_{2}(\m)$, one can write this as $(\lambda-\mathscr{L}_{0})h=\LL h-\mathscr{L}_{0}h$ and as such
$$\|h\|_{\X_{1}}=\|\mathcal{R}(\lambda,\mathscr{L}_{0})(\LL-\mathscr{L}_{0})h\|_{\X_{1}} \leq \ep_{1,1}(\alpha)\,\|\mathcal{R}(\lambda,\mathscr{L}_{0})\|_{\mathscr{B}(\X_{1})}\,\|h\|_{\X_{2}}$$
where we noticed that, since $\lambda\neq 0$ and $\mathrm{Re}\lambda >-\mu_\star$, $\lambda \in \varrho(\mathscr{L}_{0})$ and 
where we used \eqref{eq:llXk}. Notice that, according to Hille-Yoshida Theorem, there exits a constant  $C_{0} >0$ such that 
$$\|\mathcal{R}(\lambda,\mathscr{L}_{0})\|_{\mathscr{B}(\X_{1})} \leq C_{0}(\mathrm{Re}\lambda+\mu_\star)^{-1} \leq C_{0}({\mu}_{\star}-{\nu}'_{*})^{-1}$$ so that
$$\|h\|_{\X_{1}} \leq C_{0}\ep_{1,1}(\alpha)({\mu}_{\star}-{\nu}'_{*})^{-1}\|h\|_{\X_{2}}.$$
Let us now estimate $\|h\|_{\X_{2}}$. Since $\LL h=\lambda\,h$ one has $(\lambda-\B_{\alpha})h=\A h$ and
$h=\mathcal{R}(\lambda,\B_{\alpha})\A h$,
so that
$$\|h\|_{\X_{2}} \leq \|\mathcal{R}(\lambda,\B_{\alpha})\|_{\mathscr{B}(\X_{2})}\,\|\A h\|_{\X_{2}} \leq \frac{C_{2}}{\mathrm{Re}\lambda+{\nu}_{*}}\|\A h\|_{\X_{2}}\leq \frac{C_{2}\,\|\A\|_{\mathscr{B}(\X_{1},\X_{2})}}{{\nu}_{*}-{\nu}_{*}'}\|h\|_{\X_{1}}\leq \frac{C_{2}\,\|\A\|_{\mathscr{B}(\X_{1},\X_{2})}}{{\mu}_{\star}-{\nu}_{*}'}\|h\|_{\X_{1}}$$
for some positive constant $C_{2}$ which gives the equivalence between the norm $\|\cdot\|_{\X_{2}}$ and the modified equivalent norm $\llbracket\cdot\rrbracket$ obtained in Proposition \ref{prop:hypo}. Thus,
$$\|h\|_{\X_{1}} \leq C_{0} \ep_{1,1}(\alpha)\,\frac{C_{2}\|\A\|_{\mathscr{B}(\X_{1},\X_{2})}}{\left({\mu}_{\star}-{\nu}_{*}'\right)^{2}}\,\|h\|_{\X_{1}}$$and one sees that, up to reduce $\alpha$, one can assume that  $C_{0} \ep_{1,1}(\alpha)\,\frac{C_{2}\|\A\|_{\mathscr{B}(\X_{1},\X_{2})}}{\left({\mu}_{\star}-{\nu}_{*}'\right)^{2}}< 1$ which implies that $h=0.$ This proves that $\lambda-\LL$ is one-to-one and its right-inverse is actually its inverse. \medskip

To estimate now $\|\mathcal{R}(\lambda,\LL)\|_{\mathscr{B}(\X_{1})}$ one simply notices that
$$\|(\mathbf{Id}-\mathcal{J}_{\alpha,k}(\lambda))^{-1}\|_{\mathscr{B}(\X_{1})} \leq \sum_{p=0}^{\infty}\|\mathcal{J}_{\alpha,k}(\lambda)\|_{\mathscr{B}(\X_{1})}^{p}\leq \frac{1}{1-r_{k}(\alpha)}, \qquad \forall \lambda \in \bm{\Omega}_{k}(\alpha)$$
from which
$$\|\mathcal{R}(\lambda,\LL)\|_{\mathscr{B}(\X_{1})}\leq \frac{1}{1-r_{k}(\alpha)}\,\|\Gamma_{\alpha,k}(\lambda)\|_{\mathscr{B}(\X_{1})}$$
and, from the previous estimates of $\|\mathcal{R}(\lambda,\B_{\alpha})\|_{\mathscr{B}(\X_{1})}$, $\|\left[\A \mathcal{R}(\lambda,\B_{\alpha})\right]^{j}\|_{\mathscr{B}(\X_{1})}$ and $\|\mathcal{R}(\lambda,\mathscr{L}_{0})\|_{\mathscr{B}(\X_{1})}$  one checks without difficulty that there exists $C_{k} >0$ such that
$$\|\Gamma_{\alpha,k}(\lambda)\|_{\mathscr{B}(\X_{1})} \leq C_{k}\sum_{j=0}^{k}(\mathrm{Re}\lambda+{\nu}_{*})^{-j} \leq C_{k}\sum_{j=0}^{k}\left({\nu}_{*}-{\nu}_{*}'\right)^{-j}$$
from which we get the result.\end{proof}

\begin{proof}[Proof of Lemma \ref{lem:PaP0}] We use Lemma \ref{lem:inverse} for some suitable $k \in \N$ and let $\gamma_{k}(\alpha):=\{z \in \mathbb{C}\;;\;|z|=r_{k}(\alpha)\}$ where $r_{k}(\alpha)$ is provided by Lemma  \ref{lem:inverse}. One has
$$\mathbb{P}_{\alpha}=\frac{1}{2i\pi}\oint_{\gamma_{k}(\alpha)}\mathcal{R}(\lambda,\LL)\d\lambda, \qquad \mathbf{P}_{0}=\frac{1}{2i\pi}\oint_{\gamma_{k}(\alpha)}\mathcal{R}(\lambda,\mathscr{L}_{0})\d \lambda.$$
To prove that $\mathbb{P}_{\alpha} \in \mathscr{B}(\X_{1},\X_{2})$, it suffices to find some suitable estimate on $\|\mathcal{R}(\lambda,\LL)\|_{\mathscr{B}(\X_{1},\X_{2})}.$ Notice that, in the space $\X_{1}$, the range of $\mathcal{R}(\lambda,\LL)$ is indeed $\X_{2}$ which is the domain of $\LL$ (i.e. $\X_{2}=\D(\LL\vert_{\X_{1}})$). Therefore, the norm $\|\cdot\|_{\X_{2}}$ is equivalent to the graph norm of $\LL$ (seen as an operator of $\X_{1}$): there exists $C_{\alpha} >0$ such that
$$\|f\|_{\X_{2}} \leq C_{\alpha}\left(\|f\|_{\X_{1}} +\|\LL f\|_{\X_{1}}\right), \qquad \forall f \in \X_{2}.$$
Then, given $\lambda \in \gamma_{k}(\alpha)$ and $g \in \X_{1}$ one has 
$$\|\mathcal{R}(\lambda,\LL)g\|_{\X_{2}} \leq C_{\alpha}\left(\|\mathcal{R}(\lambda,\LL)g\|_{\X_{1}}+ \|\LL\mathcal{R}(\lambda,\LL)g\|_{\X_{1}}\right)$$
Since $\LL\mathcal{R}(\lambda,\LL)g=-g+\lambda\mathcal{R}(\lambda,\LL)g$ and $|\lambda|=r_{k}(\alpha)$ we get
$$\|\mathcal{R}(\lambda,\LL)g\|_{\X_{2}} \leq C_{\alpha}\left((1+r_{k}(\alpha))\|\mathcal{R}(\lambda,\LL)g\|_{\X_{1}} + \|g\|_{\X_{1}}\right).$$
Using \eqref{eq:estimR}, one has $\|\mathcal{R}(\lambda,\LL)\|_{\mathscr{B}(\X_{1})} \leq M_{k}(\alpha)$ for all $\lambda \in \gamma_{k}(\alpha)$ for some positive constant $M_{k}(\alpha)$ depending only on $k,\alpha$ and on ${\nu}_{*}'-{\nu}_{*}$. This shows that 
$$\sup_{\lambda\in \gamma_{k}(\alpha)}\|\mathcal{R}(\lambda,\LL)\|_{\mathscr{B}(\X_{1},\X_{2})}:=C(k,\alpha) < \infty$$
and this proves the bound on $\|\mathbb{P}_{\alpha}\|_{\mathscr{B}(\X_{1},\X_{2})}.$ Let us now prove \eqref{eq:limproj}. Recall that
$$\mathbb{P}_{\alpha}-\mathbf{P}_{0}=\frac{1}{2i\pi}\oint_{\gamma_{k}(\alpha)}\left[\mathcal{R}(\lambda,\LL)-\mathcal{R}(\lambda,\mathscr{L}_{0})\right]\d\lambda$$
with $\mathcal{R}(\lambda,\mathscr{L}_{0})-\mathcal{R}(\lambda,\LL)=\mathcal{R}(\lambda,\mathscr{L}_{0})(\mathscr{L}_{0}-\LL)\mathcal{R}(\lambda,\LL)$. However, even if for small $\alpha$, one can make $\mathscr{L}_{0}-\LL$  small, it appears difficult to obtain bounds on $\|\mathcal{R}(\lambda,\mathscr{L}_{0})-\mathcal{R}(\lambda,\LL)\|_{\mathscr{B}(\X_{1})}$ because of the domain loss in \eqref{eq:llXk}. Indeed, such a domain loss would require uniform bound on $\|\mathcal{R}(\lambda,\LL)\|_{\mathscr{B}(\X_{1},\X_{2})}$ for $\alpha \simeq 0$ and such bound cannot hold true because the range of $\mathcal{R}(\lambda,\mathscr{L}_{0})$ is not $\X_{2}$. We have then to proceed in a different way, following the  approach of \cite[Lemma 2.17]{Tr}. We apply Lemma \ref{lem:inverse}. We simply write 
$$\bm{G}_{\alpha}(\lambda)=\sum_{j=0}^{k-1}\mathcal{R}(\lambda,\B_{\alpha})\left[\A\,\mathcal{R}(\lambda,\B_{\alpha})\right]^{j}, \qquad 0 \leq \alpha < \underline{\alpha}^\dagger_{k}$$
so that Lemma \ref{lem:inverse} reads  $\mathcal{R}(\lambda,\LL)=\bm{G}_{\alpha}(\lambda)(\mathbf{Id}-\mathcal{J}_{\alpha,k}(\lambda))^{-1} + \mathcal{R}(\lambda,\mathscr{L}_{0})\left[\A \mathcal{R}(\lambda,\B_{\alpha}\right]^{k}(\mathbf{Id}-\mathcal{J}_{\alpha,k}(\lambda))^{-1}$
while one proves without difficulty that, since $\mathscr{L}_{0}=\A +\B_{0}$, it holds
$$\mathcal{R}(\lambda,\mathscr{L}_{0})=\bm{G}_{0}(\lambda) + \mathcal{R}(\lambda,\mathscr{L}_{0})\left[\A\,\mathcal{R}(\lambda,\B_{0})\right]^{k}.$$
Since $\lambda \mapsto \mathcal{R}(\lambda,\B_{0})$ and $\lambda \mapsto \mathcal{R}(\lambda,\B_{\alpha})$ are both analytic on $\overline{\mathbb{D}}(0,r_{k}(\alpha))$, one has 
\begin{equation}\label{eq:oiGG}
\oint_{\gamma_{k}(\alpha)}\bm{G}_{\alpha}(\lambda)\d\lambda=\oint_{\gamma_{k}(\alpha)}\bm{G}_{0}(\lambda)\d\lambda=0.\end{equation}
Consequently
$$\mathbf{P}_{0}=\frac{1}{2i\pi}\oint_{\gamma_{k}(\alpha)}\mathcal{R}(\lambda,\mathscr{L}_{0})\d\lambda=\frac{1}{2i\pi}\oint_{\gamma_{k}(\alpha)}\mathcal{R}(\lambda,\mathscr{L}_{0})\left[\A\,\mathcal{R}(\lambda,\B_{0})\right]^{k}\d\lambda$$
while
\begin{multline*}
\mathbb{P}_{\alpha}=\frac{1}{2i\pi}\oint_{\gamma_{k}(\alpha)}\bm{G}_{\alpha}(\lambda)(\mathbf{Id}-\mathcal{J}_{\alpha,k}(\lambda))^{-1}\d\lambda 
+\frac{1}{2i\pi}\oint_{\gamma_{k}(\alpha)}\mathcal{R}(\lambda,\mathscr{L}_{0})\left[\A\mathcal{R}(\lambda,\B_{\alpha})\right]^{k}(\mathbf{Id}-\mathcal{J}_{\alpha,k}(\lambda))^{-1}\d\lambda\\
=\frac{1}{2i\pi}\oint_{\gamma_{k}(\alpha)}\bm{G}_{\alpha}(\lambda)\mathcal{J}_{\alpha,k}(\lambda)(\mathbf{Id}-\mathcal{J}_{\alpha,k}(\lambda))^{-1}\d\lambda 
+\frac{1}{2i\pi}\oint_{\gamma_{k}(\alpha)}\mathcal{R}(\lambda,\mathscr{L}_{0})\left[\A\mathcal{R}(\lambda,\B_{\alpha})\right]^{k}(\mathbf{Id}-\mathcal{J}_{\alpha,k}(\lambda))^{-1}\d\lambda
\end{multline*}
where we used \eqref{eq:oiGG} in the first integral. From this, we get
\begin{equation*}\begin{split}
\mathbb{P}_{\alpha}-\mathbf{P}_{0}&=\frac{1}{2i\pi}\oint_{\gamma_{k}(\alpha)} \mathcal{R}(\lambda,\mathscr{L}_{0})\left\{
\left[\A\mathcal{R}(\lambda,\B_{\alpha})\right]^{k}(\mathbf{Id}-\mathcal{J}_{\alpha,k}(\lambda))^{-1}-\left[\A \mathcal{R}(\lambda,\B_{0})\right]^{k}
\right\}\d\lambda \\
&\phantom{+++++} + \frac{1}{2i\pi}\oint_{\gamma_{k}(\alpha)}\bm{G}_{\alpha}(\lambda)\mathcal{J}_{\alpha,k}(\lambda)(\mathbf{Id}-\mathcal{J}_{\alpha,k}(\lambda))^{-1}\d\lambda\\
&=\frac{1}{2i\pi}\oint_{\gamma_{k}(\alpha)}\mathcal{R}(\lambda,\mathscr{L}_{0})\left[\A\mathcal{R}(\lambda,\B_{\alpha})\right]^{k}\left[(\mathbf{Id}-\mathcal{J}_{\alpha,k}(\lambda))^{-1}-\mathbf{Id}\right]\d\lambda \\
&\phantom{++++++} +\frac{1}{2i\pi}\oint_{\gamma_{k}(\alpha)}\mathcal{R}(\lambda,\mathscr{L}_{0})\left\{\left[\A \mathcal{R}(\lambda,\B_{\alpha})\right]^{k}-\left[\A\mathcal{R}(\lambda,\B_{0})\right]^{k}\right\}\d\lambda \\
&\phantom{++++++ +++} + \frac{1}{2i\pi}\oint_{\gamma_{k}(\alpha)}\bm{G}_{\alpha}(\lambda)\mathcal{J}_{\alpha,k}(\lambda)(\mathbf{Id}-\mathcal{J}_{\alpha,k}(\lambda))^{-1}\d\lambda\\
&=:\mathbb{I}_{1,\alpha}+ \mathbb{I}_{2,\alpha}+\mathbb{I}_{3,\alpha}.\end{split}\end{equation*}
According to \eqref{eq:Jalk} and arguing as in the proof of \eqref{eq:estimR}, for any $\lambda \in \gamma_{k}(\alpha),$ the integrand in $\mathbb{I}_{3,\alpha}$ is such that
$$\|\bm{G}_{\alpha}(\lambda)\mathcal{J}_{\alpha,k}(\lambda)(\mathbf{Id}-\mathcal{J}_{\alpha,k}(\lambda))^{-1}\|_{\mathscr{B}(\X_{1})}
\leq \frac{r_{k}(\alpha)}{1-r_{k}(\alpha)}\|\bm{G}_{\alpha}(\lambda)\|_{\mathscr{B}(\X_{1})} \leq \frac{C_{k}r_{k}(\alpha)}{1-r_{k}(\alpha)}$$
for some positive constant depending on $k$ and on ${\nu}_{*}^{'}.$ Thus, $\|\mathbb{I}_{3,\alpha}\|_{\mathscr{B}(\X_{1})}= \mathcal{O}(r_{k}(\alpha)).$
In the same way, since $\left[(\mathbf{Id}-\mathcal{J}_{\alpha,k}(\lambda))^{-1}-\mathbf{Id}\right]=\mathcal{J}_{\alpha,k}(\lambda)(\mathbf{Id}-\mathcal{J}_{\alpha,k}(\lambda))^{-1}$, one gets that the integrand of $\mathbb{I}_{1,\alpha}$ is such that
\begin{multline*}
\left\|\mathcal{R}(\lambda,\mathscr{L}_{0})\left[\A\mathcal{R}(\lambda,\B_{\alpha})\right]^{k}\left[(\mathbf{Id}-\mathcal{J}_{\alpha,k}(\lambda))^{-1}-\mathbf{Id}\right]\right\|_{\mathscr{B}(\X_{1})} 
\leq \frac{r_{k}(\alpha)}{1-r_{k}(\alpha)}\|\mathcal{R}(\lambda,\mathscr{L}_{0})\left[\A\mathcal{R}(\lambda,\B_{\alpha})\right]^{k}\|_{\mathscr{B}(\X_{1})}
\end{multline*}
and, using \eqref{eq:ARB}, one gets easily that $\mathbb{I}_{1,\alpha}=\mathcal{O}(r_{k}(\alpha)).$
Now, concerning the integrand of $\mathbb{I}_{2,\alpha}$, one has
{\begin{eqnarray*}
\left[\A \mathcal{R}(\lambda,\B_{\alpha})\right]^{k}-\left[\A\mathcal{R}(\lambda,\B_{0})\right]^{k}& =& \sum_{j=0}^{k-1}\left[\A \mathcal{R}(\lambda,\B_{\alpha})\right]^{j}\A\,\bigg(\mathcal{R}(\lambda,\B_{\alpha})-\mathcal{R}(\lambda,\B_{0})\bigg)\left[\A\mathcal{R}(\lambda,\B_{0})\right]^{k-j-1}\\
& =& \sum_{j=0}^{k-1}\left[\A \mathcal{R}(\lambda,\B_{\alpha})\right]^{j+1}\bigg(\B_{\alpha}-\B_{0}\bigg)\mathcal{R}(\lambda,\B_{0})\left[\A\mathcal{R}(\lambda,\B_{0})\right]^{k-j-1}.
\end{eqnarray*}}
Since $k-j-1 \neq 0$ for all $j \in \{0,\ldots,k-2\}$, one can exploit the regularizing effect of $\A$ and prove, as in \eqref{eq:ARB} that, for all $\lambda \in \gamma_{k}(\alpha)$,  
\begin{multline*}
\left\|\big(\B_{\alpha}-\B_{0}\big)\mathcal{R}(\lambda,\B_{0})\left[\A\mathcal{R}(\lambda,\B_{0})\right]^{k-j-1}\right\|_{\mathscr{B}(\X_{1})} \\\leq \|\B_{\alpha}-\B_{0}\|_{\mathscr{B}(\X_{2},\X_{1})}\,\|\mathcal{R}(\lambda,\B_{0})\|_{\mathscr{B}(\X_{2})}\,\|\left[\A\mathcal{R}(\lambda,\B_{0})\right]^{k-j-1}\|_{\mathscr{B}(\X_{1},\X_{2})}
\leq C_{j,k}\ep_{1,1}(\alpha)\end{multline*}
for some positive constant $C_{j,k}$ where we used  \eqref{eq:llXk} since $\LL-\mathscr{L}_{0}=\B_{\alpha}-\B_{0}$. Next, for $j=k-1$, one deduces from  \eqref{eq:llXk} and \eqref{eq:ARB} that, for all $\lambda \in \gamma_{k}(\alpha)$,  
\begin{equation}\label{eq:3spaces}\begin{split} 
\left\|\left[\A \mathcal{R}(\lambda,\B_{\alpha})\right]^{k}\bigg(\B_{\alpha}-\B_{0}\bigg)\mathcal{R}(\lambda,\B_{0})\right\|_{\mathscr{B}(\X_{1})}&\leq \left\|\left[\A \mathcal{R}(\lambda,\B_{\alpha})\right]^{k}\right\|_{\mathscr{B}(\X_{0},\X_{1})}  \|\B_{\alpha}-\B_{0}\|_{\mathscr{B}(\X_{1},\X_{0})}\|\mathcal{R}(\lambda,\B_{0})\|_{\mathscr{B}(\X_{1})}\\
& \leq  C_{k-1,k}\ep_{0,0}(\alpha)\end{split}\end{equation}
for some positive constant $C_{k-1,k} >0.$
One concludes from this easily that $$\mathbb{I}_{2,\alpha}=\mathcal{O}(\ep_{1,1}(\alpha))+ \mathcal{O}(\ep_{0,0}(\alpha))$$
and the proof is complete.\end{proof}
\begin{nb} It is not clear whether the above Lemma is valid in the space $\X_{0}$. This comes from the fact that our last estimate \eqref{eq:3spaces} relies on the estimate of $\B_{\alpha}-\B_{0}$ in $\mathscr{B}(\X_{1},\X_{0})$. This explains why we need to work on the scales of three Banach spaces $\X_{2},\X_{1}$ and $\X_{0}$ and cannot work directly on $\X_{0}$ (and $\X_{1}=\D(\LL))$. This was already observed in a similar framework in \cite{Tr} and comes from the fact that the elastic limit $\alpha \to 0$ is strongly ill-behaved because of the loss of domain induced by the drift term. In particular, it appears difficult to apply directly the classical spectral perturbation theory developed in \cite{kato}.\end{nb}

\section{Main properties of the solutions to the rescaled Boltzmann equation}\label{app:point}

We prove in this Appendix the main properties of the solutions to \eqref{rescaBE} that we used in Section \ref{sec:entropy}. 

\subsection{Creation and propagation of algebraic and exponential moments} First, we prove the following evolution for the moments of $\psi(t,\xi)$. We set
$$m_{s}(t)=\int_{\R^{d}}\psi(t,\xi)|\xi|^{s}\d\xi, \qquad \forall\, s \geq 0.$$
We follow the approach in \cite{cmpde} and introduce, for all $s,p >0$
\begin{equation}\label{eq:Ssp}
S_{s,p}(t)=\sum_{k=1}^{k_{p}}\left(\begin{array}{c}p \\k\end{array}\right)\left(m_{sk+1}(t)m_{s(p-k)}(t)+m_{sk}(t)m_{s(p-k)+1}(t)\right)\end{equation}
where $k_{p}=\left[\frac{p+1}{2}\right]$ is the integer part of $\frac{p+1}{2}$. 

\begin{lem}\phantomsection\phantomsection \label{lem:appmom}
Let $f_{0} \in L^{1}_{3}(\R^{d})$ be a given nonnegative initial datum with $n_{f_{0}},T_{f_{0}} >0$. For any $\alpha \in (0,\alpha_{\star})$, let $\psi(t,\xi)$ be the unique solution to \eqref{rescaBE}. There exists $\tilde{\alpha}_{0} \in (0,\alpha_{\star})$ such that for $\alpha\in(0,\tilde{\alpha}_{0})$, $s\in (0,2]$ and $p_{0} > 2/s$, one has, for any $t\geq 0$ and any $p\geq p_{0} >2/s$,
$$\dfrac{\d }{\d t}m_{sp}(t) \leq (1-\alpha)\varrho_{sp/2}S_{s,p}(t)-K_{1}\,m_{sp+1}(t)+\alpha\,sp K_{2}\,m_{sp}(t)+\alpha spd \,m_{sp-1}(t)$$
where $K_{1}=1- \varrho_{\frac{sp_{0}}{2}}$, $\varrho_{k}$ is defined by \eqref{varrhoK} and $K_{2}$ is a positive constant depending only on $d$, $\alpha _0$ and
$\int_{\R^d} \psi_0(\xi) |\xi|^3 \, \d\xi$.  \end{lem}\phantomsection
\begin{proof} As in \cite{cmpde} (see also \cite[Lemma 3.1]{jde}), one has
$$\int_{\S^{d-1}}\left(|\xi'|^{2k}+|\xi_{*}'|^{2k}\right)\d\sigma \leq \varrho_{k}\,\left(|\xi|^{2}+|\xi_{*}|^{2}\right)^{k} \qquad \forall k \geq 1$$
where $\varrho_k$ is defined by \eqref{varrhoK}.
Notice that the mapping $k \geq 0 \mapsto \varrho_{k} \in (0,1)$ is decreasing and $\lim_{k\to\infty}\varrho_{k}=0.$ Introduce
$$\beta_{k}(\alpha)=(1-\alpha)\varrho_{k}.$$
After multiplying \eqref{rescaBE} by $|\xi|^{sp}$ and arguing as in \cite{cmpde} and \cite[Lemma 3.1]{jde} with $k=\frac{sp}{2}$, we obtain easily that
\begin{multline*}
\dfrac{\d}{\d t}m_{sp}(t) \leq \frac{1}{2}\beta_{\frac{sp}{2}}(\alpha)\int_{\R^{d}\times\R^{d}}\psi(t,\xi)\psi(t,\xi_{*})|\xi-\xi_{*}|\left(\left(|\xi|^{2}+|\xi_{*}|^{2}\right)^{\frac{sp}{2}}-|\xi|^{sp}-|\xi_{*}|^{sp}\right)\d\xi\d\xi_{*}\\
-\left(1-\beta_{\frac{sp}{2}}(\alpha)\right)\int_{\R^{d}\times\R^{d}}\psi(t,\xi)\psi(t,\xi_{*})|\xi-\xi_{*}|\,|\xi|^{sp}\d\xi\d\xi_{*}\\
+\big((d+sp)\mathbf{B}_{\psi}(t)-\mathbf{A}_{\psi}(t)\big)m_{sp}(t)
+sp \mathbf{B}_{\psi}(t) \bm{v}_{\psi}(t)\cdot \int_{\R^d} \xi|\xi|^{sp-2} 
\psi(t,\xi) \, \d\xi.
\end{multline*}
Since $\psi(t,\xi)$ has zero momentum and mass one, one has 
$$\int_{\R^{d}}\psi(t,\xi_{*})|\xi-\xi_{*}|\d\xi_{*} \geq \left|\xi-\int_{\R^{d}}\psi(t,\xi_{*})\xi_{*}\d\xi_{*}\right|=|\xi|,$$
which yields the lower bound:
$$\int_{\R^{d}\times\R^{d}}|\xi|^{sp}\psi(t,\xi)\psi(t,\xi_{*})|\xi-\xi_{*}|\d\xi\d\xi_{*} \geq m_{sp+1}(t).$$
One estimates the first integral as in \cite{cmpde} to get 
\begin{multline*}
\dfrac{\d}{\d t}m_{sp}(t) \leq \beta_{\frac{sp}{2}}(\alpha) S_{s,p}(t) -\left(1-\beta_{\frac{sp}{2}}(\alpha)\right)m_{sp+1}(t)\\
+\left((d+sp)\mathbf{B}_{\psi}(t)-\mathbf{A}_{\psi}(t)\right)m_{sp}(t)
+sp \mathbf{B}_{\psi}(t) \bm{v}_{\psi}(t)\cdot \int_{\R^d} \xi|\xi|^{sp-2} 
\psi(t,\xi) \, \d\xi.\end{multline*}
Now, one checks that 
$$(d+sp)\mathbf{B}_{\psi}(t)-\mathbf{A}_{\psi}(t) \leq \frac{\alpha\,sp}{d}\int_{\R^{d}}\Q_{-}(\psi,\psi)(t,\xi)|\xi|^{2}\d\xi \leq \frac{\alpha\,sp}{d}\left(m_{3}(t)+\tfrac{d}{2}m_{1}(t)\right)$$
and 
$$\big| \mathbf{B}_{\psi}(t) \bm{v}_{\psi}(t) \big|\leq \alpha \int_{\R^{d}}\Q_{-}(\psi,\psi)(t,\xi)|\xi|\d\xi \leq \alpha \big(d/2+m_1(t)^2\big)\,,$$
which results in
\begin{equation}\label{eq:msp}
\begin{split}
\dfrac{\d}{\d t}m_{sp}(t) \leq  &\beta_{\frac{sp}{2}}(\alpha)S_{s,p}(t) -\left(1-\beta_{\frac{sp}{2}}(\alpha)\right)m_{sp+1}(t) \\
&+\frac{\alpha\,sp}{d}\left(m_{3}(t)+(d/2)^{3/2}\right)m_{sp}(t) +\alpha spd \,m_{sp-1}(t)\,,
\end{split}\end{equation}
where we used that $m_{1}(t) \leq \sqrt{d/2}$. In particular, for $s=1$ and $p=3$, one obtains that
\begin{align*}
\dfrac{\d}{\d t}m_{3}(t) \leq  6 \beta_{\frac{3}{2}}(\alpha)\big(m_1(t)m_3(t)+(d/2)^2\big) &- \big(1-\beta_{\frac{3}{2}}(\alpha)\big)m_{4}(t) \\
&+\frac{3\alpha}{d}\left(m_{3}(t)+(d/2)^{3/2}\right)m_{3}(t) +\frac{3\alpha \,d^2}{2}\,.\end{align*}
H\"older inequality implies that $m_4(t)\geq \frac{2}{d}\: m_3(t)^2$ and one deduces that 
\begin{equation*}
\dfrac{\d}{\d t}m_{3}(t) \leq  -\left(1-\beta_{\frac{3}{2}}(\alpha)-\frac{3\alpha}{2}\right)\frac{2}{d}\:m_{3}(t)^2 +\left( 6 \beta_{\frac{3}{2}}(\alpha)+\frac{3\alpha}{2}\right)(d/2)^{1/2}m_{3}(t) +\frac{3\,d^2}{2}\: (\beta_{\frac{3}{2}}(\alpha)+ \alpha)\,.\end{equation*}
Let us fix $\tilde{\alpha}_{0}\in(0,\alpha_\star)$ satisfying 
$$\tilde{\alpha}_{0}<\frac{1-\varrho_{3/2}}{3/2-\varrho_{3/2}}.$$
Then, for $\alpha\in(0,\tilde{\alpha}_{0})$, one obtains 
\begin{equation*}
\dfrac{\d}{\d t}m_{3}(t) \leq  -\left(1-\beta_{\frac{3}{2}}(\tilde{\alpha}_{0})-\frac{3\tilde{\alpha}_{0}}{2}\right)\frac{2}{d}\:m_{3}(t)^2
+\left( 6 \varrho_{\frac{3}{2}}+\frac{3}{2}\right)(d/2)^{1/2}m_{3}(t) +\frac{3\,d^2}{2}\: (\varrho_{\frac{3}{2}}+ 1)\,.\end{equation*}
This shows that, for $\alpha < \tilde{\alpha}_{0}$, there is an explicit constant $\overline{M}_{3}$, depending only on $d$ (and $\tilde{\alpha}_{0}$) such that
\begin{equation}\label{moment3}
\sup_{t\geq 0}m_{3}(t) \leq \max \left(\overline{M}_{3}, \int_{\R^d} \psi_0(\xi) |\xi|^3 \, \d\xi\right).
\end{equation}
With this, \eqref{eq:msp} becomes, for any $\alpha \in (0,\tilde{\alpha}_{0})$,
\begin{equation*}
\dfrac{\d}{\d t}m_{sp}(t) \leq  \beta_{\frac{sp}{2}}(\alpha) S_{s,p}(t) - \left(1-\beta_{\frac{sp}{2}}(\alpha)\right) m_{sp+1}(t) 
+\alpha\,s\,p\,K_{2}\,m_{sp}(t)+\alpha spd \,m_{sp-1}(t)\,,\end{equation*}
where $K_2$ only depends on $d$, $\tilde{\alpha}_{0}$ and $m_3(0)$. This proves the result.
\end{proof}

\begin{lem}\label{appB.2}
Let $f_{0} \in L^{1}_{3}(\R^{d})$ be a given nonnegative initial datum with $n_{f_{0}},T_{f_{0}} >0$. For any $\alpha \in (0,\alpha_{\star})$, let $\psi(t,\xi)$ be the unique solution to \eqref{rescaBE}. There exists $\tilde{\alpha}_{0} \in (0,\alpha_{\star})$ such that for $\alpha\in(0,\tilde{\alpha}_{0})$ and $p > 0$, there exists some constant $C_{p}\geq 0$ depending only on $p$, $n_{f_{0}}$, $T_{f_{0}}$ and $\|f_0\|_{L^1_3}$ such that 
\begin{equation}\label{msp}
m_{p}(t)\leq C_{p} \max \left\{1,t^{-p}\right\} \qquad \mbox{ for } t>0.
\end{equation}
\end{lem}

\begin{proof}
We know from \eqref{moment3} that the third moment of solution $\psi(t)$ is uniformly propagated.  Moreover, it follows from Lemma \ref{lem:appmom} that, for any $p>3$, $m_p(t)$ becomes finite for all positive time. Now, observe that by H\"older's inequality
$$S_{1,p}(t)\leq C_{p}\,m_{1}(t)\,m_{p}(t), \qquad m_{p-1}(t)\leq (m_p(t))^{1-1/p}\qquad \mbox{ and }\qquad m_{p+1}(t)\geq (m_{p}(t))^{1+1/p}.$$
Consequently, we infer from Lemma \ref{lem:appmom} that $m_{p}(t)$ satisfies the differential inequality
\begin{equation}\label{eq1B.2}
\frac{d}{dt}m_{p}(t)\leq C_1 m_p(t) - K_1 (m_{p}(t))^{1+1/p} + C_2 (m_p(t))^{1-1/p}.
\end{equation}
Thus, for $t\in (0,1]$, \eqref{msp} follows by comparison with the upper solution $x^{*} = C/t^{p}$ of the differential equation $x'=C_{1}x - K_{1}x^{1+1/p} + C_{2}x^{1-1/p}$.  Then, once the moment is finite at time $t=1$, the same estimate \eqref{eq1B.2} implies that \eqref{msp} holds for $t\geq1$.
\end{proof}

\begin{proof}[Proof of Theorem \ref{theo:tail}] Introduce, as in \cite{cmpde}, 
$$E_{s}(t,z)=\int_{\R^{d}}\psi(t,\xi)\exp(z|\xi|^{s})\d\xi=\sum_{p=0}^{\infty}m_{sp}(t)\frac{z^{p}}{p!}, \qquad s \geq 0,\quad z >0,$$
and, for $n \in \mathbb{N}$
$$E_{s}^{n}(t,z)=\sum_{p=0}^{n}m_{sp}(t)\frac{z^{p}}{p!}, \qquad I_{s}^{n}(t,z)=\sum_{p=0}^{n}m_{sp+1}(t)\frac{z^{p}}{p!}.$$
 We consider here $s=1$ and fix $n \in \mathbb{N}$. We shall show that there exists $\tilde{a}\in(0,1)$ independent of $n$ such that, 
for any $a\in(0,\tilde{a})$ and any $0\leq t\leq 1$, one has 
$$E_{1}^{n}(t,at^\beta) < 4.$$ 
Since this is true for all $n \in \mathbb{N}$, this would imply the result for $t\leq 1$. Notice that \eqref{msp} implies that for $a\leq1$, 
$$E_{1}^n(t,at^\beta) \leq 1+ \sum_{p=1}^n C_{p} \frac{t^{(\beta-1)p}}{p!}$$ 
Since $\beta>1$, there exists $\tilde{t}$ small enough and depending on $n$ such that $E_{1}^{n}(t,at^\beta) < 4$ for all $a \leq 1$ and $t\in(0,\tilde{t})$. 
For $p_{0} >2/s$, by  Lemma \ref{lem:appmom},
\begin{align*}
\dfrac{\d}{\d t}\sum_{p=p_{0}}^{n}m_{p}(t)\frac{(at^{\beta})^{ p}}{p!} 
\leq \sum_{p=p_{0}}^{n}\varrho_{p/2}&S_{1,p}(t)\frac{(a t^{\beta})^{p}}{p!}
-K_{1}\sum_{p=p_{0}}^{n}m_{p+1}(t)\frac{(at^{\beta})^{p}}{p!}
+\alpha\,K_{2}\sum_{p=p_{0}}^{n}m_{p}(t) \frac{(a t^{\beta})^{p}}{(p-1)!}\\
&+\alpha\,d \sum_{p=p_{0}}^{n}m_{p-1}(t) \frac{(a t^{\beta})^{p}}{(p-1)!}
+ \beta\sum_{p=p_{0}}^{n}m_{p}(t)\frac{a^{p}t^{\beta p-1}}{(p-1)!} .
\end{align*}
Now, we have
\begin{align*}
\sum_{p=p_{0}}^{n}m_{p}(t) \frac{(a t^{\beta})^{p}}{(p-1)!}& = \sum_{p=p_{0}-1}^{n-1}m_{p+1}(t) \frac{(a t^{\beta})^{p+1}}{p!}
\leq a\sum_{p=0}^{n}m_{p+1}(t)\frac{(at^\beta)^{p}}{p!}\,,\\
\sum_{p=p_{0}}^{n}m_{p-1}(t) \frac{(a t^{\beta})^{p}}{(p-1)!}& =  
\sum_{p=p_{0}-2}^{n-2}m_{p+1}(t) \frac{(a t^{\beta})^{p+2}}{(p+1)!}
\leq a^2\sum_{p=0}^{n}m_{p+1}(t)\frac{(at^\beta)^{p}}{p!}\,,\\
\sum_{p=p_{0}}^{n}m_{p}(t)\frac{a^{p}t^{\beta p-1}}{(p-1)!}& =  
\sum_{p=p_{0}-1}^{n-1}m_{p+1}(t)\frac{a^{p+1}t^{\beta p+\beta-1}}{p!}
\leq a\sum_{p=0}^{n}m_{p+1}(t)\frac{(at^{\beta})^{ p}}{p!}\,.
\end{align*}
Choosing $a$ small enough so that $a K_{2} \leq \frac{K_{1}}{4}$, $a^2d\leq \frac{K_{1}}{4}$
and $\beta a\leq\frac{K_{1}}{4}$ we get that
\begin{equation*}\begin{split}
\dfrac{\d}{\d t}\sum_{p=p_{0}}^{n}m_{p}(t)\frac{(at^\beta)^{p}}{p!} &\leq  \sum_{p=p_{0}}^{n}\varrho_{p/2}S_{1,p}(t)\frac{(at^\beta)^{p}}{p!}-\frac{K_{1}}{4}\;I_{1}^{n}(t,at^\beta)+K_{1}\sum_{p=0}^{p_0-1}m_{p+1}(t)\frac{(at^\beta)^{p}}{p!}\\
& \leq  \sum_{p=p_{0}}^{n}\varrho_{p/2}S_{1,p}(t)\frac{(at^\beta)^{p}}{p!}-\frac{K_{1}}{4}\;I_{1}^{n}(t,at^\beta) + \frac{1}{t}\;\tilde{C}_{p_{0}},
\end{split}\end{equation*}
with $\tilde{C}_{p_{0}}=\sum_{p=0}^{p_0-1} \frac{C_{p+1}}{p!}$, where we used \eqref{msp} and $a\leq 1$.  From here, we can then argue exactly as in \cite[Theorem 1]{cmpde} to get the result for $t\leq 1$.

\medskip
We shall now show that for any $a\in(0,\tilde{a})$ and any $t\geq 1$, one has $E_{1}^{n}(t,a) < 4$.  Since this is true for all $n \in \mathbb{N}$, this would imply the result for $t\geq 1$. 

Notice that we have just proved that $E_{1}^{n}(1,a) < 4$ for all $a \in(0, \tilde{a})$. 
Then, for $p_{0} >2/s$ and $t\geq 1$, 
by  Lemma \ref{lem:appmom},
\begin{align*}
\dfrac{\d}{\d t}\sum_{p=p_{0}}^{n}m_{p}(t)\frac{a^{p}}{p!} 
\leq \sum_{p=p_{0}}^{n}\varrho_{p/2}&S_{1,p}(t)\frac{a^{p}}{p!}
-K_{1}\sum_{p=p_{0}}^{n}m_{p+1}(t)\frac{a^{p}}{p!}\\
&+\alpha\,K_{2}\sum_{p=p_{0}}^{n}m_{p}(t) \frac{a^{p}}{(p-1)!}
+\alpha\,d \sum_{p=p_{0}}^{n}m_{p-1}(t) \frac{a^{p}}{(p-1)!}.
\end{align*}
Recall that the last two sums are bounded from above  by 
$a\sum_{p=0}^{n}m_{p+1}(t)\frac{a^{p}}{p!}$ and $a^2\sum_{p=0}^{n}m_{p+1}(t)\frac{a^{p}}{p!}$ respectively.  Thus, for $a\in(0,\tilde{a})$, we get 
\begin{equation*}\begin{split}
\dfrac{\d}{\d t}\sum_{p=p_{0}}^{n}m_{p}(t)\frac{a^{p}}{p!} &\leq  \sum_{p=p_{0}}^{n}\varrho_{p/2}S_{1,p}(t)\frac{a^{p}}{p!}-\frac{K_{1}}{2} I_{1}^{n}(t,a)+ K_1 \sum_{p=0}^{p_0-1}m_{p+1}(t)\frac{a^{p}}{p!}\\
& \leq  \sum_{p=p_{0}}^{n}\varrho_{p/2}S_{1,p}(t)\frac{a^{p}}{p!}-\frac{K_{1}}{2}I_{1}^{n}(t,a)+ K_1\tilde{C}_{p_{0}}
\end{split}\end{equation*}
with $\tilde{C}_{p_0}=\sum_{p=0}^{p_0-1}\frac{C_{p+1} }{p!}$, where we used \eqref{msp}. We can then argue exactly as in \cite[Theorem 2]{cmpde} to get the result.\end{proof}
We end this Appendix with well-known estimates about $\Q_{\pm}$ on the weighted $L^{1}$-spaces.
\begin{lem}\phantomsection \label{lem:estimatQ}
For any $b >0$, set 
$$m_{b}(\xi)=\exp(b|\xi|), \qquad \xi \in \R^{d}.$$
Then, for any $q \geq 0$, there exists $C_{b,q} >0$ such that
$$\|\Q_{\pm}(h,g)\|_{L^{1}_q(m_{b})} + \|\Q_{\pm}(g,h)\|_{L^{1}_q(m_{b})} \leq C_{b,q}\|h\|_{L^{1}_{q+1}(m_{b})}\,\|g\|_{L^{1}_{q+1}(m_{b})}$$
and
$$\|\Q_{\pm}(h,h)\|_{L^{1}(m_{b})} \leq C_{b}\|h\|_{L^{1}_{2}(m_{b})}\,\|h\|_{L^{1}(m_{b})}.$$ 
\end{lem}
\begin{proof} Without any loss of generality, one shall assume that $h$ and $g$ are nonnegative. One first notices that, for any $h,g \in L^{1}(m_{b})$, one has
$$\|\Q_{+}(h,g)\|_{L^{1}_{q}(m_{b})}=\|\Q_{+}(h,g)m_{b}\langle \cdot \rangle^{q}\|_{L^1}=\sup_{\|\psi\|_{L^\infty}=1}\int_{\R^{d}}\Q_{+}(h,g)(v)\,m_{b}(v)\psi(v)\langle v\rangle^{q} \d v.$$
To estimate this last integral, one can assume without loss of generality that $h,g,\psi$ are nonnegative. Then, using the weak formulation of $\Q_{+}$:
$$\int_{\R^d}\Q_{+}(h,g)(v)\,m_{b}(v)\psi(v)\langle v\rangle^{q}  \d v=\int_{\R^d \times \R^d\times\S^{d-1}} h(v)g(v_{*})|v-v_{*}|\,m_{b}(v_{*}')\psi(v_{*}')\langle v'_{*}\rangle^{q} \,\d v\d v_{*}\d\sigma$$
where the post-collision velocity $v_{*}'$ is defined by \eqref{eq:rel:vit}. Clearly $|v_{*}'| \leq  |v|+|v_{*}|,$
i.e. $m_{b}(v_{*}') \leq m_{b}(v)m_{b}(v_{*})$ and $\langle v'_*\rangle^{q} \leq \langle v\rangle^{q} \langle v_{*}\rangle^{q}.$  Therefore,
\begin{equation*}
\int_{\R^d}\Q_{+}(h,g)(v)\,m_{b}(v)\psi(v)\langle v\rangle^{q}  \d v \leq \int_{\R^d \times \R^d} \left(m_{b}(v)\langle v\rangle^{q} \,h(v)\right)\left(m_{b}(v_{*})\langle v_{*}\rangle^{q} g(v_{*})\right)|v-v_{*}|\, \d v\d v_{*}.
\end{equation*}
One recognizes that this last integral is equal to $\int_{\R^d}\Q_{+}(m_{b}\langle \cdot\rangle^{q}\,h,m_{b}\langle \cdot\rangle^{q}\,g)(v) \d v$ and this proves that
\begin{equation}\label{eq:hmb}
\|\Q_{+}(h,g)\|_{L^{1}_{q}(m_{b})}\leq \|\Q_{+}(m_{b}\langle \cdot\rangle^{q}\,\,h,m_{b}\,\langle \cdot\rangle^{q}\,g)\|_{L^1}.\end{equation}
Then, the estimate follows easily from the well-known boundedness of the bilinear operator $\Q_{+} \:: \: L^1_1(\R^d) \times L^1_1(\R^d) \to L^1(\R^d)$ (see, e.g. \cite[Theorem 1]{AlCaGa}). The proof for $\Q_{-}$ is simpler since $\Q_{-}(h,g)(v) \leq h(v)\langle v\rangle \|g\|_{L^{1}_{1}}$ for any nonnegative $h,g$. Thus, $\|\Q_{-}(h,g)\|_{L^{1}_{q}(m_{b})} \leq \|g\|_{L^{1}_{1}}\|h\|_{L^{1}_{q+1}(m_{b})}$. 

For the quadratic estimate, one notices first that, by virtue of the above estimate,
$$\|\Q_{-}(h,h)\|_{L^{1}(m_{b})} \leq \|h\|_{L^{1}_{1}(m_{b)}}\|h\|_{L^{1}_{1}} \leq C_b\,\|h\|_{L^{1}_{2}(m_{b})}\|h\|_{L^{1}(m_{b})}.$$
Let us now focus on $\Q_{+}(h,h).$ From \eqref{eq:hmb}, it suffices to prove that,
\begin{equation}\label{eq:Q+L1}
\|\Q_{+}(f,f)\|_{L^{1}} \leq C\|f\|_{L^{1}_{2}}\|f\|_{L^{1}}, \qquad \forall f \in L^{1}_{2}(\R^{d}).\end{equation}
Indeed, applying this with $f=h\,m_{b}$ would yield the result. Now, using the weak formulation,
\begin{equation*}\begin{split}
\|\Q_{+}(f,f)\|_{L^{1}}&=\sup_{\|\psi\|_{L^\infty}=1}\int_{\R^{d}}\Q_{+}(f,f)(v)\,\psi(v) \d v\\
&=\sup_{\|\psi\|_{L^\infty}=1}\int_{\R^{d}\times\R^{d}\times\S^{d-1}}f(v)f(v_{*})|v-v_{*}|\psi(v_{*}')\d v\d v_{*}\d\sigma
\end{split}\end{equation*}
Using $|v-v_{*}| \leq \langle v\rangle\langle v_{*}\rangle \leq \frac{1}{2}\langle v\rangle^{2} + \frac{1}{2}\langle v_{*}\rangle^{2}$  we get easily
\begin{equation*}
\int_{\R^{d}\times\R^{d}\times\S^{d-1}}f(v)f(v_{*})|v-v_{*}|\psi(v_{*}')\d v\d v_{*}\d\sigma \leq \|f\|_{L^{1}_{2}}\|f\|_{L^{1}}\|\psi\|_{L^{\infty}}
\end{equation*}
for any $\psi \in L^{\infty}(\R^{d})$. This proves \eqref{eq:Q+L1}.
\end{proof}

\subsection{Propagation of Lebesgue and Sobolev norms}

We start with the proof of Lemma \ref{lem:Lp}. As said in the core of the paper, we provide the proof only for $p=2$ since it is the only case we   are dealing with.
\begin{proof}[Proof of Lemma \ref{lem:Lp}] Let $p=2$ and $\eta \geq 0$ be given. Multiply \eqref{rescaBE} by $\psi(t,\xi)\langle \xi\rangle^{2\eta}$ and integrate over $\R^{d}$. We obtain
\begin{multline*}
\frac{1}{2}\dfrac{\d}{\d t}\|\psi(t)\|_{L^{2}_{\eta}}^{2} + \Big(\mathbf{A}_{\psi}(t)-(\tfrac{d}{2}+\eta)\mathbf{B}_{\psi}(t)\Big)\|\psi(t)\|_{L^{2}_{\eta}}^{2} + \eta\,\mathbf{B}_{\psi}(t)\|\psi(t)\|_{L^{2}_{\eta-1}}^{2}\\
  + \eta\,\mathbf{B}_{\psi}(t)\bm{v}_{\psi}(t) \cdot \int_{\R^{d}}\psi(t,\xi)^{2}\xi\langle \xi\rangle^{2\eta-2}\d\xi \\
\leq (1-\alpha)\int_{\R^{d}}\langle \xi\rangle^{2\eta}\psi(t,\xi)\Q_{+}(\psi,\psi)(t,\xi)\d\xi - \int_{\R^{d}}\langle \xi\rangle^{2\eta}\psi(t,\xi)\Q_{-}(\psi,\psi)(t,\xi)\d\xi.
\end{multline*}
Recalling that $\int_{\R^{d}}\psi(t,\xi_{*})|\xi-\xi_{*}|\d\xi_{*} \geq \kappa_{0}\langle \xi\rangle$ for some explicit $\kappa_{0} >0$ we get
$$\int_{\R^{d}}\langle \xi\rangle^{2\eta}\psi(t,\xi)\Q_{-}(\psi,\psi)(t,\xi)\d\xi \geq \kappa_{0}\|\psi(t)\|_{L^{2}_{\eta+1/2}}^{2}.$$
Since moreover there exists a positive constant $\mathbf{K} >0$ such that 
$$\max(|\mathbf{A}_{\psi}(t)|,|\mathbf{B}_{\psi}(t)|,|\mathbf{B}_{\psi}(t)\bm{v}_{\psi}(t)|) \leq \mathbf{K}\,\alpha$$ for all $\alpha \in (0,\alpha_{0})$ so that one can choose $\alpha^{\star}_{\eta} \in(0,\alpha_{0})$ small enough so that 
\begin{equation}\label{eq:Kalphaeta}|\mathbf{A}_{\psi}(t)|+(\tfrac{d}{2}+2\eta)|\mathbf{B}_{\psi}(t)| +  \eta|\mathbf{B}_{\psi}(t)\bm{v}_{\psi}(t)| \leq \frac{\kappa_{0}}{2} \qquad \forall t \geq 0, \quad \alpha \in (0,\alpha_{\eta}^{\star}).\end{equation} Then, for all $\alpha \in (0,\alpha_{\eta}^{\star})$, it holds
\begin{equation*}
\frac{1}{2}\dfrac{\d}{\d t}\|\psi(t)\|_{L^{2}_{\eta}}^{2} + \frac{\kappa_{0}}{2}\|\psi(t)\|_{L^{2}_{\eta+1/2}}^{2} 
\leq (1-\alpha)\int_{\R^{d}}\langle \xi\rangle^{2\eta}\psi(t,\xi)\Q_{+}(\psi,\psi)(t,\xi)\d\xi.\end{equation*}
At this stage, we handle the last integral as for the classical Boltzmann equation (see \cite[Theorem 1]{AloGa}) and get the result.
\end{proof}

We extend the above result to $H^{1}_{\eta}$ spaces:
\begin{prop}\phantomsection\label{prop:h1} Let $\eta\geq 0$. There exists $\tilde{\alpha}_\eta\in(0,1)$ 
such that, for any $\psi_{0} \in L^{1}_{3}(\R^{d})$ with unit mass and satisfying
$$\psi_{0} \in  L^{1}_{\eta+3/2+\frac{d(d-2)}{d-1}}(\R^{d})\cap L^{2}_{\eta+3/2}(\R^{d})\cap H^{1}_{\eta}(\R^{d}),$$
then, the unique solution $\psi(t,\xi)$  to \eqref{rescaBE} with initial condition $\psi_0$ satisfies
$$\sup_{t \geq 0}\|\psi(t)\|_{H^{1}_{\eta}(\R^{d})}:=C_{\eta} <\infty.$$
\end{prop}

The proof is based on the following regularity estimates for Boltzmann operator due initially to \cite{bouchut} and extended in \cite{MoVi}
\begin{theo}\phantomsection\label{regularite} For all $s \geq 0$ and all $\eta \geq 0$, it holds
\begin{equation*}
\|\Q_{+}(g,f)\|_{H^{s+\frac{d-1}{2}}_{\eta}} \leq C_{d}\left(\|g\|_{H^{s}_{\eta+2}}\,\|f\|_{H^{s}_{\eta+2}}+\|g\|_{L^{1}_{\eta+1}}\|f\|_{L^{1}_{\eta+1}}\right).
\end{equation*}
for some positive constant $C_{d}$ depending only on the dimension $d$.
\end{theo}
With this in hands, the proof is standard computation:
\begin{proof}[Proof of Proposition \ref{prop:h1}] For any $i=1,\ldots,d$, set $\Phi(t,\xi)=\partial_{\xi_{i}}\psi(t,\xi)$. It is straightforward to check that $\Phi(t,\xi)$ satisfies
\begin{multline*}
\partial_{t}\Phi(t,\xi)+\left(\mathbf{A}_{\psi}(t)+\mathbf{B}_{\psi}(t)\right)\Phi(t,\xi)+ \mathbf{B}_{\psi}(t)(\xi-\bm{v}_{\psi}(t)) \cdot\nabla_{\xi} \Phi(t,\xi)
\\
=(1-\alpha)\partial_{\xi_{i}}\Q_{+}(\psi,\psi)(t,\xi)-\partial_{\xi_{i}}\Q_{-}(\psi,\psi)(t,\xi)\end{multline*}
Multiplying by $\langle \xi \rangle^{2\eta}\Phi(t,\xi)$ and integrating over $\R^{d}$ we get  
\begin{multline*}
\frac{1}{2}\dfrac{\d}{\d t}\|\Phi(t)\|_{L^{2}_{\eta}}^{2} + \left(\mathbf{A}_{\psi}(t)+(1-\frac{d}{2}-\eta)\mathbf{B}_{\psi}(t)\right)\|\Phi(t)\|_{L^{2}_{\eta}}^{2} +\eta\,\mathbf{B}_{\psi}(t)\|\Phi(t)\|_{L^{2}_{\eta-1}}^{2} \\
+\eta\mathbf{B}_{\psi}(t)\bm{v}_{\psi}(t) \cdot \int_{\R^{d}} \xi\Phi(t,\xi)^{2} \,\langle \xi\rangle^{2\eta-2}\d\xi\\
\leq (1-\alpha)\int_{\R^{d}}\langle \xi\rangle^{2\eta}\Phi(t,\xi)\,\partial_{\xi_{i}}\Q_{+}(\psi,\psi)(t,\xi)\d \xi- \int_{\R^{d}}\langle \xi\rangle^{2\eta}\Phi(t,\xi)\partial_{\xi_{i}}\Q_{-}(\psi,\psi)(t,\xi)\d\xi.
\end{multline*}
Notice that 
$$\partial_{\xi_{i}}\Q_{-}(\psi,\psi)(t,\xi)=\Phi(t,\xi)\int_{\R^{d}}\psi(t,\xi_{*})|\xi-\xi_{*}|\d\xi_{*} + \Q_{-}(\psi,\Phi)(t,\xi)
$$
we get as before that
\begin{equation*}\int_{\R^{d}}\langle \xi\rangle^{2\eta}\Phi(t,\xi)\partial_{\xi_{i}}\Q_{-}(\psi,\psi)(t,\xi)\d\xi \geq \kappa_{0}\|\Phi(t)\|_{L^{2}_{\eta+1/2}}^{2} +\int_{\R^{2d}}\langle \xi\rangle^{2\eta}\Phi(t,\xi) \Q_{-}(\psi,\Phi)(t,\xi)\d\xi.
\end{equation*}
As for \eqref{eq:Kalphaeta}, one can choose $\tilde{\alpha}_{\eta} \in(0,\alpha_{0})$ small enough so that 
\begin{equation*}|\mathbf{A}_{\psi}(t)|+(1+\frac{d}{2}+2\eta)|\mathbf{B}_{\psi}(t)| +  \eta|\mathbf{B}_{\psi}(t)\bm{v}_{\psi}(t)| \leq \frac{\kappa_{0}}{2} \qquad \forall t \geq 0, \quad \alpha \in (0,\tilde{\alpha}_{\eta}).\end{equation*} It thus holds, for  $\alpha \in (0,\tilde{\alpha}_{\eta})$, 
\begin{align*}
\frac{1}{2}\dfrac{\d}{\d t}\|\Phi(t)\|_{L^{2}_{\eta}}^{2} &+ \frac{1}{2}\kappa_{0}\|\Phi(t)\|_{L^{2}_{\eta+1/2}}^{2} \\
&\leq \|\Phi(t)\|_{L^{2}_{\eta+1/2}}\|\partial_{\xi_{i}}\Q_{+}(\psi(t),\psi(t))\|_{L^{2}_{\eta-1/2}} + \|\Phi(t)\|_{L^{2}_{\eta}}\|\Q_{-}(\psi(t),\Phi(t))\|_{L^{2}_{\eta}}\,.\end{align*}
At this stage the proof is exactly the same as the one usually used for the Boltzmann equation. Namely, noticing that
$$\Q_{-}(\psi,\Phi)(t,\xi)=\psi(t,\xi)\int_{\R^{d}}\Phi(t,\xi_{*})|\xi-\xi_{*}|\d\xi_{*}=\psi(t,\xi) \int_{\R^{d}}\psi(t,\xi_*)\partial_{\xi_{i}}|\xi-\xi_{*}|\d\xi_{*}$$
and $\left|\int_{\R^{d}}\psi(t,\xi_*)\partial_{\xi_{i}}|\xi-\xi_{*}|\d\xi_{*}\right| \leq \int_{\R^{d}}\psi(t,\xi_{*})\d\xi_{*}=1$ we see that
$$\|\Q_{-}(\psi(t),\Phi(t))\|_{L^{2}_{\eta}} \leq \|\psi(t)\|_{L^{2}_{\eta}}.$$
Using now Theorem \ref{regularite} with $s=\frac{3-d}{2}$ 
and the uniform propagation of $L^{2}_{\eta+3/2}$ and $L^{1}_{\eta+1/2}$-norms, we get easily to the conclusion.
\end{proof}

\subsection{Pointwise lower bounds}

We recall the following spreading properties of $\Q_+$ in general dimension $d \geq 3$:
\begin{prop}\phantomsection\label{spread}
For any $v_0 \in \mathbb{R}^d$ and any $\delta >0$, one has
\begin{equation*}
\mathrm{Supp}\left(\Q_{+}\big(\mathbf{1}_{\mathds{B}(v_0,\delta)}\,;\,\mathbf{1}_{\mathds{B}(v_0,\delta)}\big)\right)={\mathds{B}(v_0,\sqrt{2}\delta)}.
\end{equation*}
More precisely, for any $0 < \chi < 1$, there exists a universal $\kappa_{0}>0$ such that
\begin{equation}\label{bound}
\Q_{+}\big(\mathbf{1}_{\mathds{B}(v_0,\delta)}\,;\,\mathbf{1}_{\mathds{B}(v_0,\delta)}\big) \geq \kappa_{0} \,\delta^{d+1} \chi^{d+1} \mathbf{1}_{\mathds{B}(v_0,(1-\chi)\sqrt{2}\delta)}  \qquad \forall \delta >0,
\end{equation}
\end{prop}
\begin{proof} We just give a sketch of the proof which is well-known \cite{pulwen}.  We can assume without loss of generality that $v_{0}=0$. Let us assume $\delta=1$. As in \cite[Lemma 5.4]{AloLocmp}, we have 
$$\Q_+(\mathbf{1}_{\mathds{B}(0,1)},\mathbf{1}_{\mathds{B}(0,1)})\geq 
\frac{|\S^{d-1}|}{d+1}\left(\frac{\sqrt{2}}{2}\right)^{d+1}
\left(1-2^{-1/d}\right)\, \mathbf{1}_{\mathds{B}\left(0,\frac{\sqrt{2}}{2}\right)}.$$
Let $r\in(1/2,1)$ and $\xi\in\R^d$ with $|\xi|=\sqrt{2}r$. Let $\varepsilon \in(0,1)$ satisfying $\varepsilon <\frac{1-r}{2+\sqrt{2}}$ and set
$$\Omega_\varepsilon (\xi)=\left\{(\xi_*,\sigma)\in\R^d\times \S^{d-1} ; \quad|\xi_*|\leq \e  , \quad |({\xi+\xi_*})\cdot\sigma|\leq \e\,|\xi+\xi_{*}|\right\}$$
 For $(\xi_*,\s)\in\Omega_\e$, one has 
$$|\xi'|^2 = \frac{|\xi|^2}{2}+\frac{|\xi_*|^2}{2}+\frac{1}{2} \, |\xi-\xi_*| \, ({\xi+\xi_*})\cdot\sigma  \leq  r^2+ \frac{\e^2}{2}+\frac{\e}{2} (\sqrt{2}r+\e)^2
\leq  r+\e(2+\sqrt{2}) <  1,$$
since $r<1$ and $\e<1$. Similarly, $|\xi'_*|<1$. Since one also has $|\xi-\xi_*|\geq \sqrt{2}r-\e$, we deduce that  
\begin{equation*}\begin{split}
\Q_+(\mathbf{1}_{\mathds{B}(0,1)},\mathbf{1}_{\mathds{B}(0,1)})(\xi)& \geq  (\sqrt{2}r-\e) \int_{\Omega_\e} \d\s \d\xi_*\geq  (\sqrt{2}r-\e)\; \frac{|\S^{d-2}|}{|\S^{d-1}|}
\int_{\mathds{B}(0,\e)} \d\xi_*\int_{-\e}^\e (1-s^2)^{\frac{d-3}{2}} \d s\\
&\geq   (\sqrt{2}r-\e) \;|\S^{d-2}|\; \frac{\e^d}{d}\; 2\e\, (1-\e^2)^{\frac{d-3}{2}}.\end{split}
\end{equation*}
Since $\sqrt{2}r-\e>\frac{3}{2+\sqrt{2}} \; r$ and $\e\leq \frac{1}{\sqrt{2}},$ we get that
$$
\Q_+(\mathbf{1}_{\mathds{B}(0,1)},\mathbf{1}_{\mathds{B}(0,1)})(\xi)\geq \frac{6 |\S^{d-2}|}{d(2+\sqrt{2})}\left(\frac{1}{2}\right)^{\frac{d-3}{2}} r \,\e^{d+1}.$$ 
We then conclude as in the proof of \cite[Proposition 5.1]{AloLocmp}.
\end{proof}
Finally, we have the following (see \cite{mmjsp})
\begin{lem}\phantomsection\label{l1}
Fix $p\in(1,\infty]$. Let  $f$ be nonnegative  such that
\begin{equation}\label{itpe1}
\int_{\mathbb{R}^{d}}f(\xi)\d \xi=m,\qquad
\int_{\mathbb{R}^{d}}f(\xi)|\xi|^{2}\d \xi \leq E<\infty, \qquad \|f\|_p < \infty.
\end{equation}
Then, there exist $v_{0} \in \R^{d},$ $r$ and $\eta_0$ depending only on $m,E$ 
and $\|f\|_p$ and such that
\begin{equation*}
\Q_{+}\big(f,  \Q_{+}(f,f)\big) \geq \eta_0\,\mathbf{1}_{\mathds{B}(v_0,\,r)}.
\end{equation*}
\end{lem}

From now we will assume the solution $\psi(t,\xi)$ to \eqref{rescaBE} to be given and \emph{fixed}. It is clear that there exists $C_{0} >0$ large enough so that
\begin{equation*}
\Q_{-}(\psi,\psi)(t,\xi) \leq C_0(1+|\xi|)\psi(t,\xi), \qquad \text{ and } \qquad \max\left(\mathbf{A}_{\psi}(t),\mathbf{B}_{\psi}(t)\right) \leq C_{0} \qquad \forall t \ge 0.
\end{equation*}
Introduce then 
$$\sigma(\xi)=C_{0}(1+|\xi|), \qquad \Sigma(t,\xi)=\mathbf{A}_{\psi}(t)+\sigma(\xi).$$
Then, one can write \eqref{rescaBE} as
\begin{equation} \begin{split}
\label{rescaBE1} 
\partial_t\psi(t,\xi)&+\mathbf{B}_\psi(t)\, (\xi-\bm{v}_{\psi}(t)) \cdot \nabla_\xi \psi(t,\xi) + \Sigma(t,\xi)\psi(t,\xi)\\
&\phantom{++++++} =(1-\alpha)\Q_{+}(\psi,\psi)(t,\xi)+ \left(\sigma(\xi)\psi(t,\xi) -\Q_{-}(\psi,\psi)(t,\xi)\right),
\end{split}\end{equation}
and, assuming $\psi(0,\xi) =\psi_0(\xi)\geq 0$, we get $\sigma(\xi)\psi(t,\xi) - \Q_{-}(\psi,\psi)(t,\xi) \geq 0$ and
\begin{equation}\label{gt}
\partial_t \psi(t,\xi)+ \mathbf{B}_{\psi}(t) (\xi-\bm{v}_{\psi}(t)) \cdot \nabla_\xi \psi(t,\xi) + \Sigma(t,\xi)\psi(t,\xi) \geq (1-\alpha)\Q_{+}(\psi,\psi)(t,\xi).
\end{equation}
We introduce the characteristic curves associated to the transport operator in \eqref{gt},
\begin{equation*}\dfrac{\d}{\d t}X(t;s,\xi)=  \mathbf{B}_{\psi}(t) \, (X(t;s,\xi)-\bm{v}_{\psi}(t)), \qquad X(s;s,\xi)=\xi,
\end{equation*}
which produces a unique global solution given by
\begin{equation}\label{cara}
X (t;s,\xi)=\xi\,\exp\left( \int_s^t \mathbf{B}_{\psi}(\tau)\, \d\tau\right)-\int_s^t\mathbf{B}_{\psi}(\sigma)\bm{v}_{\psi}(\sigma)\exp\left( \int_\sigma^t \mathbf{B}_{\psi}(\tau)\, \d\tau\right)\d\sigma.
\end{equation}
In order to simplify notation let us introduce the evolution family $(\mathcal{S}_s^t)_{t \geq s  \geq 0}$ defined by
\begin{equation*}
\big[\mathcal{S}_s^t\,h\big](\xi):=\exp\left(-\int_s^t \Sigma\big(\tau,X(\tau;t,\xi)\big)\, \d\tau\right)h\big(X(s;t,\xi)\big) \qquad \forall t \geq s \geq 0,\;\; \forall h=h(v).
\end{equation*}
The evolution family preserves positivity, thus according to \eqref{gt} the solution $\psi(t,\xi)$ to \eqref{rescaBE} satisfies the following \emph{Duhamel inequality}
\begin{equation}\label{sol}
\psi(t,\xi) \geq \left[\mathcal{S}_0^t \psi_0\right](\xi)
+ (1-\alpha)\int_0^t \big[\mathcal{S}_s^t\Q_{+}\left(\psi(s,\cdot),\psi(s,\cdot)\right)\big](\xi)\d s.
\end{equation}

We have the following analogue of \cite[Lemma 5.15]{AloLocmp} where however the characteristic functions are not contractions anymore. We recall here that  there is some positive constant $\bm{b} >0$ depending only on $\|\psi_{0}\|_{L^{1}_{3}}$ such that
$$\left|\mathbf{B}_{\psi}(s)\right| \leq \bm{b}\,\alpha \qquad \mbox{ and } \qquad 
 \left|\mathbf{B}_{\psi}(s)\bm{v}_{\psi}(s)\right| \leq \bm{b}\,\alpha,\qquad \forall s \geq 0.$$
\begin{lem}\phantomsection\label{minS}
For any nonnegative $h=h(\xi) \geq 0$ and any $t \geq s \geq 0,$ one has
\begin{equation}\label{boundSt}
\big[\mathcal{S}_s^t\,h\big](\xi)  \geq \left(\lambda_s^t\right)^{d}\exp\left(-\sigma(\xi)\bm{u}_{\alpha}(t-s)\right)\big[\mathcal{T}^t_s h\big](\xi)
\end{equation}
where 
$$\lambda_s^t=\exp\left(- \int_s^t \mathbf{B}_{\psi}(\tau)\d \tau\right), \qquad \bm{u}_{\alpha}(\tau)=\frac{\exp(\bm{b}\alpha \tau)-1}{\bm{b}\alpha}, \qquad \tau > 0$$ and $\big[\mathcal{T}^t_s h\big](\xi):=h(X(s;t,\xi))$ for any $\xi\in\R^d$.
\end{lem}
\begin{proof} Notice that, for any $t \geq s\geq 0$, 
$\exp\left( \int_s^t \mathbf{B}_{\psi}(\tau)\, \d\tau\right)\leq 
\exp(\bm{b}\alpha(t-s)),$
and thus, using \eqref{cara}, we check without difficulty that 
$$|X(\tau;t,\xi)|\leq (1+|\xi|)\,\exp(\bm{b}\alpha(t-\tau)) -1, \qquad \forall \,0\leq \tau\leq t.$$
Therefore,
\begin{equation*}
\Sigma\big(\tau,X(\tau;t,\xi)\big) \leq \mathbf{A}_{\psi}(\tau) + \sigma(\xi) 
\exp\left(\bm{b}\,\alpha(t-\tau)\right)\big) \qquad \forall\, 0 \leq \tau \leq t, \quad \forall \xi \in \mathbb{R}^d.
\end{equation*}
Integrating this over $(s,t)$ we get
$$\int_{s}^{t}\Sigma\big(\tau,X(\tau;t,\xi)\big)\d\tau \leq \int_{s}^{t}\mathbf{A}_{\psi}(\tau)\d\tau  + \sigma(\xi)\bm{u}_{\alpha}(t-s),\qquad \forall t \geq s \geq 0. $$
Finally, since $\mathbf{A}_{\psi}(t)-d\mathbf{B}_{\psi}(t) \leq 0$ for all $t \geq 0$, we get that 
$$\exp\left(-\int_{s}^{t}\mathbf{A}_{\psi}(\tau)\d\tau\right) \geq \left(\lambda_{s}^{t}\right)^{d}$$
which yields the conclusion.
\end{proof}

We can then prove the analogue of \cite[Lemma 5.17]{AloLocmp}
\begin{lem}\phantomsection\label{SstQ}
For any nonnegative $f=f(\xi) \geq 0$ it holds
\begin{equation}\label{Tscal}
\T_s^t \Q_{+}\big(\T_0^s f, \T_\tau^s \Q_{+}(\T_0^\tau f,\T_0^\tau f)\big)=
(\lambda_\tau^0)^{d+1}(\lambda_s^0)^{d+1}\,\T_0^t \Q_{+}\big(\,f\,,\Q_+(f,f)\big)\,,
\end{equation}
for any $0 \leq \tau \leq s \leq t.$

In particular, when $f$ is compactly supported with support included in $\mathds{B}(0,\varrho)$ ($\varrho >0$), then for any $t >0$ there exists $C(t,\varrho) >0$ such that
\begin{equation}\label{concl}
\Ss_s^t \Q_{+}\big(\Ss_0^s f, \Ss_\tau^s \Q_{+}(\Ss_0^\tau f,\Ss_0^\tau f)\big) \geq C(t,\varrho)\T_0^t \Q_{+}\left(\,f\,,\Q_{+}(f,f)\right)\,, \qquad \forall\;\; 0 \leq \tau \leq s \leq t
\end{equation}
\end{lem}
\begin{proof} The proof is a simple adaptation of that of \cite[Lemma 5.17]{AloLocmp}. The proof of \eqref{Tscal} is exactly the same. It relies on the still valid following relation: for any $h$ and any $0\leq \tau\leq s$
\begin{equation}\label{simpleTscal}
\T_\tau^s \Q_+(h,h)= \left(\lambda_\tau^s\right)^{d+1} \Q_+(\T_\tau^s h,\T_\tau^s h).
\end{equation}
For the proof of \eqref{concl}, we just recall the main steps. For any $t \geq s\geq 0$, one has using \eqref{cara},
\begin{equation*}
|X(s;t,\xi)|
 \geq  |\xi|\,\exp\left( -\int_s^t \mathbf{B}_{\psi}(\tau)\, \d\tau\right)
-\int_s^t\bm{b}\alpha\exp(\bm{b}\alpha(\sigma-s))\d\sigma 
\geq  \lambda_s^t \, |\xi|- e^{\bm{b}\alpha(t-s)}+1.\end{equation*}
If $f(v)=0$ for any $|v| \geq \varrho$, then $\Ss_s^t f(\xi)=0$ for any $|\xi| \geq \lambda_t^s \left(\varrho+ e^{\bm{b}\alpha(t-s)} -1  \right)$ and \eqref{boundSt} shows that
\begin{equation*}\label{est1}
\Ss_s^t f  \geq \left(\lambda_s^t\right)^{d}\exp\left(-\sigma\left(\lambda_t^s\left(\varrho+ e^{\bm{b}\alpha(t-s)} -1  \right)\right)\bm{u}_{\alpha}(t-s)\right) 
\T_s^tf.
\end{equation*}
In particular,
\begin{equation*}
\Q_{+}\big(\Ss_0^\tau f,\Ss_0^\tau f \big) \geq \left(\lambda_0^\tau\right)^{2d}\exp\left(-2\sigma\left(\lambda_\tau^0\left(\varrho+e^{\bm{b}\alpha\tau}-1\right)\right)\bm{u}_{\alpha}(\tau)\right) \Q_{+}\left(\T_0^\tau f, \T_0^\tau f\right).
\end{equation*}
Now, the support of $\Q_{+}\left(\Ss_0^\tau f, \Ss_0^\tau f\right)$ is included in $\mathds{B}\left(0,\sqrt{2}\lambda_\tau^0 \left(\varrho+e^{\bm{b}\alpha\tau}-1\right)\right)$. Hence, the support of  $\Ss_\tau^s \Q_{+}\big(\Ss_0^\tau f,\Ss_0^\tau f\big)$ is included in 
$$\mathds{B}\left(0,\lambda^\tau_s\left(\sqrt{2}\lambda_\tau^0 \left(\varrho+e^{\bm{b}\alpha\tau}-1\right)+e^{\bm{b}\alpha(s-\tau)} -1\right)\right)\subset\mathds{B}\left(0,\sqrt{2}\lambda_s^0 \left(\varrho+e^{\bm{b}\alpha s}-1\right)\right) .$$
Consequently, we get thanks to \eqref{boundSt}, \eqref{simpleTscal} and the above estimates that
\begin{equation*}
\Ss_\tau^s \Q_{+}\big(\Ss_0^\tau f,\Ss_0^\tau f\big) \geq C_0(s,\tau,\varrho) \Q_{+}\left(\T_0^sf,\T_0^sf\right)
\end{equation*}
with
\begin{multline*}
C_0(s,\tau,\varrho)=\left(\lambda_0^\tau\right)^{2d}\left(\lambda_\tau^s\right)^{2d+1}\exp\left(-2\sigma\left(\lambda_\tau^0\left(\varrho+e^{\bm{b}\alpha\tau}-1\right)\right)\bm{u}_{\alpha}(\tau)\right)\\ \times\exp\left(-\sigma\left(\sqrt{2}\lambda_s^0 \left(\varrho+e^{\bm{b}\alpha s}-1\right)\right)\bm{u}_{\alpha}(s-\tau)\right).
\end{multline*}
Since the support of $\Q_{+}\big(\Ss_0^s f, \Ss_\tau^s \Q_{+}(\Ss_0^\tau f,\Ss_0^\tau f)\big)$ is included in $\mathds{B}\left(0,2\lambda_s^0 \left(\varrho+e^{\bm{b}\alpha s}-1\right)\right)$ it follows that the one of $\Ss_s^t \Q_{+} \big(\Ss_0^s f, \Ss_\tau^s \Q_{+}(\Ss_0^\tau f,\Ss_0^\tau f)$ is included in 
$$\mathds{B}\left(0,\lambda_t^s\left(2\lambda_s^0 \left(\varrho+e^{\bm{b}\alpha s}-1\right)+e^{\bm{b}\alpha (t-s)}-1\right)\right)
\subset\mathds{B}\left(0,2\lambda_t^0 \left(\varrho+e^{\bm{b}\alpha t}-1\right)\right) $$
Hence, 
\begin{equation*}
\Ss_s^t \Q_{+} \big(\Ss_0^s f, \Ss_\tau^s \Q_{+}(\Ss_0^\tau f,\Ss_0^\tau f)\big)  \geq C_1(t,s,\tau,\varrho) \, \T_0^t\Q_{+}\big(f,\Q_{+}(f,f)\big)
\end{equation*}
with
\begin{eqnarray*}
C_1(t,s,\tau,\varrho)&=& C_0(s,\tau,\varrho)\left(\lambda_s^t\right)^{d}\left(\lambda_0^s\right)^{d}\left(\lambda_s^0\right)^{2(d+1)}\exp\left(-\sigma\left(\lambda_s^0 \left(\varrho+e^{\bm{b}\alpha s}-1\right)\right)\bm{u}_{\alpha}(s)\right)\,\\
&&  \times \exp\left(-\sigma\left(2\lambda_t^0\left(\varrho+e^{\bm{b}\alpha t}-1\right)\right)\bm{u}_{\alpha}(t-s)\right)\\
&=&(\lambda_{0}^{t})^{d-2}\lambda_{\tau}^{t}\lambda_{s}^{t}\, \exp\left(-\sigma\left(\sqrt{2}\lambda_s^0 \left(\varrho+e^{\bm{b}\alpha s}-1\right)\right)\bm{u}_{\alpha}(s-\tau)\right)\\ 
& & \times  \exp\left(-2\sigma\left(\lambda_\tau^0\left(\varrho+e^{\bm{b}\alpha\tau}-1\right)\right)\bm{u}_{\alpha}(\tau)-\sigma\left(\lambda_s^0 \left(\varrho+e^{\bm{b}\alpha s}-1\right)\right)\bm{u}_{\alpha}(s)\right)\\
& & \times \exp\left(-\sigma\left(2\lambda_t^0\left(\varrho+e^{\bm{b}\alpha t}-1\right)\right)\bm{u}_{\alpha}(t-s)\right)
\end{eqnarray*}
Reminding that $\sigma(v)=C_0 + C_0|v|$, one observes that, on the one hand 
$$\bm{u}_{\alpha}(s-\tau)+ 2\bm{u}_{\alpha}(\tau)+ \bm{u}_{\alpha}(s)+\bm{u}_{\alpha}(t-s)\leq 5 \bm{u}_{\alpha}(t), \qquad 0\leq \tau\leq s \leq t,$$
and, on the other hand,  
\begin{multline*}
\sqrt{2}\lambda_s^0 \left(\varrho+e^{\bm{b}\alpha s}-1\right)\bm{u}_{\alpha}(s-\tau)
+2 \lambda_\tau^0\left(\varrho+e^{\bm{b}\alpha\tau}-1\right)\bm{u}_{\alpha}(\tau)\\
+\lambda_s^0 \left(\varrho+e^{\bm{b}\alpha s}-1\right)\bm{u}_{\alpha}(s)
+2\lambda_t^0\left(\varrho+e^{\bm{b}\alpha t}-1\right)\bm{u}_{\alpha}(t-s) \\
\leq 7\bm{u}_{\alpha}(t) e^{\bm{b}\alpha t} \left(\varrho+\bm{b}\alpha \bm{u}_{\alpha}(t)\right),
\end{multline*}
where we used that $\lambda_{s}^{0} \leq e^{\bm{b}\alpha\,s}\leq e^{\bm{b}\alpha\,t}$.  
Now, using that $\lambda_{s}^{t}\geq e^{-\bm{b}\alpha\,(t-s)}\geq e^{-\bm{b}\alpha\,t} $ 
and setting
\begin{equation}\label{CtR}
C(t,\varrho)=\exp\left(-d\bm{b}\,\alpha\,t -5C_0\bm{u}_{\alpha}(t) -7C_0\bm{u}_{\alpha}(t) e^{\bm{b}\alpha\,t} \left(\varrho+\bm{b}\alpha\bm{u}_{\alpha}(t)\right)\right)
\end{equation}
we get the result.\end{proof}

\begin{prop}\phantomsection\label{propoR0}
Assume that the initial datum $\psi_{0} \in L^{1}_{3}(\R^{d})$ has mass $1$, momentum $0$ and kinetic energy $d/2 >0$ and that 
$$\psi_{0} \in L^{p}(\R^{d})$$
for some $p >1$. Let $\psi(t,\cdot)$ be the solution to the rescaled equation \eqref{rescaBE}. For any $\tau_1 > 0$, there exist $R_1 >0$ large enough (depending only on $\psi_0$) and $\mu_1 >0$ such that
\begin{equation}\label{propR0}
\psi(t,\cdot) \geq \mu_1 \mathbf{1}_{\mathds{B}(0,R_1)}(\cdot)\,, \qquad \forall\;\; t \geq \tau_1.
\end{equation}
Moreover, for any sequence $(\chi_k)_k \in (0,1)$ and increasing sequence $(\tau_k)_k$ one has
\begin{equation}\label{gtmukRk}
\psi(t,\cdot) \geq \mu_k \mathbf{1}_{\mathds{B}(0,R_k)}\,, \qquad \forall\;\; t \geq \tau_k
\end{equation}
with
\begin{equation}\label{induction}
\left\{
\begin{split}
R_{k+1}&=(1-\chi_k)\sqrt{2}\left(R_k +1- e^{\bm{b}\alpha(\tau_{k+1}-\tau_{k})}\right)
e^{-\bm{b}\alpha(\tau_{k+1}-\tau_{k})}\\
\mu_{k+1}&=(1-\alpha)\kappa_{0}\chi_k^{d+1}\; \mu_k^2 \; \left(R_k +1- e^{\bm{b}\alpha(\tau_{k+1}-\tau_{k})}\right)^{d+1} \;\Xi_{R_k}(\tau_{k+1}-\tau_k), \quad \forall k \in \mathbb{N} 
\end{split}
\right.
\end{equation}
where $\kappa_{0}$ is the positive constant appearing in \eqref{bound} and we set for any $s \geq 0$ and $R > 0$,
\begin{equation*}
\Xi_R(s)=\int_0^s \exp\left(-d\bm{b}\alpha\tau-C_0(1+e^{\bm{b}\alpha\tau}(\sqrt{2}R+e^{\bm{b}\alpha\tau}-1))\bm{u}_{\alpha}(\tau)\right)\d \tau.
\end{equation*}
\end{prop}
\begin{proof} We describe briefly the main steps of the proof which follows  the one of \cite[Proposition 5.18]{AloLocmp} and \cite[Theorem 4.9]{mmjsp}. Notice only that, because $\mathbf{B}_{\psi}(\cdot)$ has no sign, the characteristic curves $X(s;t,\cdot)$ are not contractive and some additional work has to be done in the initialization step.  \medskip

\noindent $\bullet$ \textit{Step 1: Initialization}. Let $t_0 >0$ be fixed and define $\widehat{g}_0(t,\cdot)=\psi(t_0+t,\cdot)$ for $t >0$, and ${G}_0=\widehat{g}_0(0,\cdot)=\psi(t_0,\cdot)$. Using Duhamel inequality \eqref{sol} one has 
\begin{equation}\label{duha}
\widehat{g}_0(t,\cdot) \geq (1-\alpha)^2\int_0^t \d s \int_0^s  \mathcal{S}_{s+t_0}^{t+t_0}\Q_{+} \left( \mathcal{S}_{t_0}^{s+t_0} {G}_0, {\mathcal{S}_{\tau+t_0}^{s+t_0}}\Q_{+}\left( {\mathcal{S}_{t_0}^{\tau+t_0}}  {G}_0,{\mathcal{S}_{t_0}^{\tau+t_0}} {G}_0\,\right)\right)\d \tau.\end{equation}
For $R>0$ large enough, we have $G_0\geq {G}_0\,\mathbf{1}_{\mathds{B}(0,R)}=:\widehat{G}_0$ and $\int_{\R^d}\widehat{G}_0(\xi)\d\xi>0$. It then follows from \eqref{concl} that, for any  $0 \leq \tau \leq s \leq t \leq T_1$,
$$\mathcal{S}_{s+t_0}^{t+t_0}\Q_{+} \left( \mathcal{S}_{t_0}^{s+t_0} {G}_0, {\mathcal{S}_{\tau+t_0}^{s+t_0}}\Q_{+}\left( {\mathcal{S}_{t_0}^{\tau+t_0}}  {G}_0,{\mathcal{S}_{t_0}^{\tau+t_0}} {G}_0,\right)\right)
\geq C_{T_1} \T_{t_0}^{t+t_0} \Q_{+}\left(\widehat{G}_0,\,\Q_{+}(\widehat{G}_0,\widehat{G}_0)\right)\,,$$
with  $C_{T_1}=\inf_{t \in [0,T_1]}C(t,R)=C(T_{1},R)$ where we recall that $C(t,R)$ is given by \eqref{CtR}. For $T_1 >0$ small enough, one has $C_{T_1} > 1/2$ and
\begin{equation*}
\widehat{g}_0(t,\cdot) \geq  (1-\alpha)^2\;\frac{t^2}{4}\; \T_{t_0}^{t+t_0} \Q_{+}\left(\widehat{G}_0,\,\Q_{+}(\widehat{G}_0,\widehat{G}_0)\right)\,, \qquad \forall\;\; 0 \leq t \leq T_1.
\end{equation*}
It now follows from Lemma \ref{l1} that there exists $v_0\in\R^d$, $r_0$ and 
$\eta_0$ depending only on $\|\widehat{G}_0\|_{L^1}$, the energy of  $\psi_0$ 
and $\|\psi_0\|_{L^p}$ such that  
$$\Q_{+}\left(\widehat{G}_0,\,\Q_{+}(\widehat{G}_0,\widehat{G}_0)\right)
\geq \eta_0 \mathbf{1}_{\mathds{B}(v_0,r_0)}. $$
This leads to 
\begin{equation*}
\widehat{g}_0(t,\xi) \geq  (1-\alpha)^2\;\frac{t^2}{4}\; \eta_0 \;\mathbf{1}_{\mathds{B}(v_0,r_0)}(X(t_0;t+t_0,\xi)) \,, \qquad \forall\;\; 0 \leq t \leq T_1.
\end{equation*}
Let $\varepsilon>0$. For $T_1$ small enough, one has for any $t\in[0,T_1]$,  
$$e^{\bm{b}\alpha t}\leq 1+\frac{\varepsilon}{2}, \qquad \mbox{ and }\qquad 
(|v_0|+1) \left(e^{\bm{b}\alpha t}-1\right)
\leq \frac{\varepsilon \, r_0}{2(1+\varepsilon)}.$$
Consequently, as soon as $|\xi-v_0|\leq \frac{r_0}{1+\varepsilon}$, one has 
\begin{equation}\label{majX}\begin{split}
|X(t_0;t+t_0,\xi)-v_0|& \leq \lambda^{t+t_0}_{t_0} |\xi-v_0|+|v_0| \left|1-\lambda^{t+t_0}_{t_0}\right| +\bm{b}\alpha \int_{t_0}^{t+t_0}\lambda^\sigma_{t_0} \d\sigma\\
&\leq  e^{\bm{b}\alpha t} |\xi-v_0| +|v_0| \left(e^{\bm{b}\alpha t}-1\right)+ \left(e^{\bm{b}\alpha t}-1\right) \leq r_0. \end{split}
\end{equation}
This means that $\widehat{g}_0(t,\cdot) \geq  (1-\alpha)^2\;\frac{t^2}{4}\; \eta_0 \;\mathbf{1}_{\mathds{B}\left(v_0,\frac{r_0}{1+ \varepsilon}\right)}$ for all $t \in (0,T_{1})$. Hence, for any $t_1\in(0,T_1/2]$, it holds $\widehat{g}_0(t,\cdot) \geq  \eta_1 \;\mathbf{1}_{\mathds{B}(v_0,r_1)}$ for any $t_{1} \leq t \leq T_{1}$ with
\begin{equation}\label{eq:eta1}
\eta_1=(1-\alpha)^2\;\frac{t_1^2}{4}\; \eta_0 \qquad \text{ and } \qquad r_{1}=\frac{r_{0}}{1+\varepsilon}.
\end{equation}
Notice at this stage that an important difference with respect to \cite{AloLocmp} and \cite[Theorem 4.9]{mmjsp} is that $r_{1} < r_{0}.$  We set $\widehat{g}_1(t,\cdot)=\widehat{g}_0(t+t_1,\cdot)$. We have thus obtained that 
\begin{equation}\label{g1chap}
\widehat{g}_1(t,\cdot) \geq  \eta_1 \;\mathbf{1}_{\mathds{B}(v_0,r_1)} \,, \qquad \forall\;\; 0 \leq t \leq \frac{T_1}{2}. 
\end{equation}
 Using again Duhamel's inequality \eqref{sol} (recall that $\widehat{g}_{1}(t,\cdot)=\psi(t+t_{0}+t_{1},\cdot)$) one has 
\begin{equation}\label{duha}
\widehat{g}_1(t,\cdot) \geq (1-\alpha) \int_0^t \mathcal{S}_{\tau+t_0+t_1}^{t+t_0+t_1}\Q_{+} \left(\widehat{g}_1(\tau,\cdot) ,\widehat{g}_1(\tau,\cdot)\right)\d \tau.\end{equation}
Let $\chi\in(0,1)$, the value of which will be fixed later. We now deduce from \eqref{g1chap}, Proposition \ref{spread} and Lemma \ref{minS} that, for any $t\in [0,T_1/2]$ and any $\xi\in\R^d$,
\begin{multline*}
\widehat{g}_1(t,\xi) \geq (1-\alpha)\, \eta_1^2\, r_1^{d+1}\chi^{d+1} \kappa_0 \int_0^t\left(\lambda_{\tau+t_0+t_1}^{t+t_0+t_1}\right)^d \exp(-\s(\xi)\bm{u}_\alpha(t-\tau))\;\\
\mathbf{1}_{\mathds{B}(v_0,(1-\chi)\sqrt{2}r_1)}(X(\tau+t_0+t_1; t+t_0+t_1,\xi))\,  \d \tau.\end{multline*}
Arguing as in \eqref{majX}, one obtains, for any $0\leq \tau\leq t \leq T_1/2$ and $\xi\in\R^d$, 
$$ \mathbf{1}_{\mathds{B}(v_0,(1-\chi)\sqrt{2}r_1)}(X(\tau+t_0+t_1; t+t_0+t_1,\xi))\geq 
\mathbf{1}_{\mathds{B}\left(v_0,\frac{(1-\chi)\sqrt{2}r_1}{1+\varepsilon}\right)}(\xi).$$
On the other hand, for any $0\leq t \leq \tfrac{T_1}{2}$, 
$\int_0^t\left(\lambda_{\tau+t_0+t_1}^{t+t_0+t_1}\right)^d\,\d \tau \geq \int_0^t e^{-d\bm{b}\alpha (t-\tau)} \, \d\tau \geq t \, e^{-d\bm{b}\alpha T_1/2}$ so that, for any $t_2\in(0,T_1/4]$, 
$$\widehat{g}_1(t,\cdot) \geq \eta_2\, \mathbf{1}_{\mathds{B}\left(v_0,r_{2}\right)}, \qquad \forall t_2\leq t\leq T_1/4, $$
with $r_{2}:=\frac{(1-\chi)\sqrt{2}\,r_1}{1+\varepsilon}$ and 
\begin{equation}\label{eq:eta2}
\eta_2=(1-\alpha) \,\eta_1^2\,  r_1^{d+1}\chi^{d+1} \kappa_0 
\exp\left(-\s\left(|v_0|+r_{2}\right)\bm{u}_\alpha(T_1/2)\right)\;  t_2 \, e^{-d\bm{b}\alpha T_1/2}.\end{equation}
One  chooses now $\varepsilon$ and $\chi$ small enough such that $\frac{(1-\chi)\sqrt{2}}{1+\varepsilon}>1$.  
Iterating this procedure, one obtains that, for any $k\geq 1$, for any $t_i\in(0,\tfrac{T_1}{2^{i}}]$ $(i=1,\ldots,k)$, there exists $\eta_k$ such that,
$$\widehat{g}_k(t,\cdot):= \psi\left(t+\sum_{i=0}^kt_i\right)\geq \eta_k  \mathbf{1}_{\mathds{B}(v_0,r_k)},\qquad \mbox{ with }\qquad r_k=\left(\frac{(1-\chi)\sqrt{2}}{1+\varepsilon}\right)^{k-1} r_1.$$
Arguing exactly as in \cite{AloLocmp,mmjsp}, there exists some explicit $\eta_\star >0$ and some arbitrarily small $t_\star >0$, both independent of the initial choice of $t_0$, such that
\begin{equation*}
\psi(t_\star + t_0,\cdot) \geq \eta_\star \mathbf{1}_{\mathds{B}(0,R)}.
\end{equation*}
Notice that $t_{\star}=\sum_{i=1}^{k}t_{i}$ and $\eta_{\star}=\eta_{k}$ where $k$ is large enough in such a way that $\mathds{B}(0,R)\subset \mathds{B}(v_0,r_k)$. Since $t_{0} >0$ is arbitrary, this proves \eqref{propR0} with $R_{1}=R$, $\mu_{1}=\eta_{\star}.$ For the proof of Theorem \ref{a-l3}, it will be important to understand the way $\mu_{1}$ depend on $t_{1}$. We obtained in Eq. \eqref{eq:eta1} that $\eta_1=\mathcal{O}(t_1^2)$ while, from \eqref{eq:eta2}, $\eta_{2}=\mathcal{O}(t_{2}\eta_{1}^{2})$ for some $t_{2} \in (0,T_{1}/4)$ to be chosen.  Iterating this procedure one can check without difficulty that $\eta_{k}=\mathcal{O}\left(t_1^{2^k}\prod_{i=2}^{k} t^{2^{k-i}}_{i}\right)$
and, picking as in \cite{pulwen} $t_{i}=t_{1}^{i}$ $(i=1,\ldots,k)$ one obtains
\begin{equation}\label{eq:mu1}
\mu_{1}=\eta_{k}=\mathcal{O}\left(t_{1}^{N_{k}}\right)\end{equation}
with $N_{k}=2^{k}+\sum_{i=2}^{k}i2^{k-i}=5\,2^{k-1}-(k+2).$
\\

\noindent $\bullet$ \textit{Second step (Implementation of the induction scheme).} For $\tau_{1} >0$ and any $t >\tau_{1}$, we get using \eqref{propR0}
\begin{eqnarray}
\psi(t,\cdot) & \geq&  (1-\alpha)\int_{\tau_1}^t \Ss_s^t \Q_{+}\big(\psi(s,\cdot),\psi(s,\cdot)\big)\d s \nonumber\\
& \geq &  (1-\alpha)\, \mu_1^2 \int_{\tau_1}^t \Ss_s^t \Q_{+} \left( \mathbf{1}_{\mathds{B}(0,R_1)}\,,\, \mathbf{1}_{\mathds{B}(0,R_1)}\right)\d s. \label{gtDuh}
\end{eqnarray}
{Since the support of $\Q_{+}\left( \mathbf{1}_{\mathds{B}(0,R_1)}\,,\, \mathbf{1}_{\mathds{B}(0,R_1)}\right)$ is included in $\mathds{B}\big(0,\sqrt{2}R_1\big)$, the one of $\Ss_s^t \Q_{+} \left( \mathbf{1}_{\mathds{B}(0,R_1)}\,,\, \mathbf{1}_{\mathds{B}(0,R_1)}\right)$ is included in $\mathds{B}\big(0,\lambda^s_t(\sqrt{2}R_1+e^{\bm{b}\alpha(t-s)}-1) \big)\subset \mathds{B}\big(0,e^{\bm{b}\alpha(t-s)}(\sqrt{2}R_1+e^{\bm{b}\alpha(t-s)}-1)\big)$ and  we get from \eqref{boundSt} and \eqref{simpleTscal} that
\begin{equation*}\label{SstQe+}
\Ss_s^t \Q_{+} \left( \mathbf{1}_{\mathds{B}(0,R_1)}\,,\, \mathbf{1}_{\mathds{B}(0,R_1)}\right)\geq\left(\lambda_s^t\right)^{2d+1}
 \exp\big(-\sigma(\bm{\omega}_{{R_{1}}}(t-s))\bm{u}_{\alpha}(t-s)\big)\Q_{+}\left( \T_s^t\mathbf{1}_{\mathds{B}(0,R_1)}\,,\, \T_s^t\mathbf{1}_{\mathds{B}(0,R_1)}\right),
\end{equation*}
where $\bm{\omega}_{R}(\tau):=e^{\bm{b}\alpha\tau}(\sqrt{2}R+e^{\bm{b}\alpha\tau}-1) $.}  Now, since $|X(s;t,\xi)|\leq \lambda_s^t |\xi| + e^{\bm{b}\alpha (t-s)}-1,$
we deduce that 
$$\T_s^t\mathbf{1}_{\mathds{B}(0,R_1)}
\geq \mathbf{1}_{\mathds{B}\left(0,\lambda_t^s\overline{R}_1(t-s)\right)}$$
where $\overline{R}_1(\tau) :=R_1+1-e^{\bm{b}\alpha\tau}$ for any $\tau \geq0.$ 
This leads to 
\begin{align*}
\Ss_s^t \Q_{+} &\left( \mathbf{1}_{\mathds{B}(0,R_1)}\,,\, \mathbf{1}_{\mathds{B}(0,R_1)}\right) \\
&\geq\left(\lambda_s^t\right)^{2d+1}\exp\big(-\sigma(\bm{\omega}_{R_1}(t-s))\bm{u}_{\alpha}(t-s)\big)\Q_{+}\left(\mathbf{1}_{\mathds{B}(0,\lambda_t^s\overline{R}_1(t-s))}\,,\, \mathbf{1}_{\mathds{B}(0,\lambda_t^s\overline{R}_1(t-s))}\right). 
\end{align*}
Using Proposition \ref{spread},   for any $\chi_1 \in (0,1)$ this can be again bounded from below by
\begin{multline*}
\kappa_{0} \left(\lambda_s^t\right)^{2d+1}
 \exp\big(-\sigma(\bm{\omega}_{R_1}(t-s))\bm{u}_{\alpha}(t-s)\big)(\lambda_{t}^{s})^{d+1}\chi_{1}^{d+1}(\overline{R}_{1}(t-s))^{d+1}\mathbf{1}_{\mathds{B}(0,(1-\chi_{1})\sqrt{2}\lambda_{t}^{s}\overline{R}_{1}(t-s))}\\
 =\kappa_{0} \left(\lambda_{s}^{t}\right)^{d} \exp\big(-\sigma(\bm{\omega}_{R_1}(t-s))\bm{u}_{\alpha}(t-s)\big)\chi_{1}^{d+1}(\overline{R}_{1}(t-s))^{d+1}\mathbf{1}_{\mathds{B}(0,(1-\chi_{1})\sqrt{2}\lambda_{t}^{s}\overline{R}_{1}(t-s))}.\end{multline*}
 Notice that a difference with respect to \cite[Prop. 5.18]{AloLocmp} is that, here, it is not true that $\lambda_{t}^{s} \geq 1$ since $\mathbf{B}_{\psi}(\tau)$ has no sign. However, one has $\lambda_{t}^{s} \geq\exp(-\bm{b}\alpha\,(t-s))$. Using \eqref{gtDuh}, one obtains
$$\psi(t,\cdot) \geq  (1-\alpha)\,\kappa_{0}\,\mu_{1}^{2} \,\chi_{1}^{d+1} \int_{0}^{t-\tau_{1}}(\overline{R}_{1}(\tau))^{d+1}\exp\left(-d\bm{b}\alpha\,\tau-\sigma(\bm{\omega}_{R_1}(\tau))\bm{u}_{\alpha}(\tau)\right)\mathbf{1}_{\mathds{B}\left(0,(1-\chi_{1})\sqrt{2}e^{-\bm{b}\alpha\tau}\overline{R}_{1}(\tau)\right)}\d \tau.$$
Therefore,
\begin{equation*}
\psi(t,\cdot) \geq \mu_2 \,\mathbf{1}_{\mathds{B}(0,R_2) }\,, \qquad \forall\;\; t \geq \tau_2 > \tau_1
\end{equation*}
with $R_2=(1-\chi_1)\sqrt{2}e^{-\bm{b}\alpha(\tau_{2}-\tau_{1})}
\left(R_{1}+1-e^{\bm{b}\alpha (\tau_2-\tau_1)}\right)$ and
\begin{equation*}
\mu_2 = (1-\alpha)\,\kappa_{0}\,\mu_{1}^{2}\, \chi_{1}^{d+1}\,
\left(R_{1}+1-e^{\bm{b}\alpha (\tau_2-\tau_1)}\right)^{d+1}\; \Xi_{R_1}(\tau_2-\tau_1).
\end{equation*}
Iterating this procedure, we obtain the result.
\end{proof}

With this we can prove Theorem \ref{a-l3}
\begin{proof}[Proof of Theorem \ref{a-l3}] We apply Proposition \ref{propoR0} to a constant sequence $(\chi_k)_k$ and bounded sequence $(\tau_k)_{k \geq 1}$.   More precisely, let $t_1 > 0$ be fixed and write
\begin{equation*}
\tau_1=\frac{t_1}{2}, \qquad \tau_{k+1}=\tau_{k}+\frac{t_1}{2^{k+1}} \qquad \forall k \geq 1.
\end{equation*}
For any given $\varepsilon >0$, set $\chi_k=\varepsilon$ for all $k \geq 1$. One deduces from \eqref{induction} that
\begin{align*}
R_{k}= \left(\sqrt{2}(1-\varepsilon)\right)^{k-1}&R_{1}\exp\left(-\bm{b}\alpha\,t_{1}\sum_{j=2}^{k}2^{-j}\right) \\ &+ \sum_{i=1}^{k-1} \left(\sqrt{2}(1-\varepsilon)\right)^{k-i}\exp\left(-\bm{b}\alpha\,t_{1}\sum_{j=i+1}^{k}2^{-j}\right) \left(1-e^{\bm{b}\alpha \,t_1 2^{-i-1}}\right),
\end{align*}
that is
\begin{multline}\label{rk}
R_{k} =  \left(\sqrt{2}(1-\varepsilon)\right)^{k-1}R_{1}
\exp\left(-\bm{b}\alpha\,t_{1}\left(\frac{1}{2}-\frac{1}{2^{k}}\right)\right)\\
- \sum_{i=1}^{k-1} \left(\sqrt{2}(1-\varepsilon)\right)^{k-i}\exp\left(-\bm{b}\alpha\,t_{1}\left(\frac{1}{2^{i+1}}-\frac{1}{2^{k}}\right)\right)\left(1-e^{-\bm{b}\alpha \,t_1 2^{-i-1}}\right).
\end{multline}
It is clear that $R_k \leq (\sqrt{2})^{k-1} R_1$ for any $k\geq 1$. On the other hand, since $1-e^{-x}\leq x$ for any $x\geq 0$, we deduce that 
$$R_k\geq (\sqrt{2}(1-\varepsilon))^{k-1} R_1
\exp\left(-\bm{b}\alpha\frac{t_{1}}{2}\right) - \bm{b}\alpha \,t_1  
\sum_{i=1}^{k-1} \left(\sqrt{2}(1-\varepsilon)\right)^{k-i} 2^{-i-1} 
\qquad \forall\;\; k \geq 1.$$
We finally obtain for any $k \geq 1$
\begin{equation}\label{eq:rk}
\left(\sqrt{2}(1-\varepsilon)\right)^{k-1} \left( R_1 e^{-\bm{b}\alpha\frac{t_{1}}{2}} 
-\frac{(1-\varepsilon) \bm{b}\alpha \,t_1}{4(1-\varepsilon)-\sqrt{2}}\right) \leq  R_k\leq (\sqrt{2})^{k-1} R_1.
\end{equation}
Moreover, by definition of $\Xi_R(s)$, one has easily 
$$\Xi_{R_k}(\tau_{k+1}-\tau_k) \geq \frac{1}{d\bm{b}\alpha}\exp\left(-C_{0}(1+\bm{\omega}_{R_k}(\tau_{k+1}-\tau_{k}))\bm{u}_{\alpha}(\tau_{k+1}-\tau_{k})\right)\left(1-\exp(-d\bm{b}\alpha(\tau_{k+1}-\tau_{k})\right),$$
where we recall that $\bm{\omega}_R(\tau)=e^{\bm{b}\alpha\tau}(\sqrt{2}R+e^{\bm{b}\alpha\tau}-1) $. In particular, since $\tau_{k+1}-\tau_k \leq t_{1} \leq 1$ one sees that there is some positive constant $c(\alpha) >0$ (independent of $t_{1}$) such that
$$\frac{1}{d\bm{b}\alpha}\left(1-\exp(-d\bm{b}\alpha(\tau_{k+1}-\tau_{k})\right) \geq c(\alpha)(\tau_{k+1}-\tau_{k}), \qquad \forall k \geq 1.$$
Moreover, $\bm{u}_{\alpha}(\tau_{k+1}-\tau_{k}) \leq \bm{u}_{\alpha}(t_{1}) \leq \bm{u}_{\alpha}(1)$ for $t_{1} \leq 1$ and $\bm{\omega}_{R_k}(\tau_{k+1}-\tau_{k})\leq \bm{\omega}_{(\sqrt{2})^{(k-1)}R_1}(1) $ so that
\begin{equation*}
\Xi_{R_k}(\tau_{k+1}-\tau_k) \geq c(\alpha) \,\left(\tau_{k+1}-\tau_k\right)\,\exp\left(-C_{0}(1+e^{\bm{b}\alpha}(2^{k/2}R_1+e^{\bm{b}\alpha}-1)  )\bm{u}_{\alpha}({1})\right)
\end{equation*}
Using \eqref{induction}, one gets, as in \cite{AloLocmp} that for any $k\geq 1$,
$$\mu_{k+1} \geq (1-\alpha) c(\alpha)\kappa_{0}\varepsilon^{d+1}\frac{t_{1}}{2^{k+1}}\left(R_{k}+1-e^{\bm{b}\alpha \,t_1 2^{-k-1}}\right)^{d+1}\exp\left(-C_{0}(1+e^{2\bm{b}\alpha} )\bm{u}_{\alpha}({1})\right)\exp(-z(\alpha)2^{k/2}R_{1})\mu_{k}^{2},$$
for some explicit $z(\alpha) >0$. Arguing as before, we infer from  \eqref{rk} that 
$$R_{k}+1-e^{\bm{b}\alpha \,t_1 2^{-k-1}} \geq  \left(\sqrt{2}(1-\varepsilon)\right)^{k-1} \left( R_1 e^{-\bm{b}\alpha\frac{t_{1}}{2}} 
-\frac{(1-\varepsilon) \bm{b}\alpha \,t_1}{4(1-\varepsilon)-\sqrt{2}}\; e^{\bm{b}\alpha\frac{t_{1}}{2}} \right),$$
where we used that $\exp\left(-\bm{b}\alpha\,t_{1}\left(\frac{1}{2^{i+1}}-\frac{1}{2^{k}}\right)\right) \leq e^{\bm{b}\alpha\,\frac{t_{1}}{2^{k}}}\leq  e^{\bm{b}\alpha\,\frac{t_{1}}{2}}$ for any $1\leq i\leq k$.
Since $t_{1} \leq 1$ we get exactly as in \cite[Prop. 5.18]{AloLocmp} that
$$\mu_{k+1} \geq C_{\alpha}(\varepsilon)\frac{t_{1}}{2^{k+1}}\exp(-z(\alpha)R_{1}2^{\frac{k-1}{2}})\mu_{k}^{2}$$
for some positive constant $C_{\alpha}(\varepsilon)$ depending only on $\alpha$ and $\varepsilon$ (but not on $k$ or $t_{1}$). Arguing as in \cite{AloLocmp}
we get that, for $\varepsilon >0$ small enough and $a_{0} >2$ there is a positive $c_{0} >0$ depending on $\alpha$ but not $t_{1}$ so that
$$\psi(t,\xi) \geq \exp(-c_{0}(1+\log(1/t_{1}))|\xi|^{a_{0}}) \qquad \forall |\xi| \geq R_{1}\footnote{In the proof of \cite{AloLocmp}, the assumption $|\xi| \geq R_{1}$ is missing but is needed} \qquad \forall t \geq t_{1}$$  
Now, for $|\xi| < R_{1}$, we have from \eqref{propR0} that
$$\psi(t,\xi) \geq \mu_{1} \qquad \forall t \geq \tau_{1}=t_{1}/2.$$
We get the lower bound \eqref{eq:pointlower} using also the estimate on $\mu_{1}$ obtained in Proposition \ref{propoR0} (see \eqref{eq:mu1}).
\end{proof}

\section{$C_{0}$-semigroup generation properties}\label{app:gener}

We prove in this section that the operator $\B_{\alpha,\delta}$ is the generator of a $C_{0}$-semigroup in $L^{1}(\m)$ for suitable choice of $\alpha,\delta$. Recall the notations of Section \ref{sec:hypo}. One has, in the underlying $L^{1}_{q}(\m)$, 
$$\LL h=\mathcal{A}_{\delta}h+\mathscr{L}_{0}^{R,\delta}h -\Sigma_{\M}\,h - \P_{\alpha}^{0}h+ \mathcal{T}_{\alpha}h$$
with $\D(\LL)=\D(\mathcal{T}_{\alpha})=\W^{1,1}_{q+1}(\m)$ where
$$\mathcal{T}_{\alpha}h(\xi)=-\mathbf{B}_{\alpha}\mathrm{div}(\xi\,h(\xi)), \qquad \forall h \in \W^{1,1}_{q+1}(\m).$$
Introduce the (anti)-drift operator with absorption
$$\mathbf{T}_{\alpha}h(\xi)=-\Sigma_{\M}(\xi)h(\xi) + \mathcal{T}_{\alpha}h(\xi), \qquad  \forall h \in \W^{1,1}_{q+1}(\m).$$
Notice that, since there are $\sigma_{0} >0,$ $\sigma_{1} >0$ such that $0 \leq \sigma_{0}\langle \xi\rangle \leq \Sigma_{\M}(\xi) \leq \sigma_{1}\langle \xi\rangle$ for all $\xi \in \R^{d}$, the domain of $\mathbf{T}_{\alpha}$ coincides with that of $\T_{\alpha}$. One has then the following elementary result where we recall that $\m(\xi)=\exp(a|\xi|)$, $\xi \in \R^{d}$:
\begin{lem}\phantomsection\phantomsection Assuming $\alpha >0$ to be small enough so that 
$$\sigma_{0} \geq  \sqrt{2}a\mathbf{B}_{\alpha} \geq 0,$$ then the above operator $\mathbf{T}_{\alpha}\::\:\D(\mathbf{T}_{\alpha} )\subset L^{1}_{q}(\m) \to L^{1}_{q}(\m)$ with $\D(\mathbf{T}_{\alpha})=\W^{1,1}_{q+1}(\m)$ is the generator of a nonnegative $C_{0}$-semigroup $\{U_{\alpha}(t)\,;\,t \geq 0\}$ in $L^{1}_{q}(\m)$ given by
$$U_{\alpha}(t)f(\xi)=\exp\left(-\int_{0}^{t}\left[d\mathbf{B}_{\alpha} +\Sigma_{\M}\left(\xi\,e^{(\tau-t)\mathbf{B}_{\alpha}}\right)\right]\d \tau\right)f(\xi\,e^{-t\mathbf{B}_{\alpha}}), \qquad f \in L^{1}_{q}(\m), \qquad t \geq 0$$
such that
\begin{equation}\label{eq:ulp}
\left\|U_{\alpha}(t)f\right\|_{L^{1}_{q}(\m)} \leq \exp\left(-\left(\frac{\sigma_{0}}{\sqrt{2}}-q\mathbf{B}_{\alpha}\right)t\right)\|f\|_{L^{1}_{q}(\m)} \qquad \forall t \geq 0, \qquad f \in L^{1}_{q}(\m).\end{equation}
In particular, $\{U_{\alpha}(t)\,;\,t \geq 0\}$ is a nonnegative contraction semigroup in $L^{1}_{q}(\m)$ as soon as {$\sigma_{0} \geq \sqrt{2}q\mathbf{B}_{\alpha}$}.\end{lem}\phantomsection
\begin{nb}\phantomsection\phantomsection Notice that $\T_{\alpha}$ does not generate a $C_{0}$-semigroup in $L^{1}_{q}(\m)$. The absorption term here is exactly what allows to prove that $U_{\alpha}(t) \in \mathscr{B}(L^{1}_{q}(\m))$ for all $t \geq 0.$\end{nb}\phantomsection
\begin{proof} Using the characteristics method to solve the evolution problem
$$\partial_{t}g(t,\xi)=\mathbf{T}_{\alpha}g(t,\xi), \qquad t > 0, \xi \in \R^{d}$$
with initial datum $g(0,\xi)=f(\xi)$ shows that the only possible candidate to be the $C_{0}$-semigroup generated by $\mathbf{T}_{\alpha}$ is indeed $\{U_{\alpha}(t)\,;\,t \geq 0\}$.   Let us show \eqref{eq:ulp}. There is no loss of generality in assuming $f$ to be nonnegative. Then,
$$\left\|U_{\alpha}(t)f\right\|_{L^{1}_{q}(\m)}=\int_{\R^{d}}\exp(a|\xi|)\langle \xi \rangle^{q}\exp\left(-\int_{0}^{t}\left[d\mathbf{B}_{\alpha} +\Sigma_{\M}\left(\xi\,e^{(\tau-t)\mathbf{B}_{\alpha}}\right)\right]\d \tau\right)f(\xi\,e^{-t\mathbf{B}_{\alpha}})\d \xi$$
and, setting $y=\xi e^{-t\mathbf{B}_{\alpha}}$, we get 
$$\left\|U_{\alpha}(t)f\right\|_{L^{1}_{q}(\m)}\leq e^{qt\mathbf{B}_{\alpha}}\int_{\R^{d}}\exp\left(a\,e^{t\mathbf{B}_{\alpha}}|y|\right)\langle y\rangle^{q}\exp\left(-\int_{0}^{t} \Sigma_{\M}\left(y\,e^{\tau\mathbf{B}_{\alpha}}\right)\d \tau\right)f(y)\d y.$$

Now, since $\Sigma_{\M}(\xi) \geq \sigma_{0}\langle \xi \rangle \geq \frac{\sigma_{0}}{\sqrt{2}}\left(1+|\xi|\right)$ for all $\xi \in \R^{d}$ we get under the assumption that $\frac{\sigma_{0}}{\sqrt{2}} \geq a\mathbf{B}_{\alpha} \geq 0$:
\begin{equation*}\begin{split}
\left\|U_{\alpha}(t)f\right\|_{L^{1}_{q}(\m)}&\leq \exp\left(-\frac{\sigma_{0}}{\sqrt{2}}t +qt\mathbf{B}_{\alpha}\right)
\int_{\R^{d}}\langle y \rangle^{q}\exp\left(a\,e^{t\mathbf{B}_{\alpha}}|y|\right)\exp\left(-a\mathbf{B}_{\alpha}|y|\int_{0}^{t} e^{\tau\mathbf{B}_{\alpha}}\d \tau\right)f(y)\d y\\
&\leq \exp\left(-\left(\frac{\sigma_{0}}{\sqrt{2}}-q\mathbf{B}_{\alpha}\right)t\right)\|f\|_{L^{1}_{q}(\m)}.\end{split}\end{equation*}
This proves the claim. It is not difficult then to prove that $\{U_{\alpha}(t)\;,\;t \geq 0\}$ is indeed a $C_{0}$-semigroup in $L^{1}_{q}(\m)$.
\end{proof}

\begin{lem}\phantomsection\phantomsection\label{lem:q-q+1}
Let $\alpha > 0$ be such that  {$\sigma_{0} \geq \sqrt{2} a \mathbf{B}_{\alpha}\geq 0.$} For any $q \geq 0$, one has
$$\|{\mathcal R}(\lambda,\mathbf{T}_{\alpha})\|_{\mathscr{B}(L^{1}_{q}(\m),L^{1}_{1+q}(\m))} \leq \frac{1}{\sigma_{0}-a\mathbf{B}_{\alpha}}, \qquad \forall \lambda >q\mathbf{B}_{\alpha}.$$
\end{lem}\phantomsection
\begin{proof} Given $f \in L^{1}_{q}(\m)$ and $\lambda >0$ large enough, we need to compute $\|{\mathcal R}(\lambda,\mathbf{T}_{\alpha})f\|_{L^{1}_{1+q}(\m)}$. First of all, since $\{U_{\alpha}(t)\,;\,t \geq 0\}$ is a nonnegative semigroup, ${\mathcal R}(\lambda,\mathbf{T}_{\alpha})$ is nonnegative and, since the positive cone of $L^{1}_{q}(\m)$ is generating, it is enough to consider $f$ nonnegative. Set then $g={\mathcal R}(\lambda,\mathbf{T}_{\alpha})f$. One has
$$(\lambda+\Sigma_{\M}(\xi))g(\xi)  + \mathbf{B}_{\alpha}\mathrm{div}(\xi\,g(\xi))=f(\xi), \qquad \xi \in \R^{d}.$$
Multiplying by $\langle \xi\rangle^{q}\m(\xi)$ and integrating over $\R^{d}$ we get
\begin{equation*}
\begin{split}
\int_{\R^{d}}(\lambda+\Sigma_{\M}(\xi))g(\xi) \langle \xi\rangle^{q}\m(\xi)\d\xi&=\|f\|_{L^{1}_{q}(\m)}-\mathbf{B}_{\alpha}\int_{\R^{d}}\mathrm{div}(\xi\,g(\xi))\langle \xi\rangle^{q}\m(\xi)\d\xi\\
&=\|f\|_{L^{1}_{q}(\m)}+\mathbf{B}_{\alpha}\int_{\R^{d}}g(\xi)\xi \cdot \nabla \left(\langle \xi\rangle^{q}\m(\xi)\right)\d\xi
\end{split}\end{equation*}
Since $\xi \cdot \nabla\left(
 \langle\xi\rangle^q \, \m(\xi)\right)=q\,|\xi|^{2}\langle \xi\rangle^{q-2}\m(\xi)+a\langle \xi\rangle^{q}|\xi|\m(\xi)$, we get
$$ \int_{\R^{d}}(\lambda+\Sigma_{\M}(\xi))g(\xi) \langle \xi\rangle^{q}\m(\xi)\d\xi \leq \|f\|_{L^{1}_{q}(\m)} + q\mathbf{B}_{\alpha}\|g\|_{L^{1}_{q}(\m)} + a\mathbf{B}_{\alpha}\|g\|_{L^{1}_{q+1}(\m)}.$$
Since $\Sigma_{\M}(\xi) \geq \sigma_{0}\langle \xi\rangle$, we get the estimate
$$\lambda \|g\|_{L^{1}_{q}(\m)} + \sigma_{0}\|g\|_{L^{1}_{q+1}(\m)}  \leq \|f\|_{L^{1}_{q}(\m)} + q\mathbf{B}_{\alpha}\|g\|_{L^{1}_{q}(\m)} + a\mathbf{B}_{\alpha}\|g\|_{L^{1}_{q+1}(\m)}.$$
Therefore, fixing $\alpha >0$ such that $\sigma_{0}\geq \sqrt{2} a\mathbf{B}_{\alpha} > a\mathbf{B}_{\alpha}$ and taking then $\lambda > q\mathbf{B}_{\alpha}$, we get
$$\|g\|_{L^{1}_{q+1}(\m)} \leq \frac{1}{\sigma_{0}-a\mathbf{B}_{\alpha}}\|f\|_{L^{1}_{q}(\m)}$$
which gives the desired estimate.
\end{proof}
\begin{nb}\phantomsection\phantomsection The above estimate directly yields $\|g\|_{L^{1}_{q}(\m)} \leq \frac{1}{\lambda-q\mathbf{B}_{\alpha}}\|f\|_{L^{1}_{q}(\m)}$, i.e.
\begin{equation}\label{eq:largelam}
\|{\mathcal R}(\lambda,\mathbf{T}_{\alpha})\|_{\mathscr{B}(L^{1}_{q}(\m) )} \leq \frac{1}{\lambda -q\mathbf{B}_{\alpha}}.\end{equation}
\end{nb}\phantomsection

Recall the definition of $\mathscr{L}_{0}^{R,\delta}:$
$$\mathscr{L}_{0}^{{R,\delta}}h(\xi)=\int_{\R^d\times \S^{d-1}} (1-\Theta_\delta )
[\M(\xi'_*)h(\xi') +\M(\xi')h(\xi'_*)-\M(\xi)h(\xi_*)] 
\:|\xi-\xi_*|\, \d\xi_*\d\sigma.$$
One can split $\mathscr{L}^{R,\delta}_{0}$ into positive and negative parts,
$$\mathscr{L}^{R,\delta}_{0}=\mathscr{L}^{R,\delta}_{0,+}-\mathscr{L}_{0,-}^{R,\delta}$$
where
$$\mathscr{L}_{0,-}^{R,\delta}h(\xi)=\int_{\R^d\times \S^{d-1}} (1-\Theta_\delta )\M(\xi)h(\xi_*)|\xi-\xi_*|\, \d\xi_*\d\sigma=\M(\xi)\int_{\R^{d}}h(\xi_{*}){\nu}_{\delta}(\xi,\xi_{*})\d \xi_{*}$$
with
$${\nu}_{\delta}(\xi,\xi_{*})=|\xi-\xi_{*}|\int_{\S^{d-1}}(1-\Theta_{\delta}(\xi,\xi_{*},\sigma))\d\sigma.$$
One has then the following whose proof is the same as that of \cite[Lemma B.1 \& Proposition B.2]{CaLo}
\begin{lem}\phantomsection\phantomsection
For any $q \geq 0$, there exists $\kappa_{q}(\delta) >0$ such that $\lim_{\delta \to 0}\kappa_{q}(\delta)=0$ and 
$$\|\mathscr{L}^{R,\delta}_{0,+}h\|_{L^{1}_{q}(\m)} \leq \kappa_{q}(\delta)\|h\|_{L^{1}_{q+1}(\m)}, \qquad \forall h \in L^{1}_{1+q}(\m),$$while 
$$\mathscr{L}^{R,\delta}_{0,-}\::\:\:L^{1}_{q}(\m) \to L^{1}_{q}(\m)$$
is bounded. 
\end{lem}\phantomsection

Introduce the sum $\mathcal{Z}_{\alpha}^{0}:=\mathscr{L}^{R,\delta}_{0,+} + \mathbf{T}_{\alpha}$
with domain $\D(\mathcal{Z}_{\alpha}^{0})=\D(\mathbf{T}_{\alpha})$. Combining both the above Lemmas, one gets that, as soon as {$\sigma_{0} \geq \sqrt{2}a\mathbf{B}_{\alpha} \geq 0$}, it holds
$$\|\mathscr{L}^{R,\delta}_{0,+}{\mathcal R}(\lambda,\mathbf{T}_{\alpha})\|_{\mathscr{B}(L^{1}_{q}(\m))} \leq \dfrac{\kappa_{q}(\delta)}{\sigma_{0}-a\mathbf{B}_{\alpha}}, \qquad \forall \lambda > q\mathbf{B}_{\alpha}.$$ As a consequence, picking $\delta >0$ small enough so that
$$\dfrac{\kappa_{q}(\delta)}{\sigma_{0}-a\mathbf{B}_{\alpha}} < 1$$
one deduces  that  $(q\mathbf{B}_{\alpha}; +\infty)\subset \varrho(\mathcal{Z}_{\alpha}^{0})$,  and
$${\mathcal R}(\lambda,\mathcal{Z}_{\alpha}^{0})={\mathcal R}(\lambda,\mathbf{T}_{\alpha})\sum_{j=0}^{\infty}\left[\mathscr{L}^{R,\delta}_{0,+}{\mathcal R}(\lambda,\mathbf{T}_{\alpha})\right]^{j}, \qquad \forall \lambda  > q\mathbf{B}_{\alpha}.$$ 
In particular, one checks from the above series representation that
\begin{equation}\begin{split}\label{eq:normZ0}
\|{\mathcal R}(\lambda,\mathcal{Z}_{\alpha}^{0})\|_{\mathscr{B}(L^{1}_{q}(\m),L^{1}_{1+q}(\m))} &\leq \|{\mathcal R}(\lambda,\mathbf{T}_{\alpha})\|_{\mathscr{B}(L^{1}_{q}(\m),L^{1}_{1+q}(\m))}\sum_{j=0}^{\infty}\left\|\left[\mathscr{L}^{R,\delta}_{0,+}{\mathcal R}(\lambda,\mathbf{T}_{\alpha})\right]^{j}\right\|_{\mathscr{B}(L^{1}_{q}(\m))}\\
&\leq \dfrac{1}{\sigma_{0}-a\mathbf{B}_{\alpha}}\sum_{j=0}^{\infty}\left(\frac{\kappa_{q}(\delta)}{\sigma_{0}-a\mathbf{B}_{\alpha}}\right)^{j}=\dfrac{1}{\sigma_{0}-a\mathbf{B}_{\alpha}-\kappa_{q}(\delta)}, \qquad \forall \lambda > q\mathbf{B}_{\alpha}.\end{split}\end{equation}
Notice also that, according to \eqref{eq:largelam}, 
\begin{equation}\label{eq:largelam1}
\lim_{\lambda \to \infty}\|{\mathcal R}(\lambda,\mathcal{Z}_{\alpha}^{0})\|_{\mathscr{B}(L^{1}_{q}(\m))}=0.\end{equation}
Introduce then $\mathcal{Z}_{\alpha}=\mathcal{Z}_{\alpha}^{0} - \P_{\alpha}^{0}$
with $\D(\mathcal{Z}_{\alpha})=\D(\mathcal{Z}_{\alpha}^{0})$. Picking $\lambda > q\mathbf{B}_{\alpha}$ one has
$$(\lambda-\mathcal{Z}_{\alpha})=(\lambda-\mathcal{Z}_{\alpha}^{0})+\P_{\alpha}^{0}$$
and,  multiplying from the right by the resolvent ${\mathcal R}(\lambda,\mathcal{Z}_{\alpha}^{0})$, one has
$$(\lambda-\mathcal{Z}_{\alpha}){\mathcal R}(\lambda,\mathcal{Z}_{\alpha}^{0})=I + \P_{\alpha}^{0}{\mathcal R}(\lambda,\mathcal{Z}_{\alpha}^{0}).$$
Notice that all the operators here are well defined since the range of ${\mathcal R}(\lambda,\mathcal{Z}_{\alpha}^{0})$ is $\D(\mathcal{Z}_{\alpha}^{0})=\D(\mathcal{Z}_{\alpha})$ (which makes first the operator on the left-hand-side well defined) and, as such,  is included in $L^{1}_{q+1}(\m)$ which is the domain of $\P_{\alpha}^{0}.$ Moreover, from \eqref{eq:normZ0} and Proposition \ref{prop:converLLL0}
$$\|\P_{\alpha}^{0}{\mathcal R}(\lambda,\mathcal{Z}_{\alpha}^{0})\|_{\mathscr{B}(L^{1}_{q}(\m))} \leq \ep_{0,q}(\alpha)\|{\mathcal R}(\lambda,\mathcal{Z}_{\alpha}^{0})\|_{\mathscr{B}(L^{1}_{q}(\m), L^{1}_{q+1}(\m))} \leq \frac{\ep_{0,q}(\alpha)}{\sigma_{0}-a\mathbf{B}_{\alpha}-\kappa_{q}(\delta)}.$$
One can therefore find $\alpha$ small enough so that
$$\|\P_{\alpha}^{0}{\mathcal R}(\lambda,\mathcal{Z}_{\alpha}^{0})\|_{\mathscr{B}(L^{1}_{q}(\m))} < 1, \qquad \forall \lambda > q\mathbf{B}_{\alpha}$$ and, as such, $I+\P_{\alpha}^{0}{\mathcal R}(\lambda,\mathcal{Z}_{\alpha})$ becomes invertible and so is $(\lambda-\mathcal{Z}_{\alpha}){\mathcal R}(\lambda,\mathcal{Z}_{\alpha}^{0})$. This proves that, for $\lambda >q\mathbf{B}_{\alpha}$, $(\lambda-\mathcal{Z}_{\alpha})$ is invertible with
$${\mathcal R}(\lambda,\mathcal{Z}_{\alpha})={\mathcal R}(\lambda,\mathcal{Z}_{\alpha}^{0})\sum_{j=0}^{\infty}\left[-\P_{\alpha}^{0}{\mathcal R}(\lambda,\mathcal{Z}_{\alpha}^{0})\right]^{j}.$$
In particular, according to \eqref{eq:largelam1}, $\lim_{\lambda \to \infty}\|{\mathcal R}(\lambda,\mathcal{Z}_{\alpha})\|_{\mathscr{B}(L^{1}_{q}(\m))}=0.$ Finally, since 
$$\mathcal{B}_{\alpha,\delta}=\mathcal{Z}_{\alpha}-\mathscr{L}_{0,-}^{R,\delta}$$
with $\mathscr{L}_{0,-}^{R,\delta}$ bounded, one deduces easily that $\lambda-\mathcal{B}_{\alpha,\delta}$ is invertible provided $\lambda$ is large enough so that 
$$\|\mathscr{L}_{0,-}^{R,\delta}\|_{\mathscr{B}(L^{1}_{q}(\m))}\|{\mathcal R}(\lambda,\mathcal{Z}_{\alpha})\|_{\mathscr{B}(L^{1}_{q}(\m))} < 1.$$ This, together with the hypo-dissipativity ensures that $\mathcal{B}_{\alpha,\delta}$ generates a $C_{0}$-semigroup.

\end{document}